\documentclass[a4paper]{article}

\usepackage{amsmath,amssymb,mathtools,amsthm}
\usepackage{graphicx,float,enumitem}
\usepackage{bm}
\graphicspath{{figures/}}

\usepackage[default,scale=1]{opensans}
\usepackage[T1]{fontenc}

\usepackage[utf8]{inputenc}
\DeclareUnicodeCharacter{00A0}{ }
\usepackage{geometry}
\newcommand{\ee}{\varepsilon}

\newcommand{\defeq}{{\coloneqq}}

\DeclareMathOperator{\N}{N}

\newcommand{\Rd}{{\mathbb R^n}}
\newcommand{\Br}{{B_R}}

\newcommand{\diver}{\nabla \cdot}

\DeclareMathOperator{\sign}{sign}
\DeclareMathOperator{\supp}{supp}

\DeclareMathOperator*{\esssup}{ess\,sup}
\DeclareMathOperator*{\essinf}{ess\,inf}
\DeclareMathOperator*{\essosc}{ess\,osc}

\newcommand*\diff{\mathop{}\!\mathrm{d}}

\usepackage[hidelinks]{hyperref}
\usepackage[nameinlink]{cleveref}

\newtheorem{theorem}{Theorem}[section]
\newtheorem{proposition}[theorem]{Proposition}%
\newtheorem{corollary}[theorem]{Corollary}%
\newtheorem{lemma}[theorem]{Lemma}%

\theoremstyle{definition}
\newtheorem{definition}[theorem]{Definition}%
\newtheorem{remark}[theorem]{Remark}%

\numberwithin{equation}{section}

\usepackage[
url=false,
isbn=false,
backend=bibtex,
maxcitenames=4,
giveninits,
maxbibnames=100,
doi=true,
uniquename=false,
block=none]{biblatex}
\renewbibmacro{in:}{}
\DeclareFieldFormat[article]{citetitle}{#1}
\DeclareFieldFormat[article]{title}{#1}

\makeatletter
\renewcommand*{\@fnsymbol}[1]{\ensuremath{\ifcase#1\or \star \or \dagger\or \ddagger\or
		\mathsection\or \mathparagraph\or \|\or **\or \dagger\dagger
		\or \ddagger\ddagger \else\@ctrerr\fi}}
\makeatother

\usepackage{todonotes}

\begin{document}
\title{Infinite-time concentration in Aggregation--Diffusion equations  with a given potential}

\author{J.A. Carrillo%
\thanks{Mathematical Institute, University of Oxford, Oxford OX2 6GG, UK. Email: \href{mailto:carrillo@maths.ox.ac.uk}{carrillo@maths.ox.ac.uk}} %
\and D. Gómez-Castro%
\thanks{Mathematical Institute, University of Oxford, Oxford OX2 6GG, UK. Email: \href{mailto:gomezcastro@maths.ox.ac.uk}{gomezcastro@maths.ox.ac.uk}}
\and
J.L. Vázquez%
\thanks{Departamento de Matemáticas, Universidad Autónoma de Madrid. Email: \href{mailto:juanluis.vazquez@uam.es}{juanluis.vazquez@uam.es}}
}
\maketitle

\begin{abstract}
	\textbf{Abstract in English.}
Typically, aggregation-diffusion is modeled by parabolic equations that combine linear or nonlinear diffusion with a Fokker-Planck convection term. Under very general suitable assumptions,
we prove
that radial solutions of the evolution process converge asymptotically in time towards a stationary state representing the balance between the two effects.
Our parabolic system is the gradient flow of an energy functional, and in fact we show that the stationary states are minimizers of a relaxed energy.
Here, we study radial solutions of an aggregation-diffusion model that combines nonlinear fast diffusion with a convection term driven by the gradient of a potential, both in balls and the whole space.
We show that, depending on the exponent of fast diffusion and the potential, the steady state is given by the sum of an explicit integrable function, plus a Dirac delta at the origin containing the rest of the mass of the initial datum. 
This splitting phenomenon is an uncommon example of blow-up in infinite time.
\end{abstract}

\noindent {\bf 2020 Mathematics Subject Classification.}
  	35K55,  	%
   	35K65,  %
    35B40,  	%
    35D40,  %
  	35Q84	%

\noindent {\bf Keywords: } {Nonlinear parabolic equations, nonlinear diffusion, aggregation, Dirac delta formation,  blow-up in infinite time, viscosity solutions}.

\section{Introduction}

Enormous work has been devoted over the last years to the study of mathematical models for Aggregation-Diffusion that are formulated in terms of semilinear parabolic equations  combining linear or non-linear diffusion with a Fokker-Planck convection term coming either from a given potential or from an interaction potential, see \cite{AMTU,DP,MV,CJMTU,CMV03,Fornaro2012,CCY19} and the references therein, and the books \cite{AGS,Villani03}. In this paper  we consider the aggregation-diffusion equation
\begin{equation}
\label{eq:main}
	\tag{P}
		\frac{\partial \rho}{\partial t} = \Delta \rho^m +  \diver ( \rho \nabla V ), \qquad \text{in } (0,\infty) \times \Rd
\end{equation}
where the potential $V(x)$ is given
and $0<m<1$, the fast-diffusion range \cite{Vazquez2007}. We take as initial data a probability measure, i.\,e.,
\begin{equation}
	\label{eq:rho probability}
	\rho_0\ge 0 , \qquad \int_{\Rd } \rho_0 \diff x= 1.
\end{equation}
We will find conditions on the radial initial data $\rho_0$ and the radial potential $V$ so that
\begin{enumerate}[label=\roman*)]
\item  we provide a suitable notion of solution of the Cauchy problem defined globally in time passing through the mass (or distribution) function,

\item as $t\to\infty$, the solution undergoes one-point blow-up of the split form
\begin{equation*}
	\rho(t)\to \mu_\infty =  \rho_\infty + (1 - \| \rho_\infty \|_{L^1 (\Rd)}) \delta_0  ,
\end{equation*}
where $\rho_\infty(x)>0$ is an explicit stationary solution of \eqref{eq:main}.
The presence of the concentrated point measure is a striking fact that needs detailed understanding and is the main motivation of this work.
Here and after we identify an $L^1$ function with the absolutely continuous measure it generates.
\end{enumerate}
It is known that Dirac measures are invariant by the semigroup generated by the fast-diffusion equation $u_t = \Delta u^m$ for $0 < m < \frac{n-2}n$ (see \cite{Brezis1983}), but they are never produced from $L^1$ initial data. Here we show that the aggregation caused by the potential term might be strong enough to overcome the fast-diffusion term and produce a Dirac-delta concentration at $0$ as $t \to \infty$, in other words, infinite-time concentration. 

The case of \eqref{eq:main} with slow diffusion $m > 1$ was studied in \cite{CJMTU,Kim2010}, where the authors show that the steady state does not contain a Dirac delta (i.e. $\| \rho_\infty \|_{L^1} = 1$). The linear diffusion case was extensively studied in \cite{AMTU,MV,OV}. The fast diffusion range  $1>m > \frac{n-2}n$ with quadratic confinement potential is also well-known and its long-time asymptotics, even for Dirac initial data, is given by integrable stationary solutions, see for instance \cite{BBDGV} and its references. See also \cite{VazWin2011} for the evolution of point singularities in bounded domains.

We will take advantage of the formal interpretation of \eqref{eq:main} as the $2$-Wasserstein flow \cite{CJMTU,CMV03,AGS} associated to the free-energy
\begin{equation}\label{eqn.freeen}
	\mathcal F [\rho] = \tfrac{1}{m-1} \int_\Rd \rho(x)^{m} \diff x + \int_\Rd V(x) \rho(x) \diff x,	
\end{equation}
in order to obtain properties of this functional in terms of the Calculus of Variations. We also take advantage of this structure to obtain a priori estimates on the solution $\rho$ of \eqref{eq:main} due to the dissipation of the energy.

\paragraph {Main assumptions and discussion of the main results.} We introduce the specific context in which point-mass concentration arises.
We first examine the special stationary  solutions that play a role in the asymptotics:
\begin{equation}\label{Vh}
	\rho_{V+h} (x) = (  \tfrac{1-m}{m} (V (x) + h)  )^{- \frac{1}{1-m}} \quad  \text{for } \ x \in \Rd,
\end{equation}
for $h\ge 0$. It is easy to check that they are solutions of \eqref{eq:main}, and they are bounded if $h>0$. We now consider the class of suitable potentials. We first assume that $V$ has a minimum at $x=0$ and is smooth:
$V \in W^{2,\infty}_{loc} (\Rd)$, 	$V \ge 0$, $V(0) = 0$.
We are interested in radial aggregating potentials, in fact we
	$V$
	{is radially symmetric and non-decreasing}.
An essential assumption in the proof of formation of a point-mass concentration is the following
small-mass condition for the admissible steady states:
\begin{equation}
	a_V = \int_{\mathbb R^n} \rho_V(x) \diff x< 1.
\end{equation}
As a simplifying assumption we will assume that
\begin{equation}
	\label{eq:rhoV in L1ee loc}
	\int_{B_1} \rho_V^{1+\ee}(x)\diff x < +\infty, \qquad \text{for some } \ee > 0 .	
\end{equation}

The bounded case  in which $\rho_{V+h_1} \le \rho_0  \le  \rho_{V+h_2}$ with $h_1, h_2 > 0$ was studied in \cite{Cao2020} and leads to no concentration.
On the contrary, we will show that there exists a class of radial initial data $\rho_0(x)\ge \rho_V(x)$ such that the corresponding solution converges as $t\to\infty$ to the split measure
\begin{equation}\label{eqn.split}
	\mu_\infty = (1- a_V) \delta_0 + \rho_V(x),
\end{equation}
in the sense of mass (which will be made precise below). Moreover, under further assumptions on $V$, we show that $\mu_\infty$ is the global minimizer in the space of measures of the relaxation of $\mathcal F$.

An important motivation for our paper is the current interest in the following model of aggregation diffusion with interaction potential
\begin{equation}
\label{eq:ADE general}
		\frac{\partial \rho}{\partial t} = \Delta \rho^m +  \diver ( \rho \nabla W*\rho )
\end{equation}
that has led to the discovery of some highly interesting features that have consequences for the parabolic theory and the the Calculus of Variations. Recent results \cite{Carrillo2019} show that, under some conditions on $W$ the
energy minimizer of the corresponding energy functional is likewise split as
\begin{equation*}
	\mu_\infty = (1 - \|\rho_\infty\|_{L^1 (\Rd)}) \delta_0 + \rho_\infty.
\end{equation*}
The presence of the concentrated point measure is known for specific choices of $W$, see \cite{carrillo2020fast}. To the best of our knowledge, there exist no results in the literature showing that solutions of the parabolic problem actually converge to these minimisers with a Dirac delta.
In this paper, we treat as a first step the where $V$ is known. We expect similar results to hold when $V$ is replaced by $W * \rho$, but the techniques will be more difficult. For example, the non-local nature of \eqref{eq:ADE general} suggests that there might not be a comparison principle.
\normalcolor

 It was shown in \cite{Benilan1981} that for very fast diffusion, $m < \frac{n-2} n$, then the solutions of the Fast Diffusion Equation $u_t = \Delta u^m$ with $u_0 \in L^1 (\mathbb R^n) \cap L^\infty (\mathbb R^n)$ vanish in finite time, i.e. $u(t, x) = 0$ for $t \ge T^*$.
When $a_V < 1$, we construct explicit initial data that preserve the total mass, and this holds for any $m \in (0,1)$.

\subparagraph{The case of a ball of radius $R$} We have a more complete overall picture when we focus on the problem posed in a ball $\Br$, adding a no-flux condition on the boundary:
\begin{equation}
	\label{eq:main bounded domain}
	\tag{P$_R$}
	\begin{dcases}
		\frac{\partial \rho}{\partial t} = \Delta \rho^m +  \diver ( \rho \nabla V_R ) & \text{in }(0,\infty) \times \Br, \\
	(\nabla \rho^m + \rho \nabla V_R) \cdot x = 0, &  \text{on }  (0,\infty) \times \partial \Br, \\
	\rho (0,x) = \rho_0 (x).
	\end{dcases}
\end{equation}
As a convenient assumption,  we require that $V_R$ does not produce flux across the boundary
\begin{equation}
	\label{eq:V no aggregation on boundary}
	\nabla V_R (x) \cdot x = 0 , \qquad \text{on } \partial \Br.
\end{equation}
We discuss this assumption on \Cref{sec:condition VR no flux}. This problem is the $2$-Wasserstein flow of the free energy
\begin{equation}
	\mathcal F_R [\rho] = \tfrac{1}{m-1} \int_\Br \rho(x)^{m} \diff x + \int_\Br V_R(x) \rho(x) \diff x	.
\end{equation}
For \eqref{eq:main bounded domain}, we show that $\mathcal F_R$ is bounded below and sequences of non-negative functions of fixed $\| \rho \|_{L^1 (\Br)} = \mathfrak m$ converge weakly in the sense of measures to
\begin{equation*}
	\mu_{\infty,\mathfrak m,R} =
	\begin{dcases}
		\rho_{V_R + h} & \text{ if there exists } h \ge 0 \text{ such that }\| \rho_{V_R + h} \|_{L^1 (\Br)}  = \mathfrak m , \\
		\rho_{V_R} + (\mathfrak{m} - \| \rho_{V_R + h} \|_{L^1 (\Br)} ) \delta_0  & \text{ if } \| \rho_{V_R} \|_{L^1 (\Br)} < \mathfrak m.
	\end{dcases}
\end{equation*}
This means that, if the mass $a_{0,R}$ cannot be reached in the class $\rho_{V_R+h}$, the remaining mass is complete with a Dirac delta at $0$. Notice that the mass of $\rho_{V_R+h}$ is decreasing with $h$, so the largest mass is that of $\rho_{V_R}$.

We construct an $L^1$-contraction semigroup of solutions $S_R$ of \eqref{eq:main bounded domain} such that, if $ \rho_{V_R} \le \rho_0 \in L^1(\Br)$ and radially symmetric, then
\begin{equation*}
	\mathcal F_R[S_R (t) \rho_0] \searrow \widetilde {\mathcal F}_R[\mu_{\infty,\mathfrak m,R}] = \mathcal F_R[\rho_{V_R}],
\end{equation*}
where $\mathfrak m = \| \rho_0 \|_{L^1 (\Br)}$ and $\widetilde {\mathcal F}_R$ is the relaxation of $\mathcal F_R$ to the space of measures presented below
(see \cite{Demengel1986})
.
The semigroup $S_R$ is constructed as the limit of the semigroup of the regularised problems written below as \eqref{eq:main regularised bounded}. Then, we recover our results by passing to the limit in $\Phi$ and $R$.

\paragraph {The mass function.}
One of the main tools in this paper will be the study of the so-called mass variable, which can be applied under the assumption of radial solutions. It works as follows. First, we introduce the spatial volume variable $v = |x|^n |B_1|$ and consider the mass function
\begin{equation}
	M_u(t,v) = \int_{ \widetilde B_v } \rho(t,x) \diff x, \qquad \widetilde B_v = \left( \tfrac{v}{|B_1|} \right)^{\frac 1 n} B_1
\end{equation}
Notice that $|\widetilde B_v| = v$.
For convenience we define
$	R_v = R^n |B_1|.
$
We will prove that $M$ satisfies the following nonlinear diffusion-convection equation in the viscosity sense
\begin{align}
\tag{M}
\label{eq:mass}
	\frac{\partial M}{\partial t} &= (n \omega_n^{\frac 1 n } v ^{\frac{n-1} n})^2   \left\{ \frac{\partial }{\partial v} \left[  \left(  \frac{\partial M }{\partial v}  \right)^m \right] + \frac{\partial M }{\partial v} \frac{\partial V}{\partial v}  \right\},
\end{align}
where $\omega_n = |B_1|$.
The diffusion term of this equation is of $p$-Laplacian type, where $p = m+1$. The weight will not be problematic when $ v > 0$, as we show in \Cref{sec:classical regularity} using the parabolic theory in DiBenedetto's book \cite{DiBenedetto1993}.

Notice that the formation of a Dirac delta at $0$ is equivalent to the loss of the Dirichlet boundary condition $M(t,0) = 0$.
Few results of loss of the Dirichlet boundary condition are known in the literature of parabolic equation.
For equations of the type $u_t=u_{xx}+|u_x|^p$, it is known (see, e.g., \cite{Arrieta2004}) that $u_x$ may blow up on the boundary in finite or infinite time, depending on the choice of boundary conditions. The case of infinite time blow-up  was revisited in \cite{SoupletVazq}.
The question of boundary discontinuity in finite time, loss of boundary condition, for the so-called  viscous Hamilton-Jacobi equations is studied in \cite{BarlesDaLio2004,PS17,PS20,mizoguchi2021singularity} and does not bear a direct relation with our results. A general reference for boundary blow-up can be found in the book \cite{QuittSoupBook}.

\paragraph {Precise statement of results.} In order to approximate the problem in $\Rd$, our choice of $V_R$ will be of the form
\begin{equation*}
	V_R (x)  \begin{dcases}
		= V(x) & |x| \le R- \ee, \\
		\le V(x) & R - \ee < |x| \le R
	\end{dcases}
\end{equation*}
and with the condition $V_R \cdot x = 0$ on $\partial B_R$.
We also define
\begin{equation*}
	a_{V,R} = \int_{\Br} \rho_{V_R} \diff x \qquad \text{ and } \qquad 	a_{0,R} = \int_{\Br} \rho_0 \diff x.
\end{equation*}
We will denote $V = V_R$ until \Cref{sec:Rn}.
\begin{theorem}[Infinite-time concentration of solutions of \eqref{eq:main bounded domain}]
	\label{thm:bounded concentrating intro}
	Assume $V \in W^{2,\infty} (\Br)$ is radially symmetric, strictly increasing, $V\ge 0$, $V (0) = 0$, $V \cdot x = 0$ on $\partial \Br$ and the technical assumption \eqref{eq:rhoV in L1ee loc}.
	\normalcolor
	Assume also that $a_{0,R} > a_{V,R}$, $\rho_0$ radially symmetric, $\rho_0 \ge \rho_V$ and $\rho_0 \in L^\infty (\Br \setminus B_{r_1})$ for some $r_1 < R$. Then, the solution $\rho$ of \eqref{eq:main bounded domain} constructed in \Cref{thm:existence bounded L1} satisfies
	\begin{equation*}
		\liminf_{t \to \infty} \int_{B_r} \rho(t,x) \diff x \ge (a_{0,R} - a_{V,R}) + \int_{B_r} \rho_V (x) \diff x , \qquad \forall r \in [0,R].
	\end{equation*}
	(i.e., there is concentration in infinite time).
	Moreover, if
	\begin{equation}
	\label{eq:rho0 below limit}	
	\int_{B_r} \rho_0 (x) \diff x \le (a_{0,R} - a_{V,R}) + \int_{B_r} \rho_V (x) \diff x \qquad \forall  r \in [0,R],
	\end{equation}
	 then for $\mu_{\infty,R} = (a_{0,R} - a_{V,R}) \delta_0  + \rho_V$ we have that
	\begin{equation*}
		\lim_{t \to \infty} d_{1} ( \rho (t) , \mu_{\infty,R} ) = 0  ,
	\end{equation*}
	where $d_1$ denotes the $1$-Wasserstein distance.
\end{theorem}

\begin{remark}
	\label{rem:bounded concentration non-radial}
	If we take a non-radial datum $\rho_0 \ge \rho_{0,r}$ with $\rho_{0,r}$ radially symmetric satisfying the hypothesis of \Cref{thm:bounded concentrating intro}, then the corresponding solution $\rho(t,x) $ of \eqref{eq:main bounded domain} constructed in \Cref{thm:existence bounded L1} concentrates in infinite time as well, due to the comparison principle.
\end{remark}

Through approximation as $R \to \infty$, we will also show that
\begin{corollary}[At least infinite-time concentration of solutions of \eqref{eq:main}]
\label{cor:Rd concentration}
	Under the hypothesis of \Cref{thm:bounded concentrating intro} and
	suitable hypothesis on the initial data (specified in \Cref{sec:Rd concentration}), we can show the existence of viscosity solutions of \eqref{eq:mass} in $(0,\infty) \times (0,\infty)$ (obtained as a limit of the problems in $\Br$), such that
	\begin{equation*}
		\lim_{t \to \infty} M(t,v) = (1-a_V) + M_{\rho_V} (v)
	\end{equation*}
	for all $v > 0$ and, furthermore, locally uniformly $(0,\infty)$. We also have that
	\begin{equation*}
			\lim_{t \to \infty} d_{1} ( \rho(t) , (1-a_V)\delta_0 + \rho_V  ) = 0.
	\end{equation*}
\end{corollary}

Through our construction of $M$, we cannot guarantee in general that $M(t,0) = 0$ for $t$ finite. Producing a priori estimates, we can ensure this in some cases.

\begin{theorem}[Infinite-time concentration for $V$ quadratic at $0$]
	\label{prop:Rd no concentration in finite time}
	Let $\rho_0 \in L^1_+ (\mathbb R^n)$ non-increasing and assume
	\begin{equation}
		\label{eq:hypothesis superquadratic}
		\frac{\partial V}{\partial r} (s)   \le C_{V} r , \qquad \text{  in  } B_{R_V} \text{ for some } C_V  > 0.
	\end{equation}
	Then, the viscosity mass solution constructed in \Cref{prop:Rd existence mass} does not concentrate in finite time, i.e. $M(t,0) = 0$.
\end{theorem}

\paragraph{The picture for power-like $V$.} Let us discuss the case where $V$ is of the form
	\begin{equation*}
		V(x) \sim  \begin{dcases}
			|x|^{\lambda_0} & |x| \ll 1, \\
			|x|^{\lambda_\infty} & |x| \gg 1. \\
		\end{dcases}
	\end{equation*}
	The condition $V \in W^{2,\infty}_{loc} (\Rd)$ means $\lambda_0 \ge 2$.
	In this setting, we satisfy \eqref{eq:hypothesis superquadratic} so concentration does not happen in finite time.
	The condition $\rho_V \in L^1(\Rd)$ (i.e. $a_V < \infty$) holds if and only if
	\begin{equation}\label{Vgrowth}
		\frac{n-\lambda_\infty}{n} < m < \frac{n-\lambda_0}{n}.
	\end{equation}
	In fact, under this condition, $\rho_V \in L^{1+\ee} (\Rd)$.
	In addition to the behaviour at $0$ and $\infty$, the \ restriction $a_V < 1$ is a condition on the intermediate profile of $V$.
	This is sufficient to construct initial data $\rho_0$ (of the shape $\rho_D$ present below) so that solutions converge to $\mu_\infty$ as $t \to \infty$, in the sense of mass.
	Due to \eqref{eq:hypothesis superquadratic}, the concentration is precisely at infinite time.
	But we do not know that $\mu_\infty$ is a global minimiser of $\mathcal F$.
	In \Cref{rem:minimisation for powers at infinity} we prove that the energy functional is bounded below whenever $m > \frac{n}{n+\lambda_\infty}$. Notice that
$
	\frac{n-\lambda_\infty}{n} < \frac{n}{n+\lambda_\infty}.
$
Therefore,
if \begin{equation*}
		\frac{n}{n+\lambda_\infty} < m < \frac{n-\lambda_0}{n}, \qquad \text{ and } \qquad a_V < 1,
\end{equation*}
then $\mu_\infty$ is the global minimiser in $\mathcal P (\Rd)$ of the relaxation of $\mathcal F$
and it is an attractor for some initial data.

\paragraph{Structure of the paper.}
In \Cref{sec:2} we write the theory in $\Br$ for a regularised problem where the fast-diffusion is replaced by a smooth elliptic non-linearity $\Phi$.
In \Cref{sec:agg-diff equation} we construct solutions of \eqref{eq:main bounded domain}, by passing to the limit as $\Phi(s) \to s^m$ the solutions of \Cref{sec:2}. In \Cref{sec:mass} we show that mass functions $M$ of the solutions of \Cref{sec:2,sec:agg-diff equation} are solutions in a suitable sense of Problem \eqref{eq:mass}, and we prove regularity and a priori estimates. In \Cref{sec:concentration} we construct initial data $\rho_0$ so that the mass $M$ is non-decreasing in time as well a space. We show that these solutions $M$ concentrate in the limit, a main goal of the paper. We recall that this means the formation of a jump at $v=0$ for $t=\infty$. \Cref{sec:FR} is dedicated to the minimisation of $\mathcal F_R$ for functions defined in $\Br$. We prove that the minimisers are precisely of the form $\mu_{\infty, \mathfrak m , R}$ described above.
In \Cref{sec:Rn}, we pass to the limit as $R \to \infty$ in terms of the mass. We show that the mass functions for suitable initial data still concentrate.
	We discuss minimisation of the function $\mathcal F$. We show the class of potentials $V$ that make $\mathcal F$ bounded below is more restrictive than for $\mathcal F_R$, and provide suitable assumptions so that $\mu_\infty$ is a minimiser.
We list some comments and open problems in \Cref{sec:final comments}. We conclude the paper with two appendixes. The first, \Cref{sec:classical regularity} recalls results from \cite{DiBenedetto1993} and compacts them into a form we use for $M$. \Cref{sec:appendix b} is devoted to mixing partial space and time regularities into Hölder regularity in space and time.

\paragraph{A comment on the notation.}
Throughout the document, we deal with the spatial variable $x$, the radial variable
$r = |x|$, and the volume variable $v = |B_1| |x|^d$. Since it will not lead to confusion and simplifies the notation, radial functions of $x$ will sometimes be evaluated or differentiated in $r$ or $v$. This must be understood as the correct substitution.
\normalcolor

\section{The regularised equation in $\Br$}
\label{sec:2}
Following the theory of non-linear diffusion, we consider in general
\begin{equation}
\label{eq:main regularised bounded}
\tag{P$_{\Phi,R}$}
	\begin{dcases}
		\frac{\partial u}{\partial t} = \Delta \Phi (u) +  \diver ( u E ) & \text{in }(0,\infty) \times \Br \\
		(\nabla \Phi(u) + u E) \cdot x = 0, &  \text{on }  (0,\infty) \times \partial \Br, \\
		u (0,x) = u_0 (x).
	\end{dcases}
\end{equation}
We assume that $\Phi \in C^1$ and elliptic we think of the problem as
\begin{equation*}
	\frac{\partial u}{\partial t} = \Delta \Phi (u) +  \nabla u \cdot E + u \diver E.
\end{equation*}
Furthermore, we assume
\begin{equation}
\label{eq:E point outwards}
	E(x) \cdot x = 0 , \qquad \text{on } \partial \Br
\end{equation}

\begin{remark}
Our results work in a general bounded domain $\Omega$, where the assumption on $E$ is that $E \cdot n(x) = 0$ on $\partial \Omega$. However, we write them in a ball of radius $R$ since our main objective is to study the long-time asymptotics of radially symmetric solutions.
\end{remark}

The diffusion corresponds to the flux
$	a(u, \nabla u) = \Phi'(u) \nabla u.
$
When $\Phi, E$ are smooth and we assume $\Phi$ is uniformly elliptic, in the sense that there exist constants such that
\begin{equation}
\label{eq:Phi elliptic}
	0 < c_1 \le \Phi'(u) \le c_2 < \infty,
\end{equation}
existence, uniqueness, and maximum principle hold from the classical theory.
The literature is extensive:
	 in $\Rd$ this issue was solved at the beginning of the twentieth century (see \cite{ladyzhenskaia1968parabolic}),
	 in a bounded domain with Dirichlet boundary condition the result can be found in \cite{Gilbarg+Trudinger2001}, and
	 the case of Neumann boundary conditions was studied by the end of the the twentieth century (for example \cite{Amann1990}), where the assumptions on the lower order term were later generalised (see, e.g. \cite{Yin2005}).
Following \cite{Amann1990}, we have that, if $u_0 \in C^2 (\overline \Br)$ then the solution $u$ of \eqref{eq:main regularised bounded} is such that
\begin{equation}
	\label{eq:bounded regularised regularity of $u$}
	u \in C^1\Big((0,T); C(\overline \Br) \Big) \cap C \Big((0,\infty) ; C^2 (\overline \Br)\Big) \cap C\Big([0,\infty)\times \overline \Br\Big).
\end{equation}

Let us obtain further properties of the solution of \eqref{eq:main regularised bounded}.

\begin{theorem}[$L^p$ estimates]
	\label{thm:Lp estimates}
	Assume $E(x) \cdot n(x) \ge 0$. For classical solutions we have that
	\begin{align}
	\label{eq:regularisation Lp estimates}
	\| u(t)_\pm \|_{L^p } &\le e^{ 
		\frac{p-1}{p} \normalcolor
		\| \diver E \|_{L^\infty} t }\| (u_0)_\pm \|_{L^p}.
	\end{align}
\end{theorem}

\begin{proof}
Let $j$ be convex. We compute
\begin{align*}
	\frac{\diff }{\diff t} \int_\Br j(u) &= \int_\Br j'(u) \diver ( \nabla  \Phi(u) + u E  ) =-  \int_\Br j''(u) \nabla u \cdot ( \nabla  \Phi(u) + u E  ) \\
	&=-  \int_\Br j''(u) \Phi'(u) |\nabla u|^2 - \int_\Br  j''(u) u \nabla u \cdot E  \le - \int_\Br   \nabla F(u) \cdot E
\end{align*}
where $F'(u) = j''(u) u \ge 0$ and we can pick $F(0) = 0$. Hence $F \ge 0$.
If $j$ has a minimum at $0$, then $j' \ge 0$ and $F \ge 0$. Since $E \cdot x \ge 0$,
\begin{align*}
	\int_\Br \nabla F(u) \cdot E &= \int_\Br \diver (F(u) E) - \int_{\Br} F(u) \diver E = \int_{\partial \Br} F(u) E \cdot \frac{x}{|x|} - \int_{\Br} F(u) \diver E \\
	&\ge - \int_{\Br} F(u) \diver E.
\end{align*}
Finally, we recover
\begin{align*}
	\int_\Br j(u (t))  &\le \int_\Br j(u_0)   + \int_0^t \int_\Br   F(u) \diver E .
\end{align*}
When $j (s) = s_\pm^p$, $ j''(s) = p (p - 1) s_\pm^{p - 2}$ we have $ F(s) = (p-1) s_\pm^{p }$.
Applying \eqref{eq:V no aggregation on boundary} we show that\\
\begin{align*}
	\int_\Br u(t)_\pm^p  &\le \int_\Br (u_0)_\pm^p  + (p-1) \| \diver E \|_{L^\infty} \int_0^t \int_\Br  u_\pm^p.
\end{align*}
By Gronwall's inequality we have that
\begin{align*}
	\int_\Br u(t)_\pm^p  &\le e^{(p-1)\| \diver E \|_{L^\infty} t }\int_\Br (u_0)_\pm^p .
\end{align*}
Taking the power $1/p$ we have \eqref{eq:regularisation Lp estimates} for $p < \infty$
and letting $p \to \infty$ we also obtain the $L^\infty$ estimate.
\end{proof}

\begin{theorem}[Estimates on $\nabla \Phi(u)$]
Let
$\Psi(s) = \int_0^s \Phi(\sigma) \diff \sigma$. Then, we have that
\begin{equation}
	\label{eq:smooth estimate nabla Phi u}
	 \frac 1 2
	 \int_{0}^{T} \int_\Br |\nabla \Phi(u)|^2  \le  \int_{\Br} \Psi(u_0) + \frac {\| E \|_{L^\infty}^2} 2 \int_0^T \int_{B_R} u(t,x)^2 \diff x \diff t.
\end{equation}
\end{theorem}
\begin{proof}
Multiplying by $\Phi(u)$ and integrating
\begin{equation*}
	\int_\Br u_t \Phi(u) = - \int_\Br \nabla  \Phi(u) \cdot (\nabla \Phi(u) + u E) = - \int_\Br |\nabla \Phi(u)|^2 - \int_\Br u \nabla \Phi(u) E.
\end{equation*}
We have that
\begin{equation*}
	\frac{\diff }{\diff t} \int_\Br \Psi(u)  + \int_\Br |\nabla \Phi(u)|^2 \le \| u (t) \|_{L^2} \| \nabla  \Phi(u(t)) \|_{L^2} \| E \|_{L^\infty}.
\end{equation*}
Applying Young's inequality we obtain
\begin{equation*}
	\frac{\diff }{\diff t} \int_\Br \Psi(u)  + \frac 1 2 \int_\Br |\nabla \Phi(u)|^2 \le \frac 1 2 \| u (t) \|_{L^2}^2 \| E \|_{L^\infty}^2.
\end{equation*}
Notice since $\Phi' \ge 0$ we have that $\Psi \ge 0$.
Hence, we deduce the result.
\end{proof}

If $\Psi(u_0) \in L^1$ and $u_0 \in L^2$ then the right-hand side is finite due to \eqref{eq:regularisation Lp estimates}.
\begin{remark}
	When $\Phi(s) = s^m$ then $\Psi(s) = \frac{1}{m+1}s^{m+1}$.
\end{remark}
In order to get point-wise convergence, we follow the approach for the Fast Diffusion equation proposed in \cite[Lemma 5.9]{Vazquez2007}. Define \begin{equation*}
	Z(s) = \int_0^s \min \{ 1, \Phi'(s) \} \diff s, \qquad z (t,x) = Z(u(t,x)).
\end{equation*}
\begin{corollary}
	We have that
	\begin{equation}
	\label{eq:smooth estimate nabla z}
	 \int_{0}^{T} \int_\Br |\nabla z|^2  \le \int_{\Br} \Psi(u_0) + \frac{\| E \|_{L^\infty}^2} 2 \int_0^T \int_{B_R} u(t,x)^2 \diff x \diff t.
\end{equation}
\end{corollary}
\begin{proof}
	Notice
$
	|\nabla z| \le |\Phi'(u)||\nabla u| = |\nabla \Phi(u)|.
$
\end{proof}

\begin{lemma}[Estimates on $u_t$ and $\nabla \Phi (u)$]
	\label{lem:bounded regularised estimates ut and nabla Phi u}
	Assume $E \cdot x = 0$ on $\partial B_R$,
	$u \in L^\infty(0,T; L^2 (B_R))$, $\Phi(u) \in L^2(0,T; H^1 (B_R))$, $\Phi (u_0) \in H^1 (B_R)$ then
	\begin{equation*}
		\Phi(u) \in L^\infty(0,T; H^1 (B_R)) \qquad \text{ and } \qquad u_t \in L^2 ( (0,T) \times \Br).
\end{equation*}
We also have, for $z(t,x) = Z(u(t,x))$ that
\begin{equation}
\label{eq:estimate Phi'uu_t}
\begin{aligned}
	\int_0^T \int_\Br |z_t|^2 &\le C \Bigg( \int_\Br |\nabla \Phi(u_0)|^2  + \frac{1} 2\int_0^T \int_\Br \Phi'(u) |\nabla u|^2|  E|^2 \\
	 & \qquad \qquad + \int_\Br u(0) \nabla \Phi(u_0) \cdot E + \int_{B_R} |u(T)|^2 |E|^2 \Bigg).
\end{aligned}
\end{equation}
\end{lemma}

\begin{proof}
Again we we will use the notation $w = \Phi(u)$. When $u$ is smooth, we can take $w_t$ as a test function and integrate in $\Br$. Notice that $w_t = \Phi'(u) u_t$, so
\begin{equation*}
	\int_\Br \Phi'(u) |u_t|^2 = \int_\Br w_t \Delta w + \int_\Br w_t \diver (u E)
\end{equation*}
Since $\nabla w = 0$ on $\partial \Br$, then also $\nabla w_t = 0$. We can integrate by parts to recover
\begin{equation*}
	\int_\Br \Phi'(u) |u_t|^2 = - \frac{\diff }{\diff t}\int_\Br |\nabla w|^2 + \int_{\partial B_R} w_t u E \cdot \frac{x}{|x|} - \int_\Br u \nabla w_t \cdot E .
\end{equation*}
Using assumption \eqref{eq:E point outwards} the second term on the right-hand side vanishes.
Integrating in $[0,T]$ we have
\begin{equation*}
	\int_0^T \int_\Br \Phi'(u) |u_t|^2 + \int_\Br |\nabla w (T)|^2  = \int_\Br |\nabla w(0)|^2 - \int_0^T \int_\Br u \nabla w_t \cdot E
\end{equation*}
Integrating by parts in time the last integral
\begin{align*}
	\int_0^T \int_\Br \Phi'(u) |u_t|^2 + \int_\Br |\nabla w (T)|^2  &= \int_\Br |\nabla w(0)|^2  + \int_0^T \int_\Br u_t \nabla w \cdot E \\
	&\qquad + \int_\Br u(0) \nabla w(0) \cdot E - \int_\Br u(T) \nabla w(T) \cdot E.
\end{align*}
Notice that $u_t \nabla w = \Phi'(u)^{\frac 1 2} u_t \Phi'(u)^{\frac 1 2} \nabla u$. Applying Young's inequality, we deduce
\begin{equation}
\label{eq:estimate Phi'uu_t}
\begin{aligned}
	\frac 1 2 \int_0^T \int_\Br \Phi'(u) |u_t|^2 + \frac 1 2 \int_\Br |\nabla w (T)|^2  &\le \int_\Br |\nabla w(0)|^2  + \frac{1} 2\int_0^T \int_\Br \Phi'(u) |\nabla u|^2|  E|^2  \\
	&\qquad + \int_\Br u(0) \nabla w(0) \cdot E + \int_{B_R} |u(T)|^2 |E|^2.
\end{aligned}
\end{equation}
From the estimates above, we know that $c_1 |\nabla u|  \le \Phi'(u) |\nabla u| = |\nabla \Phi (u)| \in L^2$. Similarly, the result follows.
Finally, we use that
\begin{equation*}
	|z_t|^2 \le |Z'(u)|^2 | u_t| ^2 \le |\Phi'(u)| |u  _t|^2 .
\end{equation*}
using that  $Z' =  \min \{ 1, \Phi' \}$.
\end{proof}

\subsection{Free energy and its dissipation when $E = \nabla V$}
When $E = \nabla V$ we have, again, a variational interpretation of the equation that leads to additional a priori estimates.
We can rewrite equation \eqref{eq:main regularised bounded} as
\begin{equation}
\label{eq:bounded regularised Wassertein}
	\frac{\partial u}{\partial t} = \diver \left(  \Phi'(u) \nabla u + u \nabla V \right) = \diver \left(  u \left\{ \frac{ \Phi'(u) }{u } \nabla u + \nabla V \right\} \right) =  \diver \left(  u \nabla  \left\{ \Theta (u) + V \right\} \right)
\end{equation}
where
\begin{equation}
\label{eq:Theta}
	\Theta(s) = \int_1^s \frac{\Phi'(\sigma)}{\sigma} \diff \sigma.
\end{equation}

\begin{remark}
	Since $c_1 \le \Phi'(\rho) \le c_2$ then $ \Theta(\rho) \sim \alpha \ln \rho$ so $\Theta^{-1} (\rho) \sim e^{\alpha^{-1} \rho}$. In particular $\Theta (0) = +\infty$. This is why we have to integrate from $1$ in this setting.
	However, when $\Phi (s) = s^m$ then $\Theta (s) = \frac{m}{m-1} ( s^{ {m-1}} - 1)$. So $\Theta^{-1} (s) = (  1 - \frac{1-m}{m} s  )^{ \frac 1 {m - 1} }$. For $\Phi$ elliptic then $\Theta^{-1} : \mathbb R \to [0,+\infty)$. However, for the FDE passing to the limit we are restricted to $s \le \frac m {1-m}$.
\end{remark}
Formulation \eqref{eq:bounded regularised Wassertein} shows that this equation is the 2-Wasserstein gradient flow of the free energy
\begin{equation*}
	\mathcal F_\Phi [u] = \int_\Br \left( \int_1^{u(x)}\Theta(s)\diff s + V(x) u(x) \right) \diff x.
\end{equation*}
Along the solutions of \eqref{eq:main regularised bounded} it is easy to check that
\begin{equation}
\label{eq:bounded regularised decay of free energy}
	\frac{\diff}{\diff t} \mathcal F_\Phi [u (t)] = - \int_\Br u \left| \nabla (\Theta(u) + V ) \right|^2 \diff x \le 0.
\end{equation}
Also, by integrating in time we have that
\begin{equation}
\label{eq:bounded regularised estimate of free energy disipation}
	0 \le \int_0^T \int_\Br u \left| \nabla (\Theta(u) + V ) \right|^2 \diff x \diff t= \mathcal F_\Phi [u_0] - \mathcal F_\Phi[u(t)]
\end{equation}

Finally, let us take a look at the stationary states. For any $H \in \mathbb R$, the solution of $\Theta(u) + V = - H$
is a stationary state.
Since $\Theta: [0,+\infty) \to \mathbb R$ is non-decreasing, we have that $H = - \Theta (u(0))$. We finally define
\begin{equation*}
	u_{V+H} \, \defeq \,  \Theta^{-1} \Big( -( H + V )  \Big).
\end{equation*}

\begin{remark}
	When $\Phi$ is elliptic $u_{V+H} \le \Theta^{-1} (-H)$.
	In the case of the FDE we have
	\begin{equation*}
		u_{V+H} = \left(  1 + \tfrac{1-m}{m} (H + V)  \right)^{ \frac 1 {m - 1} } = \rho_{V+h}.
	\end{equation*}
	where $h = H + \frac{m}{1-m}$.
	When $h > 0$ we have $\rho _{V+h}$ is bounded, but $\rho_V$ is not bounded.
\end{remark}

\subsection{Comparison principle and $L^1$ contraction}
Let us present a class of solutions which have a comparison principle, and are therefore unique.
\begin{definition}
We define strong $L^1$ solutions of \eqref{eq:main regularised bounded} as distributional solutions such that
\begin{enumerate}
	\item  $u \in C( [0,T] ; L^1 (\Br) )$.
	\item 	$\Phi(u) \in L^1(0,T; W^{1,1} (\Br))$ , $\Delta \Phi(u) \in L^1 ((0,T) \times \Br)$.
	\item $u_t \in L^2 (0,T; L^1(\Br))$.
\end{enumerate}
\end{definition}

\begin{theorem}
	\label{thm:ball comparison Phi smooth}
	Assume $E \cdot n(x) = 0$. Let $u, \overline u$ be two strong $L^1$ solutions of \eqref{eq:main regularised bounded}. Then, we have that
	\begin{equation*}
		\int_\Br [u (t) - \overline u(t)]_+ \le \int_\Br [u (0) - \overline u (0)]_+
	\end{equation*}
	In particular
	$
		\| u(t) - \overline u (t) \|_{L^1 (\Br)} \le \| u(0) - \overline u(0)\|_{L^1 (\Br)}
	$
	and,
	for each $u_0 \in L^1 (\Br)$, there exists at most one strong $L^1$ solution.
\end{theorem}

\begin{proof}
	We now have that $w = \Phi (u) - \Phi (\overline u)$. Let $j$ be convex and denote $p = j'$. We have, using the no flux condition
	\begin{align*}
		\int_0^T \int_\Br (u - \overline u)_t p(w) &= \int_0^T\int_\Br p(w) \diver \{ \nabla w + (u - \overline u)E  \} \\
		&=- \int_0^T\int_\Br p'(w) | \nabla w|^2
		+  \int_0^T\int_\Br  p(w) \diver \left\{ (u - \overline u)   E \right\} .
	\end{align*}
	Notice that $\nabla u = \frac{1}{\Phi'(u)} \nabla \Phi(u) \in L^1((0,T) \times \Br)$ due to \eqref{eq:Phi elliptic}.
	Expanding the divergence, 
	we have
	\begin{equation*}
		\int_0^T\int_\Br (u - \overline u)_t p(w) \le \int_0^T\int_\Br p(w) \left( \nabla (u - \overline u) E + (u - \overline u) \diver E \right) .
	\end{equation*}
	Then, as $p \to \sign_+$, we have $p(w) \to \sign_0^+ ( w ) = \sign_0^+ (u - \overline u)$ and
	\begin{equation*}
		\int_0^T\int_\Br ( [u - \overline u]_+ )_t  \le \int_0^T\int_\Br  \left( \nabla [u - \overline u]_+ E + [u - \overline u]_+ \diver E \right) =  \int_0^T\int_\Br  \diver (  [u - \overline u]_+ E ).
	\end{equation*}
	Using again that $E \cdot n(x) = 0$ on $\partial \Br$, we recover a $0$ on the right hand side. This completes the proof.
\end{proof}

\begin{remark}[Uniform continuity in time]
Because of the $L^1$ contraction, and the properties of the semigroup
$
	\| u(t + h) - u(t) \|_{L^1} \le \| u(h) - u(0) \|_{L^1}.
$
If $u(h) \to u(0)$ in $L^1$, we have uniform continuity in time $\omega(h) = \| u(h) - u(0) \|_{L^1}$.
\end{remark}

\begin{remark}[On the assumption $E \cdot x = 0$ on $\partial \Br$]
\label{sec:condition VR no flux}
Notice that to recover  the $L^p$ estimates in \Cref{thm:Lp estimates} (which depend on $\| \diver E \|_{L^\infty}$) we assume only that $E \cdot x \ge 0$ on $\partial \Br$. However, later (as in \Cref{lem:bounded regularised estimates ut and nabla Phi u} and \Cref{thm:ball comparison Phi smooth}) we require $E \cdot x = 0$ on $\partial \Br$. The estimates in these results do not include $\diver E$, and so it seems possible to extended the results to this setting by approximation.
\end{remark}

\section{The Aggregation-Fast Diffusion Equation}
\label{sec:agg-diff equation}
We start this section by providing a weaker notion of solution
\begin{definition}
	We say that $\rho \in L^1((0,T) \times \Br$ is a weak $L^1$ solution of \eqref{eq:main bounded domain} if $\rho^m \in L^1 ( 0,T; W^{1,1} (\Br) )$ and, for every $\varphi \in L^\infty (0,T; W^{2,\infty} (\Br) \cap W^{1,\infty}_0 (\Br)) \cap C^1 ([0,T]; L^1(\Br))$ we have that
	\begin{equation*}
		\int_\Br \rho(t) \varphi(t) - \int_0^t \int_\Br \rho(s) \varphi_t (s) \diff s = - \int_0^t \int_\Br \left( \nabla  \rho^m \cdot \varphi + \rho \nabla V \cdot \nabla \varphi \right) + \int_\Br \rho_0 \varphi(0).
	\end{equation*}
	for a.e. $t \in (0,T)$.
\end{definition}
If $\nabla V \cdot n(x) = 0$ we then have $\nabla \rho^m \cdot n = 0$ and we can write the notion of very weak $L^1$ solution by integrating once more in space the diffusion term
\begin{equation*}
		\int_\Br \rho(t) \varphi(t) - \int_0^t \int_\Br \rho(s) \varphi_t (s) \diff s =  \int_0^t \int_\Br \left(   \rho^m \Delta \varphi - \rho \nabla V \cdot \nabla \varphi \right) + \int_\Br \rho_0 \varphi(0).
	\end{equation*}

\begin{theorem}[$L^1$ contraction for $H^1$ solutions bounded below]
	Assume that
		$\rho, \overline \rho$ are weak $L^1$ solutions of \eqref{eq:main bounded domain} with initial data $\rho_0 $ and $\overline \rho_0$,
		$\rho, \overline \rho \in H^1 ((0,T) \times \Br)$,
		and
		$\rho, \overline \rho \ge c_0 > 0$.
	Then
	\begin{equation*}
		\int_\Br (\rho(t) - \overline \rho(t))_+ \le \int_\Br (\rho_0 - \overline \rho_0)_+.
	\end{equation*}
\end{theorem}
\begin{proof}
	Since the solutions are in $H^1$ and are bounded below, then $\rho^m, \overline \rho^m \in H^1 ((0,T) \times B_R)$. 	Let $p$ be non-decreasing and smooth.
	By approximation by regularised choices, let us define $w = \rho^m - \overline \rho^m$ and $\varphi = p ( w )$. Thus we deduce
		\begin{align*}
		\int_0^t \int_\Br ( \rho(s) - \overline \rho(s))_t p(w) &= -\int_0^t \int_\Br \left( p'(w) |\nabla w|^2 + ( \rho - \overline \rho)  \nabla p(w) \cdot \nabla V \right) .
	\end{align*}
	Proceeding as in \Cref{thm:ball comparison Phi smooth} for $\Phi$ smooth and using \eqref{eq:V no aggregation on boundary} we have that
	\begin{equation*}
		\int_0^t \int_\Br ( [ \rho(s) - \overline \rho(s) ]_+)_t \diff s \le 0,
	\end{equation*}
	and this proves the result.
\end{proof}

We can now construct a semigroup of solutions.
We begin by constructing solutions for regular data, by passing to the limit in regularised problems with a sequence of smooth non-linearities $\Phi_k(s) \to \Phi(s) = s^m$.
We consider the sequence $\Phi_k$ of functions given by $\Phi_k(0) = 0$ and
\begin{equation}
\label{eq:Phik}
	\Phi_k' (s) \sim \begin{dcases}
		m k^{m-1} & s > k, \\
		m s^{m-1} & s \in [k^{-1}, k], \\
		m k^{1-m} & s < k^{-1}
	\end{dcases}
\end{equation}
up to a smoothing of the interphases.
We define
\begin{equation*}
		Z(s) = \int_0^s \min \{ 1 , \Phi' (\sigma) \} \diff \sigma = \int_0^s \min \{ 1 , m \sigma^{m-1} \} \diff \sigma = \begin{dcases}
			s & s < m^{-\frac 1 {1-m}} \\
			C_m + s^m & s \ge m^{-\frac 1 {1-m}}
		\end{dcases}
		.
	\end{equation*}

\begin{theorem}[Existence of solutions for regular initial data]
	\label{thm:existence bounded domain regular data}
	Assume $V \in W^{2,\infty} (\Br)$, $V \ge 0$, $V (0) = 0$, $V \cdot x = 0$ on $\partial \Br$ and the technical assumption \eqref{eq:rhoV in L1ee loc}. 	Let $\rho_0 $ be such that
	\begin{equation*}
		0 < \ee \le \rho_0 \le \ee^{-1}, \qquad \rho_0 \in H^1 (\Br).
	\end{equation*}
	Then, the sequence $u_k$ of solutions for \eqref{eq:main regularised bounded} where $\Phi = \Phi_k$ given by \eqref{eq:Phik} is such that
	\begin{alignat*}{2}
		u_k &\rightharpoonup \rho \qquad && \text{weakly in } H^1 ((0,T) \times \Br) \\
		u_k &\to \rho \qquad && \text{a.e. in }(0,T) \times \Br \\
		\Phi_k (u_k) &\rightharpoonup \rho^m \qquad &&\text{weakly in } L^2 (0,T, H^1 (\Br))
	\end{alignat*}
	and $\rho$ is a weak $L^1$ solution of the problem. Moreover, we have that $\rho \ge \omega(\ee) > 0$,
	\begin{equation*}
		\| \rho(t) \|_{L^1} = \|\rho_0 \|_{L^1}, \qquad 	\| \rho(t) \|_{L^q} \le e^{ t \| \Delta V \|_{L^\infty (B_R)} } \| \rho_0 \|_{L^q}.
	\end{equation*}
	In fact, $\rho$ is the unique weak $L^1$ solution which is $H^1$ and bounded below.
\end{theorem}

\begin{proof}
	First, we point out that that $\Phi_k(\rho_0) \in W^{1,\infty} (B_R)$.
	Hence, for the approximation, by \eqref{eq:regularisation Lp estimates}, $u_k \in L^\infty ((0,T) \times B_R)$ and, due to \eqref{eq:smooth estimate nabla Phi u}, $\Phi_k (u_k) \in L^2( 0,T; H^1 (B_R))$ with uniform norm bounds.
	This ensures (up a subsequence)
	\begin{alignat*}{2}
		u_k &\rightharpoonup \rho \qquad && \text{weak-$\star$ in } L^\infty((0,T) \times B_R),\\
		\Phi_k (u_k) & \rightharpoonup \phi \qquad && \text{weakly in }L^2 ( (0,T) \times B_R ), \\
		\nabla \Phi_k (u_k) &\rightharpoonup \nabla \phi \qquad && \text{weakly in }L^2 ( (0,T) \times B_R ), \\
		Z_k (u_k (t,x)) &\rightharpoonup Z^* \qquad && \text{weakly in }H^1 ((0,T) \times B_R), \\
		Z_k (u_k (t,x)) &\to Z^* \qquad && \text{a.e. in } (0,T) \times B_R.
	\end{alignat*}

	Let us characterise $\phi$ as $\Phi (\rho)$. For $k > m^{\pm \frac 1 {1-m}}$ we can compute clearly $\min \{1, \Phi_k'\}$ from \eqref{eq:Phik} and hence we have
	\begin{equation*}
		Z_k'(s) - Z'(s) =
		\begin{dcases}
			0 & s \in [0,k], \\
			m(k^{m-1} - s^{m-1}) & s > k ,
		\end{dcases}
	\end{equation*}
	Since $u_k$ are uniformly bounded in $L^\infty$, taking $k$ large enough we have that
	\begin{equation*}
		Z_k(u_k) = Z(u_k).
	\end{equation*}
	Thus $Z(u_k)$ converges pointwise to $Z^*$. But $Z$ is continuous and strictly increasing, so it is invertible. Thus $u_k \to Z^{-1} (Z^*)$ a.e. in $(0,T) \times B_R$. Since, if both exist, the weak $L^2$ and a.e. limits must coincide (apply Banach-Saks theorem and Césaro mean arguments), then $u_k \to \rho$ a.e. in $(0,T) \times B_R$. Finally, due to the locally uniform convergence of $\Phi_k \to \Phi$,  $\Phi_k (u_k) \to \Phi (\rho)$ a.e. and hence $\phi = \Phi (\rho)$.
	
	We can now upgrade to strong convergence, using the uniform $L^\infty$ bound $|u_k| \le C$. Hence, together with the point-wise convergence, we can apply the Dominated Convergence Theorem to show that our chosen subsequence also satisfies
	$
		u_k \to \rho $ in $ L^q ( (0,T) \times B_R ), \ \forall q \in [1,\infty).
	$
	
	Let us show that we maintain an upper and positive lower bound. The upper bound is uniform $e^{t \| \Delta V \|_{L^\infty (\Br)}} \| \rho_0 \|_{L^\infty (\Br)}$. Since as $H \to \infty$ the stationary states $\Theta_k^{-1} (V + H)$ tend to cero uniformly, then we can choose $H_k$
	so that
	\begin{equation*}
		\rho_0 \ge \Theta_k^{-1} (V + H_k) \ge \omega(\ee).
	\end{equation*}
	Thus, $u_k \ge \omega(\ee)$ and, therefore, so is $\rho$.
	In fact, due to this lower bound
	\begin{equation*}
		|\nabla u_k| \le \frac {1}{Z'(\omega(\ee))} |\nabla Z(u_k)| , \qquad |( u_k)_t| \le \frac {1}{Z'(\omega(\ee))} |(u_t)|
	\end{equation*}
	and so the convergence $u_k \rightharpoonup \rho$ is also weak in $H^1$ (up to a subsequence). But then $\rho$ is the unique weak $L^1$ solution with this property. Since the limit is unique, the whole sequence $u_k$ converges to $\rho$ all the senses above.
\end{proof}

\begin{corollary}[Approximation of the free energy]
	\label{cor:bounded domain and data properties of free energy}
	Under the hypothesis of \Cref{thm:existence bounded domain regular data} we have that
	\begin{equation*}
		\mathcal F_{\Phi_k} [u_k (t)] \to \mathcal F_R[\rho(t)] , \qquad \text{for a.e. } t > 0 .
	\end{equation*}
	and
	\begin{equation*}
		\int_0^T \int_\Br \rho \left| \nabla \left(  \tfrac{m}{m-1} \rho^{m-1} + V  \right)\right|^2 \le \mathcal F_R [\rho_0] - \mathcal F_R[\rho(T)].
	\end{equation*}
	In particular, $\mathcal F_R [\rho(t)]$ is a non-increasing sequence.
	
\end{corollary}
\begin{proof}
	Since $u_k \to \rho$ converges a.e. in $(0,T) \times \Br$, then for a.e. $t > 0$ we have that $u_k (t) \to \rho(t)$. Since $u_k$ is uniformly bounded, then the Dominated Convergence Theorem ensures the convergence of $\mathcal F_{\Phi_k} [u_k]$.
	
	Taking into account \eqref{eq:bounded regularised estimate of free energy disipation}, then the sequence
	$
		u_k^{\frac 1 2} \nabla (\Theta(u_k) + V)
	$
	is uniformly in $L^2 ((0,T) \times \Br)$. Therefore, up to a subsequence, it has limit $\xi (x)$. We can write
	\begin{equation*}
		u_k^{\frac 1 2} \nabla (\Theta(u_k) + V) = \frac{\Phi_k'(u_k)}{u_k^{\frac 1 2}} \nabla u_k + u_k^{\frac 1 2} \nabla V = u_k^{-\frac 1 2} \left(  \nabla \Phi_k(u_k) + u_k \nabla V \right).
	\end{equation*}
	We know that
	$
		\nabla \Phi_k(u_k) + u_k \nabla V \rightharpoonup \nabla \rho^m + \rho \nabla V
	$
	weakly in $L^2$. On the other hand, since we know $u_k, \rho \ge \omega(\ee)$ we can apply the intermediate value theorem to show that, up a to further subsequence,
	\begin{equation*}
		\int_0^T \int_\Br |u_k^{-\frac 1 2} - \rho^{-\frac 1 2}|^2 \diff x = \int_0^T\frac 1 4 \int_\Br |\eta_k (x)|^{-3} |u_k - \rho|^2 \diff x \le C \int_0^T \int_\Br  |u_k - \rho|^2 \diff x \to 0.
	\end{equation*}
	where the strong convergence $L^2$ follows, up to a further subsequence, from the weak $H^1$ convergence.
	Using the product of strong and weak convergence
	\begin{equation*}
		u_k^{\frac 1 2} \nabla (\Theta(u_k) + V) \rightharpoonup \rho^{\frac 1 2} \nabla \left( \tfrac m{m-1} \rho^m + V \right), \qquad \text{weakly in } L^1 ((0,T) \times \Br).
	\end{equation*}
	But this limit must coincide with $\xi$, so the limit holds also weakly in $L^2$. The weak lower-continuity of the $L^2$ yields the result.
\end{proof}
We are also able to deduce from these energy estimates an $L^1$ bound of $\nabla \rho^m$. Unlike \eqref{eq:smooth estimate nabla Phi u} this bound can use only local boundedness of $\nabla V$.
\begin{corollary}
	In the hypothesis of \Cref{thm:existence bounded domain regular data} we have that 	\begin{align}
	\label{eq:bounded estimate nabla rhom in L1}
	\int_0^T \int_K |\nabla \rho^m| &\le \| \rho_0 \|_{L^1(\Br)}^{\frac 1 2} \left(\mathcal F_R[\rho_0] - \mathcal F_R[\rho(T)] + \int_0^T \int_K \rho |\nabla V|^2 \right)^{\frac 1 2}, \qquad \forall K \subset \overline{\Br}.
\end{align}
\end{corollary}
\begin{proof}
	We therefore have that
\begin{align*}
	\int_0^T \int_K |\nabla \rho^m| &= \int_0^T\int_K \rho  |\tfrac{m}{m-1} \nabla \rho^{m-1}| \le  \int_0^T \| \rho(t) \|_{L^1(\Br)}^{\frac 1 2} \left( \int_K \rho \left| \nabla  \tfrac{m}{m-1} \rho^{m-1} \right|^2 \right)^{\frac 1 2} \diff t
\end{align*}
Hence, we conclude the result using \Cref{cor:bounded domain and data properties of free energy}, Jensen's inequality and the conservation of the $L^1$ norm.
\end{proof}

Now we move to $L^1$ data.
We first point out that $L^m(\Br) \subset L^1(\Br)$ so any $\rho \in L^1$ has finite $\mathcal F_R [\rho]$. To be precise, by applying Hölder's inequality with $p = \frac 1 m > 1$ we have the estimate
\begin{equation}
\label{eq:L1 controls Lm over compacts}
	\int_K \rho^m \le
	|K|^{{1-m}} \|\rho\|_{L^1(K)}^{m}.
\end{equation}

Now we apply density in $L^1$ of the solutions with ``good'' initial data, via the comparison principle
\begin{theorem}[Existence of solution for $L^{1}$ initial data]
	\label{thm:existence bounded L1}
	Under the assumptions of \Cref{thm:existence bounded domain regular data},
	there exists a semigroup $S(t) : L_+^{1} (\Br) \to L^{1} (\Br)$ with the following properties
	\begin{enumerate}
		\item For $ 0 < \ee^{-1} \le \rho_0 \le \ee$ and $\rho_0 \in H^1 (\Br)$, $S(t) \rho_0$ is the unique weak $L^1$ solution constructed in \Cref{thm:existence bounded domain regular data}.
		\item We have $\| S(t) \rho_0 \|_{L^1(\Br)} = \| \rho_0 \|_{L^1 (\Br)}$.
		\item We have $L^1$ comparison principle and contraction
		\begin{equation*}
			\int_\Br [S(t) \rho_0 - S(t) \overline \rho_0 ]_+ \le \int_\Br [ \rho_0 -  \overline \rho_0 ]_+, \qquad \int_\Br |S(t) \rho_0 - S(t) \overline \rho_0 | \le \int_\Br | \rho_0 -  \overline \rho_0 |.
		\end{equation*}
		\item \label{it:existence bounded L1 item 3}
		If $\rho_0 \in  L^{1+\ee}_+(\Br)$ is the limit of the solutions $u_k$ of \eqref{eq:main regularised bounded} with \eqref{eq:Phik} and
		\begin{equation*}
			\| \rho (t) \|_{L^{1+\ee}} \le Ce^{\frac{\ee}{1+\ee} t \| \Delta V \|_{L^\infty}} \| \rho_0 \|_{L^{1+\ee}}.
		\end{equation*}
		\item If $\rho_0 \in  L^1_+ (\Br)$ and \eqref{eq:V no aggregation on boundary}, then $\rho$ is a very weak $L^1$ solution.

		\item If $\rho_0 \in  L^1_+ (\Br)$, then $\mathcal F_R[\rho(t)]$ is non-increasing and we have \eqref{eq:bounded estimate nabla rhom in L1}. Hence, it is a weak $L^1$ solution.
	\end{enumerate}
\end{theorem}

\begin{remark}
	Notice that there is no concentration in finite time. This is due the combination of the $L^1$ contraction with the uniform $L^{1+\ee}$ estimate \eqref{eq:regularisation Lp estimates}. By the $L^1$ contraction, the sequence $S(t) \max\{ \rho_0 , k \}$ is Cauchy in $L^1$ and hence it has a limit in $L^1$. No Dirac mass may appear in finite time.
	In $\mathbb R^n$ we do not have an equivalent guarantee that $S(t) \rho_{0,k} \in L^1 (\mathbb R^n)$ for some approximating sequence. We will, however, have this information in the space $\mathcal M (\Rd)$.
\end{remark}

\begin{remark}
	Notice that the construction of $S(t)$ is unique, since for dense data it produces the unique $H^1$ solution bounded below (which also comes as the limit of the approximations), and then it is extended into $L^1$ by uniform continuity.
\end{remark}

\begin{proof}[Proof of \Cref{thm:existence bounded L1}]
	We start by defining $S(t) \rho_{0} = \rho$ for the solutions constructed in \Cref{thm:existence bounded domain regular data}. Let us construct the rest of the situations.

	\textbf{Step 1. $0 < \ee \le \rho_0 \le \ee^{-1}$ but not necessarily in $H^1$.} We regularise $\rho_0$ by any procedure such that $H^1 (\Br) \ni \rho_{0,\ell} \to \rho_0$ in $L^{1+\ee}$ and a.e..
	Hence $0 < \ee \le \rho_{0,\ell} \le \ee^{-1}$ for $\ell$ large enough.
	By using stationary solutions and \eqref{eq:regularisation Lp estimates} we have that $0 < \omega(\ee) \le \rho_{\ell} \le C(t)$.
	By the $L^1$ contraction, for all $t > 0$, $S(t) \rho_{0,\ell}$ is a Cauchy sequence, and hence it has a unique $L^1$ limit.
	Let
	\begin{equation*}
		S(t) \rho_0 = L^1-\lim_\ell S(t) \rho_{0,\ell}.
	\end{equation*}
	We have
	\begin{equation*}
		\int_\Br |S(t) \rho_{0,\ell} - S(t) \rho_0| \le \int_\Br |\rho_{0,\ell} - \rho_0|, \qquad \ell > \ell_0.
	\end{equation*}
	For this subsequence $\rho_\ell^m$ converge to $\rho^m$ a.e. and, up to a further subsequence, in $L^\infty$-weak-$\star$, and hence $S(t) \rho_0$ is a weak $L^1$ solution.
	
	Taking a different $\overline \rho_0$ with the same properties, and $\overline \rho_{0,\ell}$ its corresponding approximation, again for $\ell$ large, $0 <  \omega(\ee) \le \overline \rho \le C(t)$.
	Then we have that
	\begin{equation*}
		\int_\Br |S(t) \rho_{0,\ell} - S(t) \overline  \rho_{0,\ell} | \le \int_\Br |\rho_{0,\ell} - \overline  \rho_{0,\ell}|, \qquad \ell > \ell_0.
	\end{equation*}
	Let $\ell \to +\infty$ we recover the $L^1$ contraction. Similarly for the comparison principle.
	
	\textbf{Step 2. $\rho_0 \in L^{1}$. Approximation by solutions of \Cref{thm:existence bounded domain regular data}.} We define
	\begin{equation*}
		\rho_{0,K} =\max \{ \rho, K \}, \qquad  \rho_{0, K, \ee} = \max \{ \rho, K \} + \ee.
	\end{equation*}
	For the solutions constructed in Step 1. we have that $\rho_{K, \ee} \searrow \rho_K$ as $\ee \searrow 0$ and as $K \nearrow +\infty$ we have $\rho_K \nearrow \rho$. By the $L^1$ contraction, we have as above that the sequence are Cauchy and hence we have $L^1$ convergence at each stage.
	The contraction and comparison are proven as in Step 1.
	
	\textbf{Step 3. \Cref{it:existence bounded L1 item 3}.} Due to the $L^{1+\ee}$ bound, we know that $u_k \rightharpoonup \rho^*$ weakly in $L^{1+\ee} ( (0,T) \times \Br)$. On the other hand, we can select adequate regularisations of the initial datum $\rho_{0,\ell}\in H^1 $ such that $\ee \le \rho_{0,\ell} \le \ee^{-1}$, and the corresponding solutions $u_{k, \ell}$ of \eqref{eq:main regularised bounded} with $\Phi = \Phi_k$ given by \eqref{eq:Phik} satisfy the $L^1$ contraction. Integrating in $(0,T)$ we have that
	\begin{equation*}
		\int_0^T \int_\Br |u_k - u_{k,\ell}| \le T \int_\Br |\rho_0 - \rho_{0,\ell}|.
	\end{equation*}
	As $k \to \infty$, by the lower semi-continuity of the norm
	\begin{equation*}
		\int_0^T \int_\Br |\rho^* - S(t) \rho_{0,\ell}| \le T \int_\Br |\rho_0 - \rho_{0,\ell}|.
	\end{equation*}
	As $\ell  \to \infty$ we recover $\rho^* = S(t) \rho_{0}$.
	
	\textbf{Step 4. $\rho_0 \in L^{1}$. Solutions in the very weak sense.} Finally, let us show that the solutions satisfy the equation in the very weak sense. Since we can integrate by parts, $\rho_{K,\ell}$ satisfies the very weak formulation, and we can pass to the limit to show that so does $\rho_K$.
	
	We have shown that $\rho_K \nearrow \rho$ in $L^1$. With the same philosophy, we prove that $\rho_K (t) \nearrow \rho(t)$ for every $t > 0$ so $\rho(t) \in L^1 (\Br)$ for a.e. and we can pass to the limit in the weak formulation. We only need to the deal with the diffusion term. We also have that $\rho_K^m \nearrow \rho^m$. Due to \eqref{eq:L1 controls Lm over compacts} and the Monotone Convergence Theorem, we deduce that $\rho^m \in L^1 ((0,T) \times \Br)$.
		
	\textbf{Step 5. Conservation of mass.} Since all the limits above hold in $L^1$, then preservation of the $L^1$ mass follows from the properties proved in \Cref{thm:existence bounded domain regular data}.
	
	\textbf{Step 6. Decay of the free energy.} Since all the limits above are taken monotonously and a.e., we can pass to the limit in
	\begin{equation*}
		\int_\Br \rho^m, \qquad \int_\Br V \rho
	\end{equation*}
	by the Monotone Convergence Theorem. Hence, the decay of the free energy proven in \Cref{cor:bounded domain and data properties of free energy} extends to $L^1$ solutions. We can also pass to the limit in \eqref{eq:bounded estimate nabla rhom in L1}.
\end{proof}

\section{An equation for the mass}
\label{sec:mass}
The aim of this section is to develop a well-posedness theory for the mass equation \eqref{eq:mass}.
We will show that the natural notion of solution in this setting is the notion of viscosity solution.
We will take advantage of the construction of the solution $\rho$ of \eqref{eq:main bounded domain} as the limit of the regularised problems \eqref{eq:main regularised bounded}.

\subsection{Mass equation for the regularised problem}

If $E$ is radially symmetric and $u$ is the solution solution of \eqref{eq:main regularised bounded}, its mass function $M$ satisfies
\begin{align*}
	\frac{\partial M}{\partial t} &=  \kappa (v)^2 \frac{\partial }{\partial v} \Phi \left(  \frac{\partial M }{\partial v}  \right) +  \kappa(v) \frac{\partial M }{\partial v}   E (v), \qquad \kappa(v) = n \omega_n^{\frac 1 n } v ^{\frac{n-1} n},
\end{align*}
by integrating the equation for $u = \frac{\partial M}{\partial v}$.
Notice that when $E = \nabla V$ then $E = \kappa(v) \frac{\partial V}{\partial v}$.
This change of variables guarantees that
\begin{align*}
	\int_{B_R} f(t,x) \diff x = |\partial B_1| \int_0^{R}
		 f(t,r)  r^{n-1} \diff r = \frac{|\partial B_1|}{|B_1|} n \int_0^{R_v} f(t,v) \diff v = \int_0^{R_v} f(t,v) \diff v,
	\end{align*}
	for radially symmetric functions.
	
\begin{theorem}[Comparison principle for masses]
	\label{thm:comparison classical solutions}
	Let $M_1$ and $M_2$ be two classical solutions of the mass problem such that $M_1(0,r) \le M_2 (0,r)$.
	Then
	$
		M_1 \le M_2.
	$
\end{theorem}

\begin{proof}
	For any $\lambda >0$, let us consider the continuous  function
	\begin{equation*}
		w (t,v) = e^{-\lambda t} (M_1 (t,v)- M_2(t,v)).
	\end{equation*}
	Notice that $w \to 0$ as either $t \to +\infty$ or $v \to  0, R_v$. Assume, towards a contradiction that $w$ reaches positive values. Hence, it reaches a positive global maximum at some point $t_0 > 0$ and $v_0 \in (0,\infty)$. At this maximum
	\begin{align*}
		0 &= \frac{\partial w}{\partial t} (t_0, v_0) = e^{-\lambda t} \frac{\partial}{\partial t} (M_1 - M_2) - \lambda e^{-\lambda t} (M_1 - M_2) \\
		0& =\frac{\partial w}{\partial v} (t_0, v_0)= e^{-\lambda t} \frac{\partial}{\partial v} (M_1 - M_2) \\
		0& \ge \frac{\partial^2 w}{\partial v^2}(t_0, v_0) = e^{-\lambda t} \frac{\partial^2}{\partial v^2} (M_1 - M_2) .
	\end{align*}
	At $(t_0, v_0)$, we simply write the contradictory result
	\begin{align*}
		0 &< \lambda e^{\lambda t} w(t_0, v_0) = \lambda (M_1 - M_2 ) =	\frac{\partial}{\partial t} (M_1 - M_2)\\
		&= (n \omega_n^{\frac 1 n} v_0^{\frac{n-1} n })^2 \left\{ \Phi'\left( \frac{\partial M_1}{\partial v}\right)  \frac{\partial^2 M_1}{\partial v^2} +  \frac{\partial M_1}{\partial v} E \right\} \\
		&\qquad \qquad -(n \omega_n^{\frac 1 n} v_0^{\frac{n-1} n })^2 \left\{ \Phi' \left( \frac{\partial M_2}{\partial v}\right)   \frac{\partial^2 M_2}{\partial v^2} +  \frac{\partial M_2}{\partial v} E \right\} \\
		&=(n \omega_n^{\frac 1 n} v_0^{\frac{n-1} n })^2 \left\{ \Phi'\left( \frac{\partial M_1}{\partial v}\right)\left(   \frac{\partial^2 M_1}{\partial v^2} - \frac{\partial^2 M_2}{\partial v^2} \right) \right\}
		\le 0. \qedhere
	\end{align*}
\end{proof}
	
Let us define the Hölder semi-norm for $\alpha \in (0,1)$
\begin{equation*}
	[f ]_{C^\alpha([a,b])} = \sup_{ \substack{ x,y \in [a,b] \\ x \ne y } } \frac{|f(x)-f(y)|}{|x-y|^\alpha}.
\end{equation*}
We have the following estimate
\begin{lemma}[Spatial regularity of the mass]
	\label{lem:M Calpha in space}
	If $u(t, \cdot) \in L^q (B_R)$ for some $q \in [1,\infty)$ then
	\begin{equation}
	\label{eq:bounded regularised M Calpha space}
		[M (t, \cdot) ]_{C^{\frac {q-1}q} ([0,R_v])} \le \| u \|_{L^q(\Br)}.
	\end{equation}
	If $q = \infty$ the same holds in $W^{1,\infty} (0,R_v)$.
\end{lemma}
\begin{proof}
For $v_1 \ge v_2$ we have
\begin{equation*}
	|M(t,v_1) - M(t, v_2)| = \int_{\widetilde B_{v_1} \setminus \widetilde B_{{v_2}}} u(t,x) \diff x \le \| u \|_{L^q} |\widetilde B_{v_1} \setminus \widetilde B_{v_2}|^{\frac {q-1}q} = \| u \|_{L^q} ( v_1 - v_2 )^{\frac{q-1}{q}}.\qedhere
\end{equation*}
\end{proof}
	
\begin{lemma}[Temporal regularity of the mass]
	There exists a constant $C > 0$, independent of $u$ or $\Phi$, such that
	\begin{equation}
	\label{eq:bound Mt in L2}
		 \int_0^T \int_{0}^{R_v} | M_t |^2 \diff v \diff t \le  C \left(    \int_{\Br} \Psi(u_0) + \| E \|_{L^\infty}^2 \int_0^T \int_{B_R} u(t,x)^2 \diff x \diff t  \right) .
	\end{equation}
	In particular, if $u_0 \in L^2$ and $\Psi(u_0) \in L^1$ then $M \in C^{\frac 1 2} (0,T; L^1 (0,R_v))$.
\end{lemma}
\begin{proof}

	Let us prove first an estimate for $\| M_t (t, \cdot) \|_{L^2 (0,R_v)}$.
	Since $M = \frac{\partial u}{\partial v}$ then $\frac{\partial M}{\partial t} = \frac{\partial u}{\partial v\partial t}$. Applying Jensen's inequality
	\begin{align*}
		\int_{0}^{R_v} \left| \frac{\partial M}{\partial t} (t,v) \right|^2 \diff v &= \int_0^{R_v} \left( \int_{\widetilde B_{v}} \frac{\partial u}{\partial t} \diff x \right)^2 \diff v  \\
		&= \int_0^{R_v} \left( \int_{\widetilde B_{v}} \diver ( \nabla \Phi(u) + u E ) \diff x \right)^2 \diff v  \\
		&= \int_0^{R_v} \left( \int_{\partial \widetilde B_{v}} (\nabla \Phi(u) + u E ) \cdot \frac x {|x|} \diff S_x \right)^2 \diff v  \\
		&\le \int_0^{R_v}  \int_{\partial \widetilde B_{v}} \left| \nabla \Phi(u) + u E \right|^2 \diff S_x \diff v  .
	\end{align*}	
	Making the change of variables $v = |B_1| r^n$ we have $|\widetilde B_v| = v = |B_1| r^n =  |B_r|$ and
	\begin{equation*}
		\int_{0}^{R_v} \left| \frac{\partial M}{\partial t} (t,v) \right|^2 \diff v \le \int_0^{R}  \int_{\partial B_r} \left| \nabla \Phi(u) + u E \right|^2 \diff S_x |B_1| n r^{n-1} \diff r  = \| \nabla \Phi(u) + u E \|_{L^2 (\Br)} ^2.
	\end{equation*}	
	Due to \eqref{eq:smooth estimate nabla Phi u} we recover \eqref{eq:bound Mt in L2}. Finally
	\begin{align*}
		\| M(t_1) - M(t_2) \|_{L^1 (0,R_v)} & =\int_0^{R_v} |M(t_2, v) - M(t_1, v)| \diff v= \int_0^{R_v} \left |\int_{t_1}^{t_2} \frac{\partial M} { \partial t } (s,v) \diff s \right| \diff v \\
		&\le  \int_0^{R_v} \int_{t_1}^{t_2} \left | \frac{\partial M} { \partial t } (s,v)\right| \diff s  \diff v \le |t_2-t_1|^{\frac 1 2} \left \| \frac{\partial M} { \partial t } \right\|_{L^2 ((0,T) \times (0,R_v))}. \qedhere
	\end{align*}
	
\end{proof}

\subsection{Aggregation-Fast Diffusion}

We recall the definition of viscosity solution for the $p$-Laplace problem, which deals with the singular ($p \in (1,2)$) and degenerate ($p > 2$) cases.
	We recall the definition found in many texts (see, e.g.,
	\cite{Juutinen2003,Medina2019}
	and the references therein).
\begin{definition}
	For $p > 1$ a function $u$ is a viscosity supersolution of $-\Delta_{p} u = f (x, u, \nabla u)$ if, $u \not \equiv \infty$, and for every $\varphi \in C^2 (\Omega)$ such that $u \ge \varphi$, $u (x_0) = \varphi (x_0)$ and $\nabla \varphi(x) \ne 0$ for all $x \ne x_0$ it holds that
	\begin{equation*}
		\lim_{r \to 0} \sup_{x \in B_r (x_0) \setminus \{x_0\}} (-\Delta_p \varphi(x)) \ge f(x_0, u(x_0), \nabla \varphi (x_0)).	
	\end{equation*}
\end{definition}
Similarly, for our problem we define
\begin{definition}
	For $m \in (0,1)$ a function $u$ is a viscosity supersolution of \eqref{eq:mass} if, for every $t_0 > 0, v_0 \in (0,R_v)$ and for every $\varphi \in C^2 ((t_0 - \ee, t_0 + \ee) \times (v_0 - \ee, v_0 + \ee))$ such that $M \ge \phi$, $M (v_0) = \varphi (v_0)$ and $\frac{ \partial \varphi}{\partial v} (v) \ne 0$ for all $v \ne v_0$ it holds that
	\begin{align}
		\frac{\partial \varphi}{\partial t} (t_0, v_0 ) - (n \omega^{\frac 1 n } v_0 ^{\frac{n-1} n})^2   \left[ \lim_{r \to 0} \sup_{0 < |v-v_0| < r} \frac{\partial }{\partial v} \left[  \left(  \frac{\partial \varphi }{\partial v}  \right)^m \right] + \frac{\partial \varphi }{\partial v}(t_0,v_0) \frac{\partial V}{\partial v} (v_0)  \right] \ge 0.
	\end{align}	
	The corresponding definition of subsolution is made by inverting the inequalities. A viscosity solution is a function that is a viscosity sub and supersolution.
\end{definition}

\begin{remark}
	Since we have a one dimensional problem, we can write the viscosity formulation equivalently by multiplying by $(  \frac{\partial \varphi }{\partial v} )^{1-m}$ everywhere, to write the problem in degenerate rather than singular form.
\end{remark}

\begin{remark}
	Our functions $M$ will be increasing in $v$. This allows to a simplification of the condition in some cases. For example, if also have a lower bound on $\rho$, in the sense that
	\begin{equation*}
		M(t, v_2) - M(t, v_1) \ge c (v_2 - v_1), \qquad \forall  v_0 + \ee \ge v_2 \ge v_1 \ge v_0 - \ee \text{ where } c > 0
	\end{equation*}
	then we know that it suffices to take viscosity test functions $\varphi$ such $\frac{\partial \varphi}{\partial v} \ge \frac c 2$. In particular, we can simplify the definition of sub and super-solution by removing the limit and the supremum.
\end{remark}
\begin{remark}
	We can define the upper jet as
	\begin{align*}
		\mathcal J^{2,+} M(t_0,v_0) = \Big\{ & (D\varphi(t_0,v_0), D^2\varphi(t_0,v_0)) \\
		&\qquad : \varphi \in C^2 ((t_0 - \ee,t_0 + \ee)\times(v_0 - \ee, v_0 + \ee)), \\
		&\qquad \qquad M(t,v) - \varphi(t,v) \le 0 = M(t_0,v_0) - \varphi(t_0,v_0) \Big\}.
	\end{align*}
	The elements of the upper jet are usually denoted by $(p,X)$. The lower jet $\mathcal J^{2,-}$ is constructed by changing the inequality above. The definition of viscosity subsolution (resp. super-) can be written in terms of the upper jet (resp. lower).
\end{remark}

\begin{theorem}[Existence from the semigroup theory for $\rho$]
	\label{thm:comparison principle for mass}
	Let $\rho_0 \in L^1 (\Br)$.
	Then
	\begin{equation*}
		M(t,v) = \int_{\widetilde B_v} S(t)[\rho_0] (x) \diff x
	\end{equation*}
	is a viscosity solution of \eqref{eq:mass} with $M(t,0) = 0$ and $M(t,R_v) = \| \rho_0 \|_{L^1 (\Br)}$.
	Furthermore, for any
	$v_1,v_2,T > 0$,
	$M \in C([0,T] \times [v_1,v_2])$ with a modulus of continuity that depends only on $n, m, v_1, v_2$, $T, \| \frac{\partial V}{\partial v} \|_{L^\infty(v_1,v_2)}$ and the modulus of continuity of $M_{\rho_0}$ in $[v_1,v_2]$ .
	Moreover, we have the following interior regularity estimate:
	for any $T_1 > 0$ and $0 < v_1 < v_2 < R_v$
	there exists $\gamma > 0$ and $\alpha \in (0,1)$ depending only on $n,m,\| \frac{\partial V}{\partial v} \|_{L^\infty(v_1,v_2)},v_1,v_2, T_1$, such that
	\begin{equation}
	\label{eq:local Calpha regularity}
			|M(t_1,v_1) - M(t_2,v_2) | \le \gamma  \left( \frac{|v_1-v_2| + \| \rho_0 \|_{L^1 (\Br)}^{\frac{m-1}{m+1}} |t_1 - t_2|^{\frac 1 {m+1}} }{ \min\{v_1, R_v-v_2 \} + \| \rho_0 \|_{L^1 (\Br)}^{\frac{m-1}{m+1}} T_1^{\frac 1 {m+1}}}   \right)^\alpha,
	\end{equation}	
	for all $(t_i, v_i) \in [T_1, +\infty) \times [v_1,v_2]$.
\end{theorem}
\begin{proof}
	\textbf{Step 1. $\ee \le \rho_0 \le \ee^{-1}$ and $\rho_0 \in H^1 (\Br)$.}
	Let us show that
	$$
		M_{u_k} \to M_\rho \qquad  \text{uniformly in }  [0,T] \times B_R.
	$$
	$M_\rho$ is a viscosity solution of \eqref{eq:mass} and $M_\rho$ is a weak local solution in the sense of \Cref{sec:classical regularity}.
	
	By our construction of $\rho$ by regularised problems in \Cref{thm:existence bounded domain regular data}, the strong $L^q$ convergence of $u_k$ to $\rho$ ensures that
	\begin{equation*}
		\int_0^T \sup_{v \in [0,R_v]}  |M_{u_k}(t,v) - M_\rho(t, v)|  \diff t \le \int_0^T \int_{\Rd} |u_k (t,x) - \rho(t,x)|\diff x \diff t \to 0.
	\end{equation*}
	So we know $M_{u_k} \to M_\rho$ in $L^1 (0,T; L^\infty (0,R_v))$, and hence (up a to a subsequence) a.e.
	
	Through estimates \eqref{eq:bounded regularised M Calpha space}, \eqref{eq:bound Mt in L2}
	, and \Cref{thm:relation between space and time regularities} we have
	\begin{equation}
	\label{eq:bound Mk Calpha}
		|M_{u_k}(t_1,v_1) - M_{u_k}(t_2,v_2)| \le C (\ee) (|v_1-v_2|^\alpha + |t-s|^\gamma), \qquad t,s \in [0,T], v_1,v_2 \in [\ee, R_v - \ee].
	\end{equation}
	To check that $M_\rho$ is a viscosity solution, we select $v_0 \in (0,R_v)$. Taking a suitable interval $(\ee , R_v - \ee ) \ni v_0$, by the Ascoli-Arzelá theorem, a further subsequence is uniformly convergent. Since we have characterised the a.e. limit we have
	\begin{equation*}
		\| M_{u_k} - M_\rho \|_{L^\infty ([0,T] \times [\ee, R_v - \ee]}) \to 0.
	\end{equation*}
	Due to the uniform convergence, we can pass to the limit in the sense of viscosity solutions and $M_\rho$ is a viscosity solution at $x_0$.
	
	The argument is classical and goes as follows (see \cite{crandall+ishii+lions1992users-guide-viscosity}). Take a viscosity test function $\varphi$ touching $M_\rho$ from above at $x_0$. Then, due to the uniform convergence $M_{u_k}$ to $M_\rho$ in a neighbourhood of $x_0$, there exists points $x_k$ where $\varphi$ touches $M_{u_k}$ from above. We apply the definition of viscosity solution for $M_{u_k}$ at $x_k$, and pass to the limit.

	Due to the pointwise convergence, $M_\rho$ also satisfies \eqref{eq:bound Mk Calpha}.

	\textbf{Step 2. $\rho_0 \in L^{1}$.}
	We pick the approximating sequence
	\begin{equation*}
		\rho_{0,K,\ee} = \max \{ \rho_0 ,K \} + \ee.
	\end{equation*}
	As we did in \Cref{thm:existence bounded L1} the $L^1$ limit of the corresponding solutions is $S(t) \rho_0$. Furthermore, the limits $\ee \searrow 0$ and $K \nearrow +\infty$ are taking monotonically in $\rho$, so also monotically in $M$. This guarantees monotone convergence in $M$. With the universal upper bound $1$ we have $L^1$ convergence.
	
	Since the $C^\alpha$ bound is uniform away from $0$, we know that $M$ maintains it and is continuous.
	Due to Dini's theorem the convergence is uniform over $[0,T] \times [\ee, R_v - \ee]$, and $M_{\rho}$ is a viscosity solution of the problem.
	
	The value $M(t,0) = 0$ is given by $S(t) \rho_0 \in L^1(\Br)$ and the value at $M(t,R_v) = a_{0,R}$ by
	the fact that $\| S(t) \rho_0 \|_{L^1(\Br)} = \| \rho_0 \|_{L^1 (\Br)} = a_{0,R}$.
	The uniform continuity is a direct application of \Cref{cor:uniform regularity interior in x up to 0 in t}. We point out that, since $\rho_0 \in L^1 (\Br)$, then $M_{\rho_0}$ is point-wise continuous, and therefore uniformly continuous over compact sets. Estimate \eqref{eq:local Calpha regularity} follows from \Cref{thm:regularity dibenedetto}.
\end{proof}

Let us now state a comparison principle, under simplifying hypothesis.

\begin{theorem}[Comparison principle of viscosity solutions if $\rho$ is bounded below]
	\label{thm:comparison principle m}
	Let $\underline M$ and $\overline M$ be uniformly continuous sub and supersolution.
	Assume, furthermore, that there exists $C_0 > 0$ such that
	\begin{equation*}
		\underline M(t,v_2) - \underline M(t, v_1) \ge C_0 (v_2 - v_1), \qquad \forall v_2 \ge v_1.
	\end{equation*}
	Then, the solutions are ordered, i.e. $\underline M \le \overline M$.
\end{theorem}

\begin{proof}
	Assume, towards a contradiction that
	\begin{equation*}
		\sup_{t > 0, v \in [0,R_v]} (\underline M(t,v) - \overline M(t,v)) = \sigma > 0.
	\end{equation*}
	Since both functions are continuous, there exists $(t_1,v_1)$ such that $\underline M(t_1,v_1) - \overline  M(t_1, v_1) > \frac {3\sigma}4$. Clearly, $t_1 , v_1 > 0$. Let us take $\lambda$ positive such that
	\begin{equation*}
		 \lambda < \frac{\sigma}{16(t_1 + 1)}.
	\end{equation*}
	With this choice, we have that
	\begin{equation*}
		2 \lambda t_1 < \frac{\sigma}{4}.
	\end{equation*}
	For this $\varepsilon$ and $\lambda$ fixed, let us construct the variable-doubling function defined as
	\begin{equation*}
		\Phi(t,s,v,\xi) = \underline M(t,v) - \overline M(s,\xi) - \frac{|v-\xi|^2 + |s-t|^2}{\varepsilon^2} -  \lambda (s + t).
	\end{equation*}
	This function is continuous and bounded above, so it achieves a maximum at some point. Let us name this maximum depending on $\ee$, but not on $\lambda$ by
	\begin{equation*}
		\Phi (t_\ee, s_\ee, v_\ee, \xi_\ee) \ge \Phi  (t_1,t_1,v_1,v_1) > \frac{3\sigma}{4}  - 2 \lambda t_1 > \frac{\sigma}{2}.
	\end{equation*}
	In particular, it holds that
	\begin{equation}
		\label{eq:estimate m - M at max of Phi ee}
		\underline M(t_\ee, v_\ee) - \overline M(s_\ee, \xi_\ee)  \ge \Phi(t_\ee,s_\ee, v_\ee, \xi_\ee) > \frac{\sigma}{2} .
	\end{equation}
	
	\noindent \textbf{Step 1. Variables collapse.} As $\Phi(t_\ee, s_\ee, v_\ee, \xi_\ee) \ge \Phi(0,0,0,0)$, we have
	\begin{equation*}
		\frac{|v_\ee-\xi_\ee|^2 + |s_\ee-t_\ee|^2}{\varepsilon^2}  + \lambda (s_\ee + t_\ee) \le \underline M(t_\ee, v_\ee) - \overline M(s_\ee, \xi_\ee) - \Phi(0,0,0,0) \le C.
	\end{equation*}
	Therefore, we obtain
	\begin{equation*}
		|v_\ee - \xi_\ee| + |t_\ee - s_\ee| \le C \varepsilon.
	\end{equation*}
	This implies that, as $\varepsilon \to 0$, the variable doubling collapses to a single point. \\
	
	We can improve the first estimate using that $\Phi(t_\ee, s_\ee, v_\ee, \xi_\ee) \ge \Phi(t_\ee, t_\ee, v_\ee, v_\ee)$. This gives us
	\begin{align*}
		\frac{|v_\ee-\xi_\ee|^2 + |s_\ee-t_\ee|^2}{\varepsilon^2} &\le \overline M(t_\ee, v_\ee) - \overline  M(s_\ee, \xi_\ee)  + \lambda(t_\ee - s_\ee) \\
		&\le \overline M(t_\ee, v_\ee) - \overline M(s_\ee, \xi_\ee) +  C \ee.
	\end{align*}
	Since $\overline M$ is uniformly continuous, we have that
	\begin{equation}
	\label{eq:comparison principle mass collapse rate}
		\lim_{\ee \to 0}\frac{|v_\ee-\xi_\ee|^2 + |s_\ee-t_\ee|^2}{\varepsilon^2} = 0.
	\end{equation}
	
	\noindent \textbf{Step 2. For $\ee > 0$ sufficiently small, the points are interior.} We show that there exists $\mu$ such that $t_\ee, s_\ee \ge \mu > 0$ for $\varepsilon>0$ small enough.  For this, since $\underline M$ and $\overline M$ are uniformly continuous we can estimate as
	\begin{align*}
		\frac{\sigma}{2} &<  \underline M(t_\ee, v_\ee) - \overline M(s_\ee, \xi_\ee) \\
		&=\underline M(t_\ee, v_\ee) - \underline M(0, v_\ee) + \underline M(0, v_\ee) - \overline M(0,v_\ee)  + \overline M(0,v_\ee) - \overline M(t_\ee, v_\ee) + \overline M(t_\ee, v_\ee) - \overline M(s_\ee, \xi_\ee) \\
		&\le \omega(t_\ee) + \omega( |v_\ee - \xi_\ee| + |t_\ee -s_\ee| ),
	\end{align*}
	where $\omega \ge 0$ is a modulus of continuity (the minimum of the moduli of continuity of $\underline M$ and $\overline M$),
	i.e.\ a continuous non-decreasing function such that $\lim_{r \to 0} \omega(r) = 0$. For $\varepsilon > 0$ such that
	\begin{equation*}
		\omega( |v_\ee - \xi_\ee| + |t_\ee -s_\ee| ) < \frac{\sigma}{4},
	\end{equation*}
	we have $\omega(t_\ee) > \frac{\sigma}{4}$. The reasoning is analogous for $s_\varepsilon$.
	For $v_\ee$ we can proceed much in the same manner
	\begin{align*}
		\frac{\sigma}{2} & < \underline M(t_\ee, v_\ee) - \overline M(s_\ee, \xi_\ee) \\
		&=\underline M(t_\ee, v_\ee) - \underline M(t_\ee, 0) + \underline M(t_\ee,0) - \overline M(t_\ee, v_\ee)   + \overline M(t_\ee, v_\ee) - \overline M(s_\ee,\xi_\ee) \\
		&\le \omega(v_\ee) + \omega( |v_\ee - \xi_\ee| + |t_\ee -s_\ee| ).
	\end{align*}
	And analogously for $\xi_\ee$. A similar argument holds for $R_v - v_\ee$ and $R_v - \xi_\ee$.
	
	\noindent \textbf{Step 3. Choosing viscosity test functions.}
	Unlike in the case of first order equations, there is no simple choice of $\varphi$ that works in the viscosity formula. We have to take a detailed look at the jet sets.
	Due to \cite[Theorem 3.2]{crandall+ishii+lions1992users-guide-viscosity} applied to $u_1 = \underline M$, $u_2 = - \overline M$  and
	\begin{equation*}
		\varphi_\ee (t,s,v,\xi)  = \frac{|v-\xi|^2 + |s-t|^2}{\varepsilon^2} +  \lambda (s + t)
	\end{equation*}
	for any $\delta > 0$,
 	there exists $\underline X$ and $\overline X$ in the corresponding jets such that
 	\begin{equation*}
		\left(\frac{\partial \varphi_\ee}{\partial (t,v)} (z_\ee) , \underline X \right) \in \mathcal J^{2,+} \underline M (t_\ee, v_\ee), \qquad  \left( - \frac{\partial \varphi_\ee}{\partial (s,\xi)} (z_\ee) , - \overline X \right) \in \mathcal J^{2,-} \overline M (s_\ee, \xi_\ee),
	\end{equation*}
	where $z_\ee = (t_\ee, s_\ee, v_\ee, \xi_\ee)$ and we have
	\begin{equation*}
		-( \delta^{-1} + \| A \| ) I \le
		\begin{pmatrix}
			\underline X \\
			& -\overline X
		\end{pmatrix}
		\le A + \delta A^2
	\end{equation*}
	where $A = D^2 \varphi_\ee (z_\ee)$. In particular, this implies that the term of second spatial derivatives satisfies $\underline X_{22} \le \overline X_{22}$ (see \cite{crandall+ishii+lions1992users-guide-viscosity}).
	Notice that
	\begin{equation*}
		\frac{\partial \varphi}{\partial t} (z_\ee) = \frac{2(t_\ee - s_\ee)}{\ee^2} + \lambda, \qquad - \frac{\partial \varphi}{\partial s} (z_\ee) =  \frac{2(t_\ee - s_\ee)}{\ee^2} - \lambda
	\end{equation*}
	and
	\begin{equation*}
		\frac{\partial \varphi}{\partial v} (z_\ee) = \frac{2(v_\ee - \xi_\ee)}{\ee^2} = - \frac{\partial \varphi}{\partial \xi} (z_\ee).
	\end{equation*}
	Since $\underline M (t_\ee, v) - \Phi(t_\ee, s_\ee, v, \xi_\ee)$ as a maximum at $ v = \xi_\ee$ we have that, for $v > v_\ee$
	\begin{equation*}
		\frac{|v - \xi_\ee|^2 - |v_\ee - \xi_\ee|^2  }{\ee^2} \ge \underline M (t_\ee, v) - \underline M(t_\ee, v_\ee) \ge C_0 (v - v_\ee).
	\end{equation*}
	Therefore, we conclude
	\begin{equation*}
		\frac{2(v_\ee - \xi_\ee)}{\ee^2} \ge C_0.
	\end{equation*}
	
	Plugging everything back into the notion of viscosity sub and super-solution
	\begin{gather*}
		\frac{2(t_\ee - s_\ee)}{\ee^2} + \lambda + H \left( v_\ee, \frac{2(v_\ee - \xi_\ee)}{\ee^2} , \underline X  \right)\le 0 \\
		\frac{2(t_\ee - s_\ee)}{\ee^2} - \lambda + H \left( \xi_\ee, \frac{2(v_\ee - \xi_\ee)}{\ee^2} , \overline X  \right)\ge 0
	\end{gather*}
	where
	\begin{equation*}
		H(v,p,X) = -(n \omega_n^{\frac 1n} v^{\frac{n-1}{n}})^2 \left\{ m p^{m-1} X_{22} + p \frac{\partial V}{\partial v}(v) \right\} .
	\end{equation*}
	
	\noindent \textbf{Step 4. A contradiction.}
	Substracting these two equations
	\begin{align*}
		0 < 2 \lambda &\le H \left( \xi_\ee, \frac{2(v_\ee - \xi_\ee)}{\ee^2} , \overline X  \right) - H \left( v_\ee, \frac{2(v_\ee - \xi_\ee)}{\ee^2} , \underline  X  \right) \\
			&= H \left( \xi_\ee, \frac{2(v_\ee - \xi_\ee)}{\ee^2} , \overline X  \right) - H \left( \xi_\ee, \frac{2(v_\ee - \xi_\ee)}{\ee^2} , \underline  X  \right) \\
			&\qquad + H \left( \xi_\ee, \frac{2(v_\ee - \xi_\ee)}{\ee^2} , \underline  X  \right) - H \left( v_\ee, \frac{2(v_\ee - \xi_\ee)}{\ee^2} , \underline  X  \right) \\
			&\le H \left( \xi_\ee, \frac{2(v_\ee - \xi_\ee)}{\ee^2} , \underline  X  \right) - H \left( v_\ee, \frac{2(v_\ee - \xi_\ee)}{\ee^2} , \underline  X  \right)\\
			&=(n \omega_n^{\frac 1n})^2 \frac{2(v_\ee - \xi_\ee)}{\ee^2} \left(v_\ee^{2\frac{n-1}{n}}\frac{\partial V}{\partial v}(v_\ee) - \xi_\ee^{2\frac{n-1}{n}}\frac{\partial V}{\partial v}(\xi_\ee)  \right) \to 0,
	\end{align*}
	since $v^{2\frac{n-1}{n}}\frac{\partial V}{\partial v}(v) = r^{n-1} \frac{\partial V}{\partial r}$ is Lipschitz continuous and \eqref{eq:comparison principle mass collapse rate}.
\end{proof}

\section{Existence of concentrating solutions}
\label{sec:concentration}
When we now take $F: (0,\infty) \to (0,\infty)$
\begin{equation}
	\label{eq:initial data with concentration}
	\rho_F (x) = \Big(  \tfrac{1-m}{m} F(  V (x) ) \Big )^{- \frac 1 {1-m}} , \qquad  F' \le 1, \qquad F(0) = 0 \qquad F(s) > 0 \text{ for all } s > 0,
\end{equation}
we have that $\rho_F \ge \rho_V$, so the corresponding solutions with $\rho(0,x) = \rho_F(x)$ satisfies
\[
	\rho (t,x) \ge \rho_V (x) , \qquad \forall t \ge 0 , x \in \Br.
\]
We will prove that with this initial data we have
$	U = \frac{\partial M}{\partial t} \ge 0
$
by showing it satisfies a PDE with a comparison principle and $U(0,\cdot) \ge 0$. First, we prove an auxiliary result for the regularised problem.

\begin{theorem}[Solutions of \eqref{eq:main regularised bounded} with increasing mass]
	\label{thm:solutions with increasing mass}
	Let $h \in \mathbb R$, $F$ be such that
	$F' \le 1$, $F \ge 0$, $F(0) = 0, $
	\begin{equation}
	\label{eq:increasing solutions Phi smooth}
		u_0 = \Theta^{-1} \Big ( h -  F (V(x)) \Big ),
	\end{equation}
	$u$ be the solution of \eqref{eq:main regularised bounded} and $M$ be its mass.
	Then, we have that
	\begin{equation}
		\label{eq:mass increasing in time}
		M(t + h, x) \ge M(t, x) , \qquad \forall h \ge 0
	\end{equation}
	and
	\begin{equation}
	\label{eq:lower bound on dv M}
	M(t, v_2) - M(t, v_1) \ge \int_{\widetilde B_{v_2} \setminus \widetilde B_{v_1}} \Theta^{-1} (h - V(x)) \diff x , \qquad \forall v_1 \le v_2.
	\end{equation}
\end{theorem}
\begin{proof}
Notice also that $F(s) \le s$ so
$
	u_0 (x) \ge \Theta^{-1} ( h - V (x) ),
$
and this is a stationary solution. Hence this inequality holds for $u(t)$ as well, due to \Cref{thm:ball comparison Phi smooth}. Thus \eqref{eq:lower bound on dv M} holds.
Since $u \in C^1 ((0,T) ; C(\overline B_R))$, we can consider
\begin{equation*}
	U(t,v) = \int_{\widetilde B_v} \frac{\partial u}{\partial t} (t,x ) \diff x = \frac{\partial M}{\partial t}.
\end{equation*}
Due to \eqref{eq:bounded regularised regularity of $u$},
$
U \in C((0,T)  \times [0,R_v]  ) .
$
Since we have
$
	M(t,0) = 0, M(t,R_v) = 1
$
the boundary conditions are
$
	U(t,0) = U(t, R_v) = 0.
$
Using the equation for the mass we have that
\begin{equation*}
	U(0,r) =  (n \omega_n v^{ \frac{n-1}n })^2  {u_0}  \frac{\partial }{\partial v} \left( \Theta(u_0) + V   \right) \ge 0,
\end{equation*}
since $\partial V / \partial v \ge 0$, by hypothesis.
Taking formally a time derivative in the equation of the mass, we obtain that
\begin{align*}
	\frac{\partial U}{\partial t} &= (n \omega_n v^{ \frac{n-1}n })^2 \left(  \frac{\partial}{\partial v} \left( \Phi'(u) \frac{\partial M}{\partial v \partial t} \right) + \frac{\partial M}{\partial t \partial v } \frac{\partial V}{\partial v}  \right )= (n \omega_n v^{ \frac{n-1}n })^2 \left(  \frac{\partial}{\partial v} \left( \Phi'(u) \frac{\partial U}{\partial v} \right) + \frac{\partial U}{ \partial v } \frac{\partial V}{\partial v}  \right ) \\
	&= A(v) \frac{\partial^2 U}{\partial v^2} + B(v) \frac{\partial U}{\partial v}
\end{align*}
where $A(v) = (n \omega_n v^{ \frac{n-1}n })^2 \Phi'(u) \ge 0$ and $B(v) = (n \omega_n v^{ \frac{n-1}n })^2 ( \frac{\partial}{\partial v} [\Phi'(u)] + \frac{\partial V}{\partial v}).$
This can be justified in the weak local sense. For $\varphi \in C_c^\infty ((0,T) \times (0,R_v))$ we can write
\begin{equation*}
	-\int_0^T\int_0^{R_v} M \frac{\partial \varphi}{\partial t} = - \int_0^T\int_0^{R_v}  \Phi \left( u \right) \frac{\partial }{\partial v} \left( (n \omega_n v^{ \frac{n-1}n })^2 \varphi \right) + \int_0^T\int_0^{R_v} (n \omega_n v^{ \frac{n-1}n })^2 \frac{\partial M}{ \partial v } \frac{\partial V}{\partial v} \varphi ,
\end{equation*}
we can simply take $\varphi = \frac{\partial \psi}{\partial t}$ and integrating by parts in time to recover
\begin{equation*}
	\int_0^T\int_0^{R_v} \frac{\partial M}{\partial t} \frac{\partial \psi}{\partial t} =  \int_0^T\int_0^{R_v}  \frac{\partial}{\partial t} \left( \Phi \left( u \right) \right) \frac{\partial }{\partial v} \left( (n \omega_n v^{ \frac{n-1}n })^2 \psi \right) - \int_0^T\int_0^{R_v} (n \omega_n v^{ \frac{n-1}n })^2 \frac{\partial M}{\partial t \partial v } \frac{\partial V}{\partial v} \psi .
\end{equation*}
Since $u$ is $C^1$ then $ \frac{\partial}{\partial t} \left( \Phi \left( u \right) \right) = \Phi'(u) \frac{\partial U}{\partial t} $ is a continuous function. Operating with the derivatives of $\psi$, we recover that
\begin{equation*}
	\int_0^T \int_0^{R_v} U \left\{ \frac{\partial \psi}{\partial t} +\frac{\partial}{\partial v} \left(A(t,v) \frac{\partial \psi}{\partial v} \right) + \frac{\partial}{\partial v} \left(B(v) \psi \right) \right\} = 0, \qquad \forall \psi \in C^\infty_c((0,T) \times (0,R_v)) .
\end{equation*}
We now show that $U$ is a solution in the weak sense, incorporating the boundary conditions.
Since $U$ is continuous and $U(t,0) = U(t,R_v) = 0$, for any $\psi$ suitably regular we can use an approximating sequence $\psi_k \in C^\infty_c((0,T) \times (0,R_v))$ to show that
\begin{equation*}
	\int_0^{R_v} U(T) \psi (T) \diff v + \int_0^T \int_0^{R_v} U \left\{ \frac{\partial \psi}{\partial t} +\frac{\partial}{\partial v} \left(A(t,v) \frac{\partial \psi}{\partial v} \right) + \frac{\partial}{\partial v} \left(B(v) \psi \right) \right\} = \int_0^{R_v} U(0) \psi (0) \diff v.
\end{equation*}
Fix $\Psi_0$ smooth and let $\Psi$ the solution of
\begin{equation}
\label{eq:bounded smooth dual U}
	\begin{dcases}
		\frac{\partial \Psi}{\partial t} = \frac{\partial}{\partial v} \left(A(T-t,v) \frac{\partial \Psi}{\partial v} \right) + \frac{\partial}{\partial v} \left(B(v) \Psi \right)	& \text{in }(0,T) \times (0,R_v) \\
		\Psi(t,0) = \Psi(0,R_v) = 0, \\
		\Psi(0,v) = \Psi_0.
	\end{dcases}
\end{equation}
If $\Psi$ is a classical interior solution, then taking as a test function $\psi(t,x) = \Psi (T-t,x)$ we have that
\begin{equation*}
	\int_0^{R_v} U(T) \Psi_0 \diff v = \int_0^{R_v} U(0) \Psi (T) \diff v.
\end{equation*}
Notice that $A(0) = 0$. Substituting $A$ by the uniformly elliptic diffusion $A(T-t,v) + \delta$, $\delta > 0$, and letting $\delta \searrow 0$, for any $\Psi_0 \ge 0$, we can construct a non-negative solution of \eqref{eq:bounded smooth dual U}. Therefore, since $U(0) \ge 0$ we have that
$
	U \ge 0
$
in $(0,T) \times (0,R_v)$, and the proof is complete.
\end{proof}

Before we continue, we point out that $M_{\rho_F}$, lies between $M_{\rho_V}$ and one of its upward translations 
\begin{remark}
	Notice that $F(0) = 0$ and $F' \le 1$ then $F(s) \le s$. Hence $\rho_F \ge \rho_V$. Integrating forward from $0$ we have
	\begin{equation*}
		M_{\rho_F} (v) = \int_{\widetilde B_v} \rho_F \diff x \ge  \int_{\widetilde B_v} \rho_V \diff x = M_{\rho_V} (v).
	\end{equation*}
	On the other hand, integrating backwards from $R_v$ we have
	\begin{equation}
	\label{eq:bounded rho0 with concentration comparison masses}
	\begin{aligned}
		M_{\rho_F} (v) &= M_{\rho_F} (R_v) - \int_{B_R \setminus \widetilde B_v} \rho_F \diff x \le  M_{\rho_F} (R_v) - \int_{B_R \setminus \widetilde B_v} \rho_V \diff x \\
		&= \Big(  M_{\rho_F} (R_v) - a_{V,R} \Big) + M_{\rho_V} (v).
	\end{aligned}
		\end{equation}
\end{remark}

Now we move to considering suitable initial data for \eqref{eq:main bounded domain}. We make the following construction
\begin{lemma}
	\label{lem:rho_D}
	Let $b_1 \in (0,V(R))$. Assume that $a_{V,R} < a_{V,0}$. There exists $\overline b_{2}(b_1) < V(R)$ such that, for all $b_{2} \in (b_1,  \overline b_{2}(b_1))$ there exists $D(b_1,b_2) < b_1$ such that the $\rho_D$ given by \begin{equation}
	\label{eq:bounded rho0 with concentration}
	F_D(s) = \begin{cases}
	s & \text{if } s \in [0,b_1 ], \\
	D & \text{if } s \in [b_1  , b_2 ] , \\
	D + s - b_2 & \text{if } s \ge b_2.
	\end{cases} \qquad \rho_D (x) =  \Big(  \tfrac{1-m}{m} F_D(  V (x) ) \Big )^{- \frac 1 {1-m}}
	\end{equation}
	satisfies
	$
	\int_{B_R} \rho_D = a_{0,R}.
	$
\end{lemma}
The function $F_D$ can be taken as the limit of functions $F_\ee \in C^1$ in the assumptions of \Cref{thm:solutions with increasing mass}. Take $F_\ee(0) = 0$ and
\begin{equation*}
F_\ee'(s) = \begin{cases}
1 & \text{if } s \in [0,b_1 ], \\
1 - \frac{1+\alpha}{\ee / 4} (s-b_1) & \text{if } s \in [b_1, b_1 + \tfrac \ee 4] \\
-\alpha & \text{if } s \in [b_1 + \tfrac \ee 4 ,b_1 + \tfrac \ee  2 ] \\
-\alpha + \frac{\alpha}{\ee / 2} (s - b_1 - \tfrac \ee  2)  & \text{if } s \in [b_1 + \tfrac \ee 2 ,b_1 + \ee ] \\
0 & \text{if } s \in [b_1 + \ee , b_2 - \ee ] , \\
\frac 1 \ee (s-b_2+\ee)  & \text{if } s \in [b_2 - \ee , b_2 ] , \\
1 & \text{if } s \ge b_2.
\end{cases}
\end{equation*}
where given $0 < b_1 < b_2 < V(R)$ and $\ee < \frac{b_2-b_1}{4}$ and $0 < D \le b_1$, we can always select
\begin{equation*}
\alpha(b_1,b_2,D,\ee) > 0 \text{ such that } \qquad 	F_\ee(b_1 + \ee) = \int_0^{b_1 + \ee}F_\ee'(s) \diff s = D.
\end{equation*}
Notice that $F_\ee(s) > 0$ for $s > 0$ and $F_\ee' \le 1$ and $F_\ee \in C^1$. This form is rather elaborate, so we pick the limit as $\ee \searrow 0$. Notice that $ \inf F_\ee' \to - \infty$ as $\ee \searrow 0$.

\begin{proof}[Proof of \Cref{lem:rho_D}]
We start by pointing out that $\int_{B_R} \rho_F$ is continuous in all parameters.
Taking $D = b_1$ we have that as $b_2 \searrow b_1$ we have that
$
	\int_{B_R} \rho_D \searrow \int_{B_R} \rho_V = a_{V,R} < a_{0,R}.
$
Hence, when $D = b_1$, there exists $\overline b_2(b_1)$ such that for $b_2 < \overline b_2$, and $D = b_1$,  $\int_{B_R} \rho_D  < a_{0,R}$. Fixed $b_1$ and $b_2$, as $D \searrow 0$ we have $\int_{B_R} \rho_D \nearrow \infty$. So there exists a choice of $D$ such that $\int_{B_R} \rho_F = a_{0,R}$.
\end{proof}

Notice that $\rho_D \in L^{1+\ee}(B_R)$ due to the assumption \eqref{eq:rhoV in L1ee loc}.
We sketch the profile in \Cref{fig:rhoD}.
\begin{figure}[h]
	\centering
	\includegraphics[width=.5\textwidth]{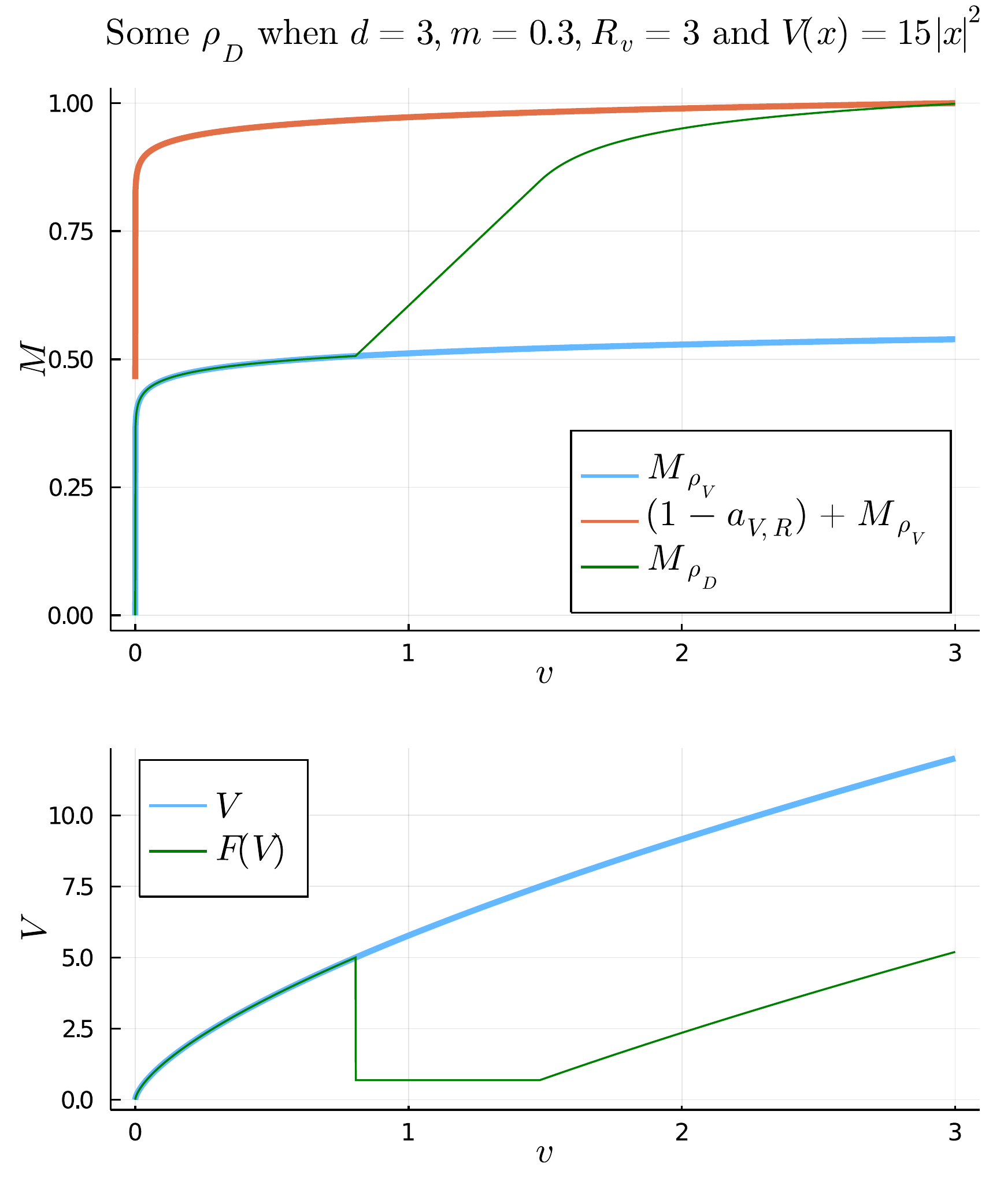}
	\caption{Example of $\rho_D$ for some parameters $b_1$ and $b_2$. In this example $a_{0,R} = 1$.}
	\label{fig:rhoD}	
\end{figure}

\begin{theorem}[Solutions of \eqref{eq:main bounded domain} with increasing mass]
\label{thm:R finite concentration mass}
	Under the hypothesis of \Cref{thm:existence bounded domain regular data},
	let $\rho_D$ be given by \eqref{eq:bounded rho0 with concentration}. Then, the mass $M$ of $\rho(t) = S(t) \rho_D$ constructed in \Cref{thm:existence bounded L1} is such that
	\begin{equation*}
			M(t, v) \nearrow (a_{0,R} - a_{V,R}) + M_{\rho_V} (v) \qquad \text{ uniformly in } [\ee,R_v].
	\end{equation*}
	In particular, $\rho(t, \cdot) \rightharpoonup (a_{0,R} - a_{V,R}) \delta_0 + \rho_V$ weak-$\star$ in the sense of measures.
\end{theorem}

\begin{proof}[Proof of \Cref{thm:R finite concentration mass}]
	\textbf{Step 1. Properties by approximation.}
Since $\rho_D \in L^{1+\ee}$, looking at how we constructed $S(t) \rho_0$ in \Cref{thm:existence bounded domain regular data,thm:existence bounded L1}, it can approximated by  $S_k (t) \rho_0$ where $S_k$ is the semigroup of \eqref{eq:main regularised bounded} with $\Phi_k$ given by \eqref{eq:Phik}.
Notice that, the associated $\Theta_k$ given by \eqref{eq:Theta} is
\begin{equation*}
	\Theta_k (s) =
		\tfrac{m}{1-m} \left( 1 -  s^{ {m-1} } \right), \quad  \text{for }s \in [k^{-1}, k] .
\end{equation*}
Hence, we recover
\begin{equation*}
	\Theta_k^{-1} (s) =
		( 1 -  \tfrac{1-m}{m} s  )^{-\frac 1 {1-m}}, \qquad  \text{for } s \in [\Theta_k(k^{-1}), \Theta_k (k)].
\end{equation*}
Taking $h = \frac{m}{1-m}$ in \eqref{eq:increasing solutions Phi smooth} we have initial data $u_{0,k}$ such that
\begin{equation*}
	u_{0,k} = (\tfrac{1}{1-m} F(V))^{-\frac 1 {1-m}} , \qquad \text{whenever }F(V(x)) \in [\Theta_k (k^{-1}), \Theta_k (k)],
\end{equation*}
and $M_{u_k}$ non-decreasing in $t$.
This corresponds to an interval of the form $v \in [\ee_k, R_v - \delta_k]$. Let us denote $u_k = S_k (t) u_{0,k}$.
Due to the $L^1$ contraction we have that
\begin{equation*}
	\int_\Br |u_k (t) - S_k (t) \rho_D | \diff x \le \int_{\Br} |u_{0,k} - \rho_D| \diff x.
\end{equation*}
Hence, by \Cref{thm:existence bounded L1} we infer that
$
	S_k (t) u_{0,k} \to S(t) \rho_D $ in $ L^1 (\Br) $ for a.e. $t > 0.
$
This guarantees the a.e. convergence of the masses. Hence, the mass function $M$, which is already a viscosity solution of \eqref{eq:mass} and $C^\alpha$ regular, also inherits the point-wise estimate from $M_{u_k}$ in \eqref{eq:lower bound on dv M}. $M$ is also non-decreasing in $t$ and $v$.
Moreover,
due to \eqref{eq:bounded rho0 with concentration comparison masses} and \Cref{thm:comparison principle for mass} due to \Cref{eq:local Calpha regularity}, we conclude that
\begin{equation}
	\label{eq:R finite concentration mass aux 1}
	M_{\rho_D} (v) \le M(t,v) \le (a_{0,R} - a_{V,R}) + M_{\rho_V} (v).
\end{equation}

\noindent \textbf{Step 2. Uniform convergence of $M(t,\cdot)$ as $t \to +\infty$.}
Since $M$ is point-wise non-decreasing in $t$ and bounded above by $a_{0,R}$, we know there exists a function $M_\infty$ such that
\begin{equation}
	\label{eq:R finite concentration mass aux 2}
	M(t,x) \nearrow M_\infty (x) , \qquad t \nearrow \infty.
\end{equation}
By the estimate \eqref{eq:local Calpha regularity} we know that $M_\infty$ belongs to $C^\alpha_{loc} ((0,R_v))$ and hence continuous in interior points. On the other hand, \eqref{eq:R finite concentration mass aux 1}
 implies
\begin{equation*}
	M_{\rho_D} (v) \le M_\infty(v) \le (a_{0,R} - a_{V,R}) + M_{\rho_V} (v).
\end{equation*}
Hence, by the sandwich theorem, $M_\infty (R_v) = a_{0,R}$ and it is continuous at $R_v$ (due to the explicit formulas we can actually show rates). Since $M_\infty$ is non-decreasing and $M_\infty \ge 0$, due to \eqref{eq:R finite concentration mass aux 2}, there exists a limit
	\begin{equation*}
		\lim_{v \to 0} M_\infty(v) \le a_{0,R} - a_{V,R}.
	\end{equation*}
Defining $M_\infty (0) = \lim_{v \to 0} M_\infty(v)$, the function is obviously continuous in $[0,R_v]$.
Hence, applying Dini's theorem, we know that
\begin{equation*}
	\sup_{v \in [\ee, R_v]} | M (t, v) - M_\infty (v) | \to 0.
\end{equation*}
Due to \eqref{eq:lower bound on dv M} and our choice of $h$, we have that
\begin{equation}
	\label{eq:lower bound on dv M infty}
	M_\infty(v_2) - M_\infty(v_1) \ge ( v_2 - v_1 ) \inf_{\widetilde B_{v_2} \setminus \widetilde B_{v_1}} \rho_V , \qquad \forall v_1 \le v_2.
\end{equation}

\noindent \textbf{Step 3. Characterisation of $M_\infty$ as a viscosity solution.}
Let us check that $M_\infty$ is a viscosity solution of
\begin{equation}
\label{eq:problem M infty}
	\frac{\partial^2 M_\infty}{\partial v^2} + \frac 1 m \left( \frac{\partial M_\infty}{\partial v} \right)^{2-m} \frac{\partial V}{\partial v} = 0.
\end{equation}
Due to our lower bound \eqref{eq:lower bound on dv M infty}, $\frac{\partial M_\infty}{\partial v}$ is bounded below.
We define the sequence of masses $M_n:[0,1] \times [0,R_v] \to \mathbb R$ given by
$
	M_n (t,v) = M(t - n, v).
$
These are viscosity solutions for \eqref{eq:mass} due to \Cref{thm:comparison principle for mass}.
We also know that
\begin{equation*}
	\sup_{(t,v) \in [0,1] \times [\ee, R_v]} |M_n(t, v) - M_\infty(v) | \to 0.
\end{equation*}
By standard arguments of stability of viscosity solutions, $M_\infty$ is also a solution of \eqref{eq:mass}. Since it does not depend on $t$, we can select spatial viscosity test functions, and hence it is a solution of \eqref{eq:problem M infty}.
Since we have removed the time dependency, we dropped also the spatial weight $(n \omega_n v^{\frac{n-1}n})^2$.

\noindent \textbf{Step 4. $M_\infty$ is $C^2 ((0, R_v))$.}

\textbf{Step 4a. Lipschitz regularity}
Since $M_\infty$ is non-decreasing, at the point of contact of a viscosity test function touching from below, we deduce
\begin{equation*}
	- \frac{\partial^2 \varphi}{\partial v^2} (v_0) \ge \frac 1 m \left( \frac{\partial \varphi}{\partial v} (v_0) \right)^{2-m} \frac{\partial V}{\partial v} (v_0) \ge 0.
\end{equation*}
Hence, $M_\infty$ is a viscosity super-solution of $- \Delta M = 0$. Due to \cite{Ishii1995}, we have that $M$ is also a distributional super-solution of $-\Delta M = 0$. Distributional super-solutions are concave. Since $M_\infty$ is concave, it is $W^{1,\infty}([\ee,R_v - \ee])$ of all $\ee > 0$.

\textbf{Step 4b. Higher regularity by bootstrap.}  Now we can treat the right-hand side as a datum
	$$f = \frac 1 m \left( \frac{\partial M_\infty}{\partial v} (v_0) \right)^{2-m}\frac{\partial V}{\partial v} \in L^\infty (\ee, R_v-\ee).$$
	Applying the regularisation results in \cite{Caffarelli1995} we recover that $M_\infty \in C^{1,\alpha}(2 \ee, R_v - 2\ee)$. Since $V \in W^{2,\infty} = C^{0,1}$, then $f \in C^{0,\beta}(2 \ee, R_v - 2\ee)$ for some $\beta > 0$, so $M_\infty \in C^{2,\beta}(4 \ee, R_v - 4\ee)$.

\noindent \textbf{Step 5. Explicit formula of $M_\infty$.}
Since $M_\infty \in C^{2} ((0,R_v)) \cap C ([0,R_v])$, we can integrate \eqref{eq:problem M infty} to show that
\begin{equation*}
	M_\infty (v) = M_\infty(0) + M_{\rho_{V+h}}
\end{equation*}
for some $h \ge 0$.
Since $M_\infty(R_v) - M_\infty (0) = a_{0,R} - M_\infty (0) \le a_V$ then, for some $h \ge 0$ we have that $M_\infty(R_v) - M_\infty (0) = a_{V + h,R}$.
By the comparison principle, which holds due to \eqref{eq:lower bound on dv M infty}, we conclude the equality $M_\infty (v) - M_\infty (0) = M_{ \rho_{V + h}} (v)$ for $v \in [0,R_v]$.
Due to \eqref{eq:lower bound on dv M infty}, the singularity at $0$ is incompatible with $h > 0$. Thus $h = 0$.
\end{proof}

\begin{remark}
	Notice that the aggregation does not occur in finite time, since we assume \eqref{eq:rhoV in L1ee loc}.
\end{remark}

\begin{proof}[Proof of \Cref{thm:bounded concentrating intro}]
	To compute the $\liminf$, it suffices to pick a $\rho_D$
	such that $M_{\rho_0} \ge M_{\rho_D}$.
	This can be done by selecting $b_1$ sufficiently close to $V(R)$.
	If we assume \eqref{eq:rho0 below limit}, by the comparison of masses we have
	\begin{equation*}
		\int_{B_r} \rho (t,x) \diff x \le (a_{0,R} - a_{V,R}) + M_{\rho_V}(v), \qquad \forall r \in [0,R].
	\end{equation*}
	Then, the $\liminf$ and $\limsup$ coincide with this upper bound, i.e.
	\begin{equation*}
		\lim_{t \to +\infty} \int_{B_r} \rho (t,x) \diff x = (a_{0,R} - a_{V,R}) + M_{\rho_V}(v), \qquad \forall r \in [0,R].
	\end{equation*}	
	To check the convergence in Wasserstein distance, we must write the convergence of the masses in $L^1$ in radial coordinates. Let $\mu_{\infty,R} = (a_{0,R} - a_{V,R}) \delta_0 + \rho_V $, then we have that
	\begin{equation*}
		d_{1} (\rho(t) , \mu_{\infty,R}) = n \omega_n \int_{0}^R\left| \int_{B_r} \rho(t,x) \diff x - \mu_{\infty,R} (B_r) \right| r^{n-1} \diff r.
	\end{equation*}
	due to the fact that the optimal transport between radial densities is radial and the characterisation of $d_1$ in one dimension (see \cite{Villani03}).	
	Since we have shown in the proof above that $ \int_{B_r} \rho(t) \diff x \le \mu_\infty (B_r) $
	\begin{equation*}
		d_{1} (\rho(t) , \mu_\infty) = n \omega_n \int_{0}^R\left(\mu_{\infty,R} (B_r) - \int_{B_r} \rho(t,x) \diff x
		 \right) r^{n-1} \diff r.
	\end{equation*}
	Due to the monotone convergence $\int_{B_r} \rho(t,x) \nearrow \mu_{\infty,R} (B_r)$ for $r \in (0,R]$, the right-hand goes to $0$ as $t \to +\infty$.
\end{proof}

\section{Minimisation of $\mathcal F_R$}
\label{sec:FR}

It is very easy to see that the free energy $\mathcal F_R$ is bounded below, in particular
\begin{equation}
	{\mathcal F_R}[\rho] \ge - \tfrac 1 {1-m} |\Br|^{\frac 1 {1-m}} \|\rho\|_{L^1(\Br)}^{m},
\end{equation}
due to \eqref{eq:L1 controls Lm over compacts} and that $V \ge 0$.
Therefore, there exists a minimising sequence. The problem is that the functional setting does not offer sufficient compactness to guarantee its minimiser is in $L^1(\Br)$. However, we can define its extension to the set of measures as
\begin{equation*}
	\widetilde{\mathcal F_R}[\mu] = - \tfrac 1 {1-m} \int_{\Br} \mu_{ac}^m + \int_{\Br} V d\mu
\end{equation*}
This is the unique extension of $\mathcal F_R$   to $\mathcal M_+ (\Br)$ that is lower-semicontinuous in the weak-$\star$ topology (see
\cite{Demengel1986} and related results in \cite{Buttazzo1989}).

Since we work on a bounded domain, tightness of measures is not a limitation.
For convenience, let us define for $\rho \in L^1 (\Br)$,
\begin{equation*}
	\mathcal E_{m,R} [\rho] = \tfrac{1}{m-1}\int_\Br \rho(x)^m \diff x.
\end{equation*}
Let us denote the set of non-negative measures of fixed total mass $\mathfrak m$ in $\Br$ as
\begin{equation*}
	\mathcal P_{\mathfrak{m}} (\overline \Br) = \{  \mu \in \mathcal M_+ (\overline \Br) : \mu (\Br) = \mathfrak{m} \}.
\end{equation*}
We have the following result
\begin{theorem}[Characterisation of the unique minimiser of $\mathcal F_R$]
	\label{thm:bounded minimising FR}
	Let us fix $	\mathfrak{m} > 0$, $V \in W^{2,\infty} (\Br)$, $V (0) = 0$ and $V$ is radially increasing. Then, any
	sequence $\rho_j$ minimising $\mathcal F_R$ over $\mathcal P_{\mathfrak{m}} (\overline \Br) \cap L^1 (\Br)$
	converges weakly-$\star$ in the sense of measures to
	\begin{equation*}
	\mu_{\infty,\mathfrak{m}} = \begin{dcases}
 		\rho_{V+h} & \text{for } h \text{ such that }a_{V+h} = \mathfrak{m}, \\
 		(\mathfrak{m}-a_{V,R})\delta_0 +\rho_{V} & \text{if }a_{V,R} < 	\mathfrak{m}. \\
 	\end{dcases}	
	\end{equation*}
 	Furthermore,
 	\begin{equation}
 	\label{eq:bounded mu infinity minimising}
 		\widetilde {\mathcal F_R}[\mu_{\infty,\mathfrak{m}}] =
 		\inf_{\mu \in \mathcal P_{\mathfrak{m}} (\overline \Br) } \widetilde {\mathcal F}_R[\mu] =
 		\inf_{\rho \in \mathcal P_{\mathfrak{m}} (\overline \Br) \cap L^1 (\Br)} \mathcal F_R[\rho]  		.
 	\end{equation}
\end{theorem}

\begin{remark}[Lieb's trick]
\label{rem:Liebs trick}
	Given a radially decreasing $\rho \ge 0$, $\rho^q \in L^1 (\Br)$ for some $ q > 0$ (for any $R \le \infty$),
	using and old trick of Lieb's (see \cite{Lieb1977,Lieb1983}) we get, for $|x| \le R$,
	\begin{equation*}
		\int_\Br \rho^q \diff x = n \omega_n \int_0^R \rho(r)^q r^{n-1}  \diff r \ge n \omega_n \int_0^{|x|} \rho(r)^q r^{n-1}  \diff r  \ge n \omega_n \rho ({x})^q \int_0^{|x|} r^{n-1}  \diff r  .
	\end{equation*}
	Hence, we deduce the point-wise estimate
	\begin{equation}
		\label{eq:Lieb's trick estimate for rhom}
		\rho(x) \le \left(   \frac{\int_{\Br} \rho^q }{n \omega_n {|x|}^{n}}    \right)^{\frac 1 q}.
	\end{equation}
	It is easy to see that \eqref{eq:Lieb's trick estimate for rhom} is not sharp. However, it is useful to prove tightness for sets of probability measures.
	Similarly, if additionally	$V \rho \in L^1 (\Br)$,  and $V \ge 0$ we can estimate
	\begin{align*}
		\int_\Br V \rho \diff x = n \omega_n \int_0^{|x|} V(r) \rho(r) r^{n-1}  \diff r \ge n \omega_n \int_0^{|x|} V(r) \rho(r) r^{n-1}  \diff r  \ge n \omega_n \rho (x) \int_0^{|x|} V(r) r^{n-1}  \diff r,
	\end{align*}
	so we recover the point-wise estimate
	\begin{equation}
	\label{eq:Lieb's trick estimate for Vrho}
		\rho (x) \le \frac{\int_\Br V \rho }{\int_{B_{|x|}} V}.
	\end{equation}
\end{remark}

\begin{proof}[Proof of \Cref{thm:bounded minimising FR}]
	The second equality in \eqref{eq:bounded mu infinity minimising} is due to the weak-$\star$ density of $L^1_+ (\Br)$ in the space of non-negative measures, and the construction of $\widetilde{\mathcal F}_R$ (see \cite{Demengel1986}).
	Let us consider a minimising sequence. Let us show that we can replace it by a radially-decreasing minimising sequence.  Let $\rho_j \in L^1_+ (\Br) $ with $\| \rho_j \|_{L^1} = \mathfrak{m}$. By standard rearrangement results
	\begin{equation*}
		\mathcal E_{m,R} [\rho_j^\star] = \mathcal E_{m,R} [\rho_j].
	\end{equation*}
	Since $V \ge 0$ and radially symmetric and non-decreasing then
	\begin{equation*}
		 \int_{\Br} V(x) \rho_j^\star (x)  \diff x \le  \int_{\Br} V(x) \rho_j (x)  \diff x.
	\end{equation*}
	Hence, there exists minimising sequence $\rho_j \in L^1 (\Br)$ that we can assume   radially non-increasing.
	Since $\rho_j \in \mathcal P_{\mathfrak m} (\overline B_R)$, by Prokhorov's theorem, this minimising sequence must have a weak-$\star$ limit in the sense of measures, denoted by $\mu_{\infty, \mathfrak m}$.
	
	We use the following upper and lower bounds that follow from \eqref{eq:L1 controls Lm over compacts}
	\begin{equation*}
		\int_\Br V \rho_j
		=
		\mathcal F_R[\rho_j] + \tfrac{1}{1-m} \int_{\Br} \rho_j^m \le \mathcal F_R[\rho_j] + \tfrac{1}{1-m} |\Br|^{1-m} \| \rho_j \|_{L^1}.
	\end{equation*}
	Due to \eqref{eq:Lieb's trick estimate for Vrho} we have a uniform bound in $L^\infty (\Br \setminus B_\ee)$ for any $\ee > 0$. Thus, there exists $\rho_\infty \in L^1_+ (\Br) \cap L^\infty (\Br \setminus B_\ee)$ for any $\ee \ge 0$ such that
	\begin{equation*}
		\mu_{\infty, \mathfrak m} = \Big(\mathfrak{m} - \| \rho_\infty \|_{L^1 (\Br)} \Big) \delta_0 + \rho_\infty.
	\end{equation*}

	Let us now characterise this measure.
	For $\varphi \in C^\infty_c (\mathbb R^n)$ we take
	\begin{equation}
		\psi (x) = \left(  \varphi (x)
		\int_{B_R} \rho_\infty (y) \diff y
		 - \int_{\Br} \varphi (y) \rho_\infty (y) \diff y  \right) \rho_\infty (x) .
	\end{equation}
	For $\varphi$ fixed, there is $\ee_0>0$ such that for $\ee < \ee_0$, $\mu_{\infty,\mathfrak{m}} + \ee \psi \in \mathcal P_{\mathfrak{m}} (\mathbb R^n)$ and, hence,
	\begin{equation*}
	  {\mathcal F_R} [\rho_\infty]	= \widetilde {\mathcal F_R} [\mu_{\infty,\mathfrak{m}} ] \le \widetilde {\mathcal F_R} \left [ \mu_{\infty,\mathfrak{m}} + \ee \psi  \right].
	\end{equation*}
	Hence, we get the expression
	\begin{equation*}
	\mathcal E_{m,R}\left [ \rho_\infty + \ee \psi  \right]  - \mathcal E_{m,R} [ \rho_\infty ] + \ee  \int_{\Br} V(x) \psi (x) \diff x \ge 0.
	\end{equation*}
	We write
	\begin{align*}
		\frac{ \mathcal E_{m,R}\left [ \rho_\infty + \ee \psi  \right]  - \mathcal E_{m,R} [ \rho_\infty ] } \ee  &= \tfrac m {m-1} \int_0^1 \left(  \int_\Br |\rho(x) + t \ee \psi (x)|^{m-2} (\rho(x) + t \ee \psi(x)) \psi (x) \diff x \right)  \diff t	.
	\end{align*}
	Since we have the estimate
	\begin{equation*}
		\left|  \int_\Omega |\rho(x) + t \ee \psi (x)|^{m-2} (\rho(x) + t \ee \psi(x)) \psi (x) \diff x \right| \le (\| \rho_\infty \|_{L^m} + \ee_0 \| \psi \|_{L^m})^{m-1}  \| \psi \|_{L^m},
	\end{equation*}
	we recover by the Dominated Convergence Theorem
	\begin{equation*}
		\lim_{\ee \to 0} \frac{ \mathcal E_{m,R}\left [ \rho_\infty + \ee \psi  \right]  - \mathcal E_{m,R} [ \rho_\infty ] } \ee = \frac m {m-1} \int_\Br \rho_\infty ^{m-1} \psi .
	\end{equation*}
	Thus, as $\ee \to 0$ the following inequality holds
	\begin{equation*}
		 \int_{\Br} I[\rho_\infty] \psi \ge 0, \qquad \text{with } I[\rho] \defeq \frac m {m-1}  \rho ^{m-1} + V.
	\end{equation*}
	Applying the same reasoning for $-\psi$ (which corresponds to taking $-\varphi$ instead of $\varphi$), we deduce the reversed inequality, and hence the equality to $0$. This means that
	\begin{align*}
		0 &= \int_{\Br} I[\rho_\infty] (x) \varphi (x) \rho_\infty(x)
		\left(
		\int_{B_R} \rho_\infty (y) \diff y
		\right)
		 \diff x  - \int_{\Br}  \left( \int_{\Br} \varphi (y) \rho_\infty(y) \diff y \right)  I[\rho_\infty] (x)  \rho_\infty (x)  \diff x  \\
		&= \int_{\Br} \varphi (x) \rho_\infty(x)  \left(  I[\rho_\infty] (x)
		\left(
		\int_{B_R} \rho_\infty (y) \diff y
		\right)
		 -  \int_{\Br}  I[\rho_\infty] (y)  \rho_\infty (y)   \diff y \right)  \diff x
	\end{align*}
	Since $\mathfrak m \delta_0$ is not a minimiser (see \Cref{rem:energy formal analysis on Diracs}) then $\rho_\infty \not \equiv  0$.
	As $\varphi$ concentrates to a point, we recover for a.e. $x$ either
	\begin{align*}
		\rho_\infty (x) = 0 \qquad \text{or} \qquad 	I[\rho_\infty] (x)  =  \frac{ \int_{\Br}  I[\rho_\infty] (y)  \rho_\infty (y)   \diff y}{%
			\int_{B_R} \rho_\infty (y) \diff y
			 } =: \mathcal C[\rho_\infty].
	\end{align*}
	Notice that the right hand of the second term is a constant.
	 Since $\rho_\infty$ is radially decreasing then there exists $R_\infty > 0$ such that
	\begin{equation*}
		\rho_\infty(x) =
		\begin{dcases}
			\left (  \tfrac{1-m}{m} (V - \mathcal C[\rho_\infty]  ) \right )^{- \frac 1 {1-m} } & |x| \le R_\infty, \\
			0 & R_\infty < |x| < R
		\end{dcases} = \rho_{V+h} (x) \chi_{B_{R_\infty}} (x)
	\end{equation*}	
	where, by evaluating close to $0$ we deduce that $h = - \mathcal C[\rho_\infty] \ge 0$.
	Notice that $\rho_\infty$ is the minimiser of the two variable function
	\begin{align*}
		f(\tau,h) = \widetilde{\mathcal F_R} & \left[\rho_{V+h} \chi_{B_{\tau}} + \left( \mathfrak m - \int_{B_{\tau}} \rho_{V+h}\right) \delta_0 \right] = \mathcal F_R [\rho_{V+h} \chi_{B_{\tau}}] \\
		&= |\partial B_1| \int_0^{\tau} \left( -\tfrac {1}{1-m} \rho_{V+h}^{m} (r) + V (r) \rho_{V+h} (r) \right) r^{n-1} \diff r
	\end{align*}
	under the total mass constraint that
	\begin{equation*}
		|\partial B_1| \int_{0}^{\tau} \rho_{V+h} (r) r^{n-1} \diff r \le \mathfrak m.
	\end{equation*}
	It is not a difficult exercise to check that the minimum is achieved with $R_\infty = R$ and $h$ as small as possible.
	When $ \mathfrak m > a_{V,R}$ (which corresponds to $h = 0$) we have to add a Dirac Delta at the origin, with the difference of the masses $\mathfrak m  - a_{V,R}$.
	
	To check this, first we point out that
	\begin{align*}
		 -\tfrac {1}{1-m} \rho_{V+h}^{m}  + V  \rho_{V+h} & = \rho_{V+h}  \left( -\tfrac {1}{1-m} \rho_{V+h}^{m-1} + V   \right) = - \rho_{V+h} \left( \tfrac{1-m}{m} V + \tfrac 1 m h \right)< 0 \qquad \text{for all } r > 0.
	\end{align*}
	We deduce that $f \le 0$ and increasing the integration domain decreases $f$, i.e $\frac{\partial f}{\partial \tau} < 0$ for all $\tau, h > 0$. On the other hand, since $\frac{\partial \rho_{V+h}}{\partial h } < 0$ for $r, h > 0$ we have that
	\begin{equation*}
		\frac{\partial f}{\partial h} = |\partial B_1|\int_0^\tau \left( \frac{m}{m-1} \rho^{m-1} + V \right) \frac{\partial \rho_{V+h}}{\partial h} r^{n-1}\diff r = - h |\partial B_1|\int_0^\tau \frac{\partial \rho_{V+h}}{\partial h} r^{n-1}\diff r > 0 .
	\end{equation*}
	Hence, the derivative is not achieved at interior points. We look at the boundaries of the domain:
	\begin{enumerate}
		\item The segment $(\tau, h) \in \{0\} \times [0,+\infty)$, where $f = 0$. These are all maximisers.
		
		\item The segment $(\tau,h) \in \{R\} \times [h_0, +\infty)$, where $h_0$ is the minimum value so that $\int_{\Br} \rho_{V+h} \le \mathfrak m$. If $\mathfrak m > a_{V,R}$, then $h_0 = 0$. Using the derivative respect to $h$, the minimum in this segment is achieved at $(R,h_0)$.
		
		\item A segment $(\tau,h) \in [0,R_0] \times \{0\}$, where $R_0$ is such that $\int_{B_{R_0}} \rho_V = \mathfrak m$. If $\mathfrak m > a_{V,R}$, then $R_0 = R$. Using the derivative respect to $\tau$, the minimum in this segment is achieved at $(R_0,0)$.
		
	\end{enumerate}
	
	Hence, if $\mathfrak m \ge a_{V,R}$, the minimum is achieved at $\tau = R$ and $h = 0$. The remaining mass is completed with a Dirac. Lastly, if $\mathfrak m < a_{V,R}$ there is an extra part of the boundary, where the mass condition is achieved with equality
	
	\begin{enumerate}[resume]
		
		\item  The curve $(\tau, \overline h(\tau))$ such that
		\begin{equation*}
					|\partial B_1| \int_{0}^{\tau} \rho_{V+\overline h} (r) r^{n-1} \diff r = \mathfrak m.
		\end{equation*}
		Notice that this segment contains the minima of the other segments.
		Taking a derivative respect to $\tau$ we deduce that
		\begin{equation*}
			\frac{\diff \overline h}{\diff \tau} = - \frac{  \rho_{V+\overline h} (\tau) \tau^{n-1} } { \displaystyle  \int_{0}^{\tau} \frac{ \partial \rho_{V+\overline h} }{\partial h} (r) r^{n-1} \diff r} \ge 0.
		\end{equation*}
		Therefore, using Leibniz's rule again we recover
		\begin{align*}
			\frac{\diff }{\diff \tau} \left ( f(\tau, \overline h (\tau)) \right ) &= - |\partial B_1| \frac{\diff }{\diff \tau} \int_0^\tau  \rho_{V+\overline h} \left( \tfrac{1-m}{m} V + \tfrac 1 m \overline h \right) r^{n-1} \diff r \\
			&= - |\partial B_1| \rho_{V+\overline h} (\tau) \left( \tfrac{1-m}{m} V (\tau) + \tfrac 1 m \overline h \right) \tau^{n-1} \\
			&\qquad -  |\partial B_1| \frac{\diff \overline h }{\diff \tau } \int_{0}^{\tau} \left(  \frac{ \partial \rho_{V+\overline h} }{\partial h} \left( \tfrac{1-m}{m} V + \tfrac 1 m \overline h \right) + \rho_{V+ \overline h}  \right)  r^{n-1} \diff r \le 0.
		\end{align*}
		Finally, the minimum is achieved for the lasted $\tau$, so again the minimum is $(R,h_0)$. \qedhere
	\end{enumerate}
\end{proof}

\begin{remark}[$\mathfrak{m} \delta_0$ is not a minimiser]
\label{rem:energy formal analysis on Diracs}
Let $\rho \in L^1_+(\Br)$ smooth be fixed and let us consider the dilations
$	
	\rho_s (x) = s^n \rho(s x)
$
for $s \ge 1$. Notice that $\rho_s \to \delta_0$ as $s \to +\infty$ in the weak-$\star$ of $\mathcal M (\overline \Br)$. As $s \to \infty$ we can compute
\begin{equation*}
	\mathcal F_R[\rho_s] = \frac{s^{n(m-1)}}{m-1} \int_\Br \rho(x)^m \diff x + \int_\Br V(s^{-1}x) \rho(x) \diff x \longrightarrow 0+V(0)\int_\Br \rho(x) \diff x = 0.
\end{equation*}
It is not difficult see that $\mathcal F_R$ takes negative values, so this is not a minimiser.
\end{remark}

In \cite{Cao2020} the authors prove that in $\Rd$ if
$
	\rho_{V+h_1} \le \rho_0 \le \rho_{V+h_2}
$
then $\rho(t) \to \rho_{V+h}$ of the same initial mass.
This shows that $\mu_{\infty,\mathfrak{m}} = \rho_{V+h}$ is attractive in the cases without Dirac Delta concentration at the origin.

We have constructed initial data $\rho_0 > \rho_V$ such that $\rho(t) \to \mu_{\infty,\mathfrak{m}}$ in the sense of their mass functions. Furthermore, we show that
\begin{lemma}[Minimisation of $\mathcal F_R$ through solution of \eqref{eq:main bounded domain}] Assume
	$ \rho_V\le \rho_0 $, \eqref{eq:rho0 below limit},
	 $a_{V,R} < a_{0,R} = \|\rho_0 \|_{L^1 (\Br)}$ and let $\rho$ be constructed in \Cref{thm:bounded concentrating intro}. Then
$$
	\mathcal F_R[\rho(t)] \searrow 	\widetilde{\mathcal F_R}[\mu_{\infty, a_{0,R}}]=\mathcal F_R[\rho_V].
$$	
\end{lemma}
\begin{proof}
From the gradient flow structure we know $\mathcal F[\rho(t)]$ is non-increasing. First, we prove $\rho(t) \to \rho_V$ in $L^1(\Br \setminus B_\varepsilon)$ for some $\ee$ small.
	We know that $\rho(t) \ge \rho_V$ so 	
	\begin{align*}
		\int_{\Br \setminus B_\varepsilon} |\rho(t) - \rho_V | &=	\int_{\Br \setminus B_\varepsilon} (\rho(t) - \rho_V )  = M(t,R)-M(t,\ee) - (M_V (R) - M_V(\ee)) ) \\
		&= (a_{0,R} - a_{V,R}) + M_V(\ee) - M(t,\ee) \to 0
	\end{align*}
	as $t \to \infty$ due to \Cref{thm:bounded concentrating intro}.
Now we can explicitly compute
\begin{align*}
	|\mathcal F[\rho(t)] - \mathcal F[\rho_V]| &\le   \left| \frac{1}{1-m}\int_{B_\varepsilon} (\rho(t)^m - \rho_V^m) \right|
	 +  \left| \frac{1}{1-m}\int_{\Br \setminus B_\varepsilon} (\rho(t)^m - \rho_V^m)  \right| \\
	& \qquad +  \left| \int_{B_\varepsilon} V (\rho(t) - \rho_V) \right|
	+ \left| \int_{\Br \setminus B_\varepsilon} V (\rho(t) - \rho_V)  \right|\\
	& \le \frac{|B_\ee|^{ {1-m}}}{1-m}( \|\rho(t)\|_{L^1}^m + \|\rho_V\|_{L^1}^m ) + \left| \frac{1}{1-m}\int_{\Br \setminus B_\varepsilon} (\rho(t)^m - \rho_V^m)  \right|\\
	& \quad + \left( \sup_{x \in B_\ee} V(x) \right) ( \|\rho(t)\|_{L^1} + \|\rho_V\|_{L^1} ) + \left( \sup_{x \in B_R} V(x) \right) 		\int_{\Br \setminus B_\varepsilon} |\rho(t) - \rho_V | .
\end{align*}
Due to the $L^1$ convergence, we can extract a sequence $t_k \to \infty$ such that $\rho(t_k) \to \rho_V$ a.e. in $\Br\setminus B_\ee$. For this subsequence, due to Fatou's lemma and $\rho(t)\ge\rho_V$ we have
\begin{equation*}
	\int_{\Br\setminus B_\ee} \rho(t_k)^m \to 	\int_{\Br\setminus B_\ee} \rho_V^m.
\end{equation*}
Collecting the above estimates, we conclude that
\begin{equation*}
	\limsup_{k \to \infty} |\mathcal F[\rho(t_k)] - \mathcal F[\rho_V]|\le \tfrac{1}{1-m}|B_\ee|^{{1-m}}(\|\rho(t)\|^m_{L^1} + \|\rho_V\|_{L^1}^m) + \left(   \sup_{x \in B_\ee} V(x) \right)( \|\rho(t)\|_{L^1} + \|\rho_V\|_{L^1} )
\end{equation*}
for any $\ee > 0$. Letting $\ee \to 0$ we recover that $\limsup_k$ is actually a $\lim_k$, and it is equal to $0$. Since $\mathcal F[\rho(t)]$ is non-increasing, we recover the limit as $t \to \infty$.
\end{proof}

\section{The problem in $\mathbb R^n$}
\label{sec:Rn}

We start by showing the existence of a viscosity solution of the mass equation \eqref{eq:mass}, by letting $R \to +\infty$.
As $R \to \infty$ we can modify $V_R$ only on $(R-1) < |x| < R$ to have $\nabla V_R (x) \cdot x = 0$ for $|x|=R$.
Fix $\rho_0 \in L^1 (\Rd)$ radially symmetric. Let $M_R$ be the solution of the mass equation with this data.
Consider the extension
\begin{equation*}
	\widetilde M_R (t,v) = \begin{dcases}
		M_R (t,v) & v \le R_v, \\
		\| \rho_0 \|_{L^1 (\Br)} & v > R_v,
	\end{dcases}
\end{equation*}
where, as above, we denote $R_v = R^n |B_1|$.
Since $\| \widetilde M_R \|_{L^\infty ( (0,\infty) \times (0,\infty))}$ we have that, up to a subsequence
\begin{equation*}
	\widetilde M_{R_k} \rightharpoonup M \text{ weak-}\star \text{ in } L^\infty ((0,\infty)^2).
\end{equation*}
We can carry the estimate in $C^\alpha ( [T_1,T_2] \times [v_1,v_2] )$ given in \eqref{eq:local Calpha regularity}, which is uniform in $R$ since $\| M_R \| _{L^\infty} \le 1$, for any $0 < T_1 < T_2 < \infty$ and $0 < v_1 < v_2 < R_v$.

Now we show $M$ is a viscosity solution. Due to the
uniform continuity provided by \Cref{thm:comparison principle for mass}
 and the Ascoli-Arzelá theorem,
for any $K = [0,T] \times [v_1,v_2]$ with $v_1, v_2 , T > 0$, we have a further subsequence that converges in $C(K)$ to some function $\widetilde M$ the uniform continuity.
It is easy to characterise $\widetilde M = M$ almost everywhere. Due to the uniform convergence, we preserve the value of $M(0,v) = M_R(0,v)$ for $v \le R_v$.
Applying the same stability arguments for viscosity solutions as in \Cref{thm:comparison principle for mass}, $M$ is a viscosity solution of the mass equation \eqref{eq:mass}.

\begin{proposition}
\label{prop:Rd existence mass}
	Assume $V \in W^{2,\infty}_{loc} (\Rd)$ is radially symmetric, strictly increasing, $V\ge 0$, $V (0) = 0$ and the technical assumption \eqref{eq:rhoV in L1ee loc}.
	Let $\rho_0 \in L^1 (\Rd)$ be radially symmetric such that $\| \rho_0 \|_{L^1} = 1$.
	Then, there exists $M \in C_{loc} ([0,+\infty] \times (0,+\infty))$ a viscosity solution of \eqref{eq:mass} in $(0,\infty) \times (0,\infty)$ that satisfies the initial condition
		 $$M(0,v) = \int_{\widetilde B_v} \rho_0 (x) \diff x.$$
	We also have the $C^\alpha_{loc}$ interior regularity estimate \eqref{eq:local Calpha regularity} with $R_v = \infty$.
\end{proposition}
	
Notice that, at this point, we do not check that $M(t,0) = 0$, and hence concentration in finite time may, in principle, happen in $\Rd$. We also do not show, at this point, that $M(t,\infty) = 1$. There could, in principle, be loss of mass at infinity.
	
\begin{remark}[Conservation of total mass if $m \in (\frac{n-2}n, 1)$]
For this we use the following comparison. We consider $\underline u_k$ the solution of the pure-diffusion equation
\begin{equation*}
	\begin{dcases}
		 \underline u_t = \Delta \Phi_k (\underline u)  & t > 0 , x \in B_{R}, \\
		\partial_n \underline u = 0 & t> 0, x \in \partial B_{R} \\
		\underline u (0,x) = u_{0}(x).
	\end{dcases}
\end{equation*}
Then the associated mass satisfies the equation
\begin{equation*}
\begin{dcases}
	\frac{\partial \underline M}{\partial t} = (n \omega_n^{\frac 1 n} v^{\frac{n-1}{n}})^2 \frac{\partial  }{\partial v} \Phi_k \left(\frac{\partial \underline M}{\partial v} \right) & t > 0, v \in (0,R_v),\\
	\underline M (t,0) = 0 , & t> 0\\
	\underline M (t,R_v) = \|u_{0}\|_{L^1(\Br)}  & t > 0.
\end{dcases}
\end{equation*}
If $u_0 \ge 0$ is radially decreasing, then so is $\frac{\partial M_k}{\partial v} = \underline u$. Therefore, in the viscosity sense
\begin{equation*}
	\frac{\partial \underline M}{\partial t} \le  (n \omega_n^{\frac 1 n} v^{\frac{n-1}{n}})^2 \left\{ \frac{\partial  }{\partial v} \Phi_k \left(\frac{\partial \underline M}{\partial v} \right) +  \frac{\partial \underline M}{\partial v}\frac{\partial V}{\partial v} \right\}.
\end{equation*}
Let $u$ be the solution of \eqref{eq:main regularised bounded}.
Due to \Cref{thm:comparison principle for mass} we have that
\begin{equation*}
	\underline M (t,v) \le \int_{\widetilde B_v} u(t,x) \diff x.
\end{equation*}
Recalling the limit through $\Phi_k$ given by \eqref{eq:Phik} and the limit $R \to \infty$, the mass constructed in \Cref{prop:Rd existence mass} we have the estimate
\begin{equation*}
	\int_{\widetilde B_v} \underline u (t,x) \diff x \le M(t,v) \le 1.
\end{equation*}
where $\underline u$ is the solution of
$
		\underline u_t = \Delta \underline u^m
$ in $\Rd$.
When $m \in (\frac{n-2}n, 1)$ we know that $\int_\Rd \underline u(t,x) \diff x = \int_\Rd u_0 (x) \diff x$ and, hence $M(t,\infty) = 1$.
\end{remark}	

\subsection{At least infinite-time concentration of the mass}
\label{sec:Rd concentration}

Let assume $a_V < 1$ and that $\rho_0$ is such that
there exists $F$ with the following properties
\begin{equation}
	\label{eq:Rd hypothesis concentration}
	\| \rho_F \|_{L^1 (\Rd)} = 1 \qquad \text{ and } \qquad
	M_{\rho_F} \le M_{\rho_0} \le (1 - a_V) + M_{\rho_V}.
\end{equation}

\begin{remark}
For example, this covers the class of initial data that satisfy the following three assumptions:
\begin{itemize}
	\item $M_{\rho_V} \le M_{\rho_0} \le (1 - a_V) + M_{\rho_V}$
	\item $\int_{\widetilde B_v}\rho_0 (x)\diff x = (1 - a_V) + \int_{\widetilde B_v} \rho_V (x) \diff x$ for $v \ge v_0$
	\item $M_{\rho_0}$ is Lipschitz in $(v_0 - \ee, v_0 + \ee)$.
\end{itemize}
In this setting, we can take a suitable initial datum $\rho_D$ as in the case of balls, and we are reduced to a problem in $[0,v_0]$, since the upper and lower bound guarantee that $M(t,v) = (1 - a_V) + \int_{\widetilde B_v} \rho_V (x) \diff x$ for all $v \ge v_0$. This is a Dirichlet boundary condition for the mass.
\end{remark}

When $\rho _0 = \rho_F$ then the associated mass $M$ obtained in \Cref{prop:Rd existence mass} satisfies
\begin{enumerate}
	\item $M$ is a viscosity solution of the mass equation and locally $C^\alpha$
	\item $M (0,v) = \int_{\widetilde B_v} \rho_F (x) \diff x$
	\item  $M $ is non-decreasing in $t$ and $x$ (due to the properties of the approximations).
	\item We have the comparison
	\begin{equation*}
		M_{\rho_V} (v) \le M_{\rho_D} (v) \le M (t,v) \le (1 - a_V) + M_{\rho_V} (v).
	\end{equation*}
	In particular $M (t, \infty) = 1$ for all $t$ finite.
\end{enumerate}
Again, there exists a point-wise limit
\begin{equation*}
	M_\infty(v) = \lim_{t \to \infty} M(t,v).
\end{equation*}
As in \Cref{thm:R finite concentration mass}, $M_\infty$ preserves the $C^\alpha_{loc}$ estimates, using Dini's theorem we can prove uniform convergence in intervals $[\ee, \ee^{-1}]$. Thus $M_\infty$ is a viscosity solution of \eqref{eq:problem M infty}.
Due to the sandwich theorem and monotonicity
\begin{equation*}
	M_\infty(0^+) \le 1 -a_V, \qquad M_\infty(+\infty) = 1.
\end{equation*}
It is easy to characterise $M_\infty$ as we have done in the case of balls.

This proves \Cref{cor:Rd concentration} under hypothesis \eqref{eq:Rd hypothesis concentration}.

\begin{remark}[Convergence of $\rho_R$ as $R \to \infty$]
Since we do not have any $L^q$ bound for $\rho$ for $q > 1$, we do not have any suitable compactness. We can extend $\rho_R(t)$ by $0$ outside $B_R$ and we do know that
$
	\| \widetilde \rho_R(t) \|_{\mathcal M (\Rd)} \le 1.
$
If we assume that \eqref{eq:sufficient condition existence of minimisers} and that $V(x) \ge c |x|^\alpha$ for $c, \alpha > 0$. The properties can be inhereted to $\rho_R$ so
\begin{equation*}
	\int_{B_R} |x|^\alpha \rho_R  \le C (1 + \mathcal F [\rho_0]).
\end{equation*}
For $\rho_0$ in a suitable integrability class, we have tightness, and hence
a weakly convergent subsequence such that
\begin{equation*}
	\widetilde \rho_R \rightharpoonup \mu \qquad \text{ weak}-\star \text{ in } L^\infty( 0,\infty; \mathcal M(\Rd) )
\end{equation*}
We also know that $\rho_R^m$ is uniformly integrable. However, since we cannot assure $\rho_R^m \rightharpoonup (\mu_{ac})^m$, we cannot characterise $\mu$ as a solution of \eqref{eq:main}. This remark is still valid for radial initial data.
\end{remark}

\subsection{Minimisation of the free energy}
Following the arguments in \cite{Balague2013,Calvez2017,Carrillo2019}, we have an existence and characterisation result for the minimiser.
In
$\Rd$ the free-energy of the FDE $u_t = \Delta u^m$ with $0 < m <1$, is not bounded below, and $u(t) \to 0$ as $t \to \infty$. In fact, the mass of solutions escapes through $\infty$ in finite time if $m < \frac{n-2}n$. We need to ask further assumptions on $V$ so that the formal critical points $\rho_{V+h}$ are in fact minimisers.

We show below that it suffices that $V$ is not critical in the sense of constants, i.e.
	\begin{equation}
		\label{eq:sufficient condition existence of minimisers}
		\inf_{\rho \in \mathcal P_{ac} (\mathbb R^n)} \left(  \frac 1 {m-1} \int_\Rd \rho^m  + (1-\varepsilon) \int_{\mathbb R^n} V (x) \rho(x) \diff x \right)  > -\infty, \qquad \text{for some } \ee > 0.
	\end{equation}
We provide an example of $V$ where this property holds below.
As in $\Br$, we define an extension of $\mathcal F$ to the space of measure as
\begin{equation*}
	\widetilde{ \mathcal F} [ \mu  ] = \mathcal E_m [  \mu_{ac}  ] + \int_{\mathbb R^d} V(x) \diff  \mu  (x)
\end{equation*}
where $ \mu_{ac} $ is the absolutely continuous part of the measure $ \mu $. Notice that, since we choose $V(0) = 0$, we have that
$
	\widetilde{ \mathcal F} [ \mathfrak m \delta_0 + \rho] = \mathcal F [\rho].
$

\begin{proposition}
	\label{lem:minimizer}
	Assume $V \ge 0$ and $V(0) = 0$ and \eqref {eq:sufficient condition existence of minimisers}.
	Then, we have the following:
	\begin{enumerate}
		\item
		\label{it:minimizer 1}
		There exists a constant $C>0$ such that
		\begin{equation*}
			\int_{\mathbb R^n} \rho^m + \int_{\mathbb R^d} V \rho \le C ( 1+ \mathcal F[ \rho] ).
		\end{equation*}
	\end{enumerate}
	If, furthermore $V$ is radially symmetric and non-decreasing then
	\begin{enumerate}[resume]
		\item
		\label{it:minimizer 2}
		There exists $\mu_\infty \in \mathcal P (\mathbb R^n)$ such that
		\begin{equation}
		\label{eq:mu infinity minimises}
			\widetilde {\mathcal  F}[\mu_\infty] =
 		\inf_{\mu \in \mathcal P (\Rd) } \widetilde{ \mathcal F}[\mu]  =
 		\inf_{\rho \in \mathcal P(\mathbb R^n) \cap L^1 (\Rd)} \mathcal F [\rho] .
		\end{equation}

		\item
		\label{it:minimizer 3}
		We have that
		\begin{equation*}
			\mu_\infty = \begin{dcases}
 				\rho_{V+h} & \text{if } a_{V+h} = 1, \\
 				(1-a_V) \delta_0 + \rho_V &  \text{if } a_{V} < 1,
 				\end{dcases}
		\end{equation*}
	\end{enumerate}
\end{proposition}

\begin{proof}[Proof of \Cref{lem:minimizer}]
	Due to the lower bound, we have that
	\begin{equation*}
		\tfrac{1}{1-m} \int_{\mathbb R^n} \rho^m \le C + (1- \varepsilon) \int_{\mathbb R^n} V \rho .
	\end{equation*}
	On the other hand, we get
	\begin{equation*}
		\int_{\mathbb R^n} V \rho = \mathcal F[\rho] + \tfrac{1}{1-m} \int_{\mathbb R^n} \rho^m \le F[\rho] + C + (1- \varepsilon) \int_{\mathbb R^n} V \rho
	\end{equation*}
	Thus
	\begin{equation*}
		\ee \int_{\mathbb R^n} V \rho \le   \mathcal F[\rho] + C
	\end{equation*}
	Finally, we recover
	\begin{equation*}
			\tfrac{1}{1-m} \int_{\mathbb R^n} \rho^m \le C + \frac{(1- \varepsilon)}\ee (\mathcal F[\rho] + C) .
	\end{equation*}
	This completes the proof of \Cref{it:minimizer 1}.
	
	Clearly, we have that
	\[
		\mathcal F[\mu] \ge  \mathcal E_m [\rho] + (1-\varepsilon) \int_{\mathbb R^n} V (x) \rho(x) \diff x.
	\]
	Hence, the infimum of $\mathcal F$ is finite.
	As in the proof of \Cref{thm:bounded minimising FR}, we can consider a minimising sequence $\rho_j$. As in \Cref{thm:bounded minimising FR} we may assume that $\rho_j$ are radially symmetric and non-increasing.
	
		Let us prove \Cref{it:minimizer 2}.
		As in \Cref{thm:bounded minimising FR}, the second equality of \eqref{eq:mu infinity minimises} is due to the weak-$\star$ density of $L^1(\Rd)$ in the set of measures and the construction of $\widetilde{\mathcal F}$.
		For our minimising sequence we know hence that
	\begin{equation*}
		\int_{\mathbb R^n} \rho_j = 1 , \qquad \int_{\mathbb R^n} \rho_j^m \le C ( 1 + \mathcal F[\rho_j]) \le  C.
	\end{equation*}
	Using Lieb's trick in \Cref{rem:Liebs trick}, we obtain that
	$
		\rho_j \le C \min \{  |x|^{-n} , |x|^{-n/m}  \}.
	$	Integrating outside of any ball $B_R$, we can estimate
	\begin{equation*}
		\int_{\mathbb R^n \setminus B_R} \rho_j \le C \left(  \int_R^\infty r^{-\frac n m + n - 1} \diff r  \right)  \le  C R^{n \left(  1 - \frac 1 m \right)}.
	\end{equation*}
	Since $m < 1$, this is a tight sequence of measures. By Prokhorov's theorem, there exists a weakly-$\star$ convergent subsequence in the sense of measures. Let its limit be $\mu_\infty$.
	
	For the proof of \Cref{it:minimizer 3},  we proceed as in \Cref{thm:bounded minimising FR}.
	Notice that we still have the estimate
	\begin{equation*}
		\rho_j(x) \le \frac{\int_{\Rd} V \rho }{\int_{B_{|x|}} V}.
	\end{equation*}
	Since $V$ is strictly increasing, this is an $L^\infty (\Rd\setminus B_\kappa)$ of any $\kappa > 0$,
	and we can repeat the argument in $\Br$.
	\end{proof}

Let us illustrate the previous theorem by giving sufficient conditions on $V$ satisfying the main assumption of \Cref{lem:minimizer}. We extend the argument in \cite{Carrillo+Delgadino2018} to show a family of potentials $V$ for which \eqref{eq:sufficient condition existence of minimisers} holds.
\begin{theorem}
	\label{thm:Rd V example}
	Assume that, for some $\alpha \in (0,m)$ we have that
	\begin{equation}
	\label{eq:Rd V sufficient concentration}
		\chi_V = \sum_{j=1}^\infty 2^{jn} V(2^{j})^{-\frac \alpha {1-m}} < \infty.
	\end{equation}
	Then, \eqref{eq:sufficient condition existence of minimisers} holds for any $\ee \in (0,1)$.
\end{theorem}

\begin{remark}
	If the function $r \mapsto V(r)^{-\frac \alpha {1-m}} r^n$ is non-increasing, then the integral criterion for series and the change of variable show that the condition becomes
	\begin{equation*}
		\int_1^\infty 2^{ny} V(2^y)^{-\frac \alpha {1-m}} \diff y =
		 \frac{1}{\ln 2}
		\int_2^\infty V(r)^{-\frac \alpha {1-m}} r^{n-1} \diff r \sim  \int_{|x| \ge 2} \rho_V^{\alpha} \diff x < \infty.
	\end{equation*}
	We are requesting that $\rho_V^{m-\delta} \in L^1$ for some $\delta \in (0,m)$. This is only slightly more restrictive than simply that $\rho_V$ gives a finite quantity in either term of $\mathcal F$.
\end{remark}

\begin{proof}[Proof of \Cref{thm:Rd V example}]
	We look first at the integral on $B_1$. Due to Hölder's inequality, we have that
	\begin{equation*}
		\frac{1}{m-1}\int_{B_1} \rho^m \ge \frac{|B_1|^{1-m}}{m-1} \left( \int_{B_1} \rho \right)^{m} .
	\end{equation*}
	On the other hand, since $V, \rho \ge 0$ we know that
	$
		\int_{B_1} V \rho \diff x \ge 0.
	$
	Hence, we only need to care about the integration on $\mathbb R^n \setminus B_1$.
	We define, for $j \ge 1$
	\begin{equation*}
		\rho_j = \int_{B_{2^j} \setminus B_{2^{j-1}}} \rho (x) \diff x.
	\end{equation*}
	First, we point out that
	\begin{equation*}
		\int_{\Rd \setminus B_1} V(x) \rho(x) \diff x \ge \sum_{j=1}^\infty V(2^{j-1}) \rho_j.
	\end{equation*}
	Due to Jensen's inequality
	\begin{align*}
		\int_{B_{2^j} \setminus B_{2^{j-1}}} \rho^m &\le |B_{2^j} \setminus B_{2^{j-1}}| \left( \frac{1}{|B_{2^j} \setminus B_{2^{j-1}}|} \int_{B_{2^j} \setminus B_{2^{j-1}}} \rho(x) \diff x \right)^m = |B_{2^j} \setminus B_{2^{j-1}}|^{1-m} \rho_j \\
		&= C_n 2^{jn(1-m)} \rho_j^m .
	\end{align*}
	Notice that $B_{2^j} \setminus B_{2^{j-1}} = 2^j (B_{1} \setminus B_{\frac 1 2})$. Hence
	\begin{equation*}
		\int_{\Rd \setminus B_1}  {\rho^m}  \le \sum_{j=1}^\infty \frac{C_n 2^{j n (1-m)}}{V(2^{j-1})^\alpha}V(2^{j-1})^\alpha \rho_j^\alpha \rho_j^{m-\alpha}.
	\end{equation*}
	Applying the triple Hölder inequality with exponents $p = (1-m)^{-1}, q = \alpha^{-1}, r = (m - \alpha)^{-1} $ we recover
	\begin{align*}
		\int_{\Rd \setminus B_1}  {\rho^m}  &\le \left(\sum_{j=1}^\infty \frac{C_n 2^{j n}}{V(2^{j-1})^\frac{\alpha}{1-m}}\right)^{1-m} \left( \sum_{j=1}^\infty V(2^{j-1}) \rho_j \right)^\alpha \left( \sum_{j=1}^\infty \rho_j \right)^{m-\alpha} \\
		&\le \chi_V^{1-m} \| \rho \|_{L^1}^{m-\alpha} \left( \int_{\mathbb R^n \setminus B_1} V (x) \rho(x) \diff x \right)^{\alpha}
	\end{align*}
	Lastly, using Young's inequality we have, for any $\ee > 0$
	\begin{equation*}
		\int_{\Rd \setminus B_1}  {\rho^m} \le \ee (1-m) \int_{\mathbb R^n \setminus B_1} V (x) \rho(x) \diff x + C({\ee, \alpha,m)}\chi_V^
		{
			\frac{1-m}{1-\alpha}
		} \| \rho \|_{L^1}^
		{
			\frac{m-\alpha}{1-\alpha}
		}
	\end{equation*}
	Therefore
	\begin{equation*}
		\frac{1}{m-1} \int_{\Rd} \rho^m + (1-\ee) \int_\Rd V \rho \ge \frac{|B_1|^{1-m}}{m-1} \left( \int_{B_1} \rho \right)^{m} - \frac{C({\ee, \alpha,m)}}{1-m}\chi_V^
		{
			\frac{1-m}{1-\alpha}
		} \| \rho \|_{L^1}^
		{
			\frac{m-\alpha}{1-\alpha}
		} .
	\end{equation*}
	This completes the proof.
\end{proof}

\begin{remark}[The power-type case $V(x) = C |x|^\lambda$ for $|x| \ge R_0$]
\label{rem:minimisation for powers at infinity}
	In this setting, \eqref{eq:Rd V sufficient concentration} becomes $m > \frac{n}{n+\lambda}$ (equivalently $\frac{n(1-m)}{m} < \lambda$), and in this case can
	take any $\alpha$ such that
	 $ \frac{n(1-m)}{\lambda} < \alpha
	<m$.
	This condition is sharp.
	Let us see that, otherwise, $\mathcal F$ is not bounded below.
	We recall the following computation, which can be found in \cite[Theorem 15]{Carrillo+Delgadino2018} following the reasoning in  \cite[Theorem 4.3]{Carrillo+Delgadino+Patacchini2019}.
	
	Assume $m < \frac{n}{n+\lambda}$. We can construct densities $\rho$ where the energy attains $-\infty$. Let
	\begin{equation*}
		\rho_\beta = \sum_{j=j_0}^\infty \frac{\rho_j}{|B_{2^{j+1}} \setminus B_{2^j}|} \chi_{B_{2^{j+1}} \setminus B_{2^j}}, \qquad \text{where } \rho_j = \frac{2^{-j\beta}}{\sum_{j=j_0}^\infty 2^{-j\beta}}.
	\end{equation*}
	where $\beta > 0$ is a constant we will choose later, and $j_0$ is such that $2^{j_0} > R_0$.
	We can explicitly compute
	\begin{equation*}
		\int_\Rd |x|^\lambda \rho_\beta(x) \diff x = \frac{2^{n+\lambda}-1}{n + \lambda} \frac{\sum_{j=j_0}^\infty 2^{-j(\beta-\lambda)}}{\sum_{j=j_0}^\infty 2^{-j\beta}}
	\end{equation*}
	This is a finite number whenever $\beta > \lambda$.
	On the other hand
	\begin{equation*}
		\int_\Rd \rho_\beta(x)^m \diff x= C(n, \lambda) \frac{\sum_{j=j_0}^\infty 2^{-j(m \beta - n(1-m))}}{\left(\sum_{j=j_0}^\infty 2^{-j\beta}\right)^m}
	\end{equation*}
	This number is infinite if $m \beta < n(1-m)$. Hence,
	\begin{equation*}
		-\tfrac{1}{1-m}\int_\Rd \rho_\beta(x)^m \diff x +   \int_\Rd C |x|^\lambda \rho_\beta (x) \diff x = -\infty , \qquad \forall  C \in \mathbb R \text{ and }\lambda < \beta < \frac{n(1-m)}{m}.
	\end{equation*}
	The case of the equality $m = \frac{n}{n+\lambda}$ is, as usual, more delicate due to the scaling. However, we still prove that
	\begin{equation*}
		\inf_{\rho \in \mathcal P \cap L^1} \left(- \int_\Rd \rho^{\frac{n}{n+\lambda}} + C \int_\Rd |x|^\lambda \rho \right)= -\infty, \qquad \forall C \in \mathbb R.
	\end{equation*}
	As in the proof of \cite[Proposition 4]{Carrillo2019}, we can take the following functions:
	\begin{equation*}
		\rho_k (x) = D_k |x|^{-(n+\lambda)} \chi_{B_k \setminus B_{R_0}}, \qquad \text{ where } D_k= \left(\int_{B_k \setminus B_{R_0}} |x|^{-(n+\lambda)} \diff x \right)^{-1}.
	\end{equation*}
	It is a direct computation that
	\begin{equation*}
		\tfrac 1 {D_k}\int_\Rd |x|^\lambda \rho_k = \int_{B_k \setminus B_{R_0}} |x|^{-n}\diff x = \tfrac 1 {D_k^{\frac{n}{n+\lambda}}}\int_\Rd \rho_k^{ \frac{n}{n+\lambda} }.
	\end{equation*}
	For any $\alpha_k > 0$, we have that the rescaling
	$
		\widetilde \rho_{k} (x) = \alpha_k^n \rho_k(\alpha_k x)
	$
	is such that
	\begin{equation*}
		a_k \defeq \frac{ \int_\Rd |x|^\lambda \widetilde\rho_k}
			{
				\left( \int_\Rd \widetilde\rho_k^{ \frac{n}{n+\lambda} } \right)^{\frac{n+\lambda}{n}}
			}
			=
			\frac{ \int_\Rd |x|^\lambda \rho_k}
			{
				\left( \int_\Rd \rho_k^{ \frac{n}{n+\lambda} } \right)^{\frac{n+\lambda}{n}}
			}
			= \left( \int_{B_k \setminus B_{R_0}} |x|^{-n}\diff x \right)^{1-\frac{n+\lambda}{n}} = \left(|\partial B_1| \log \frac{ k}{R_0}\right)^{1-\frac{n+\lambda}{n}}
			\to 0.
	\end{equation*}
	For any sequence $b_k$ which is yet to be determined, we can pick $\alpha_k$ so that
	$		
	\int_{\Rd} \widetilde\rho_k^{\frac{n}{n+\lambda}} = b_k$ by taking
	$$
	\alpha_k = \left(  \frac{ b_k }{ \int_\Rd \rho_k^{\frac{n}{n+\lambda}}} \right) ^{-\frac {n + \lambda} {\lambda n}}.
	$$ 	
	Then, passing to the notation $m = \frac{n}{n+\lambda}$, we recover that
	\begin{align*}
		-\int_\Rd \widetilde\rho_k^{m} + C \int_\Rd |x|^\lambda \widetilde\rho_k
		&= -b_k + C a_k b_k^{\frac 1 m} = - b_k^{\frac 1 m} \left(  b_k^{-\frac{1-m}{m}} - Ca_k \right) = - b_k^{\frac 1 m - \ee} ,
	\end{align*}
	if pick the sequence $b_k$ so that $b_k^{-\frac{1-m}{m}} - Ca_k =  b_k^{-\ee}$.
	Notice that the function $g_{a,b} (s) = s^a - s^b$ is strictly increasing near $0$ if $a < b$.
	Hence, for $k$ large enough and $\ee > \frac{1-m} m$, we can solve $Ca_k = b_k^{-\frac{1-m}{m}} - b_k^{-\ee}$, and we recover $b_k \to +\infty$ as $k \to \infty$. Hence, taking $\ee \in ( \frac{1-m} m, \frac 1 m)$, and $k \to \infty$, we prove the result.
\end{remark}

\begin{remark}
	With the sequence $\rho_k$ above, we can also prove that
	\begin{equation}
	\label{eq:Carlson critical}
		\inf_{\rho \in \mathcal P \cap L^1} \frac{ \int_\Rd |x|^\lambda \rho}{ \left( \int_\Rd \rho^{ \frac{n}{n+\lambda} } \right)^{\frac{n+\lambda}{n}} } = 0.
	\end{equation}
	This corresponds to the borderline case of the Carlson type inequalities
	 	\begin{equation*}
		\left( \int_\Rd \rho  \right)^{ 1 - \frac{n (1-m)}{\lambda m} } \left( \int_\Rd |x|^\lambda \rho \right)^{\frac{n (1-m)}{\lambda m}} \ge c_{n,\lambda,m} \left( \int_\Rd \rho^m \right)^{\frac 1 m} , \qquad \forall \tfrac{n}{n+\lambda} < m <1 \text{ and } \rho \ge 0.
	\end{equation*}
	which are known with the explicit constant (see, e.g., \cite[Lemma 5]{Carrillo2019}).
\end{remark}

\subsection{Infinite-time concentration if $V$ is quadratic at $0$}
\label{sec:pure aggregation}

Our aim in this section is to compare the solutions of \eqref{eq:main} with the solutions of the pure-aggregation problem
\begin{equation}
	\label{eq:aggregation}
	\frac{\partial  \rho }{\partial t} = \diver (\rho \nabla \widetilde V),
\end{equation}
where $\widetilde V$ is a different potential.
The  equation for the mass can be written in radial coordinates as
\begin{equation}
	\label{eq:aggregation mass}
	\frac{\partial  M}{\partial t} = \frac{\partial  M}{\partial r} \frac{\partial \widetilde V}{\partial r}.
\end{equation}
We will show that infinite-time aggregation happens for \eqref{eq:aggregation} if and only if
\begin{equation}
		\label{eq:hypothesis for concentration in infinite time}
		\int_{0^+} \left( \tfrac{\partial \widetilde V}{\partial r} (s) \right)^{-1} \diff s = +\infty.
	\end{equation}
Clearly, a sufficient condition that $\frac{\partial \widetilde V}{\partial r} \le C r$ near $0$. This is the so-called Osgood condition used to distinguish infinite from finite time blow-up in aggregation equations \cite{BCL}.

\begin{proposition}
\label{prop:solutions of pure aggregation}
	Assume $\widetilde V \in C^2 (\Rd)$, is radially symmetric, $\widetilde V(0) = 0$, $\frac{\partial \widetilde V}{\partial r} (r) > 0 $ for $r > 0$, \eqref{eq:hypothesis for concentration in infinite time} and let $  M_0$ be a continuous, non-decreasing and bounded function. Then
	\begin{enumerate}
		\item There exists a unique classical solution by characteristics $M(t,r)$ of \eqref{eq:aggregation mass} defined for all $t, r > 0$.
		
		\item We have $M(t,0) = 0$ for all $t > 0$, i.e. there is no concentration in finite time.
	\end{enumerate}
\end{proposition}

\begin{proof}[Proof of \Cref{prop:solutions of pure aggregation}]
	\Cref{eq:aggregation mass}  is a first order linear PDE that we can solve by characteristics.
We can look at the characteristic curves of constant mass
$
	\overline M(t, r_c(t, r_0)) = \overline M(0,r_0).
$
Taking a derivative we recover
$
	\frac{\diff r_c}{\diff t} (t) = - \frac{\partial \widetilde V}{\partial r} (r_c(t)).
$
These are the same characteristics obtained when applying the method directly to \eqref{eq:aggregation}.
Clearly $ r_c (t, r_0) \le r_0.$
Since $V \in C^2 (\Rd)$, these characteristics exists for some time $t(r_0) > 0$, and are unique up to that time.
Hence, let
\begin{equation}
	\label{eq:pure aggregation condition characteristics}
	t = \int_{r_c(t, r_0)}^{r_0} \left(  \frac{\partial \widetilde V}{\partial r} (s) \right)^{-1} \diff s.
\end{equation}
Concentration will occur if $r_c(t, r_0) = 0$ for some $r_0 > 0$ and $t < \infty$, which is incompatible with \eqref{eq:hypothesis for concentration in infinite time}.
Notice that since $0 < r_c(t,r_0) \le r_0$, these functions are defined for all $t > 0$.
Let us check that $r_c(t,r_0)$ do not cross, and hence can be used as characteristics. If two of them cross at time $t$, we have that
	\begin{equation*}
			 \int_{r_0 }^{r_c(t,r_0)} \left(  \frac{\partial \widetilde V}{\partial r} (s) \right)^{-1} \diff s =- t = \int_{r_1}^{r_c(t,r_1)} \left(  \frac{\partial \widetilde V}{\partial r} (s) \right)^{-1} \diff s .
	\end{equation*}
	Since $r_c(t,r_0) = r_c(t,r_1)$ then we get
	\begin{equation*}
		\int_{r_0}^{r_1} \left(  \frac{\partial \widetilde V}{\partial r} (s) \right)^{-1} \diff s = 0.
	\end{equation*}
	As $ \frac{\partial \widetilde V}{\partial r}  > 0$ outside $0$, then $r_0 = r_1$ and the characteristics are the same.
	Due to the regularity of $\widetilde V$, there is continuous dependence and, since the characteristics point inwards and do not cross, they fill the space $[0,+\infty) \times [ 0, +\infty)$.

Finally, notice also that  $\frac{\partial \widetilde V}{\partial r} (0) = 0 $ and positive otherwise, then for any $r_0 > 0$ we have that
$
	\lim_{t \to +\infty} r_c (t, r_0) = 0 .
$
	Since $\widetilde V$ is $C^2$, then we have $\partial \widetilde V / \partial r (0) = 0$ so $r_c(t,0) = 0$, i.e. $M(t,0) = 0$.
\end{proof}

\begin{proposition}
	\label{prop:monotone solutions of aggregation}
	Let $\rho$ be a solution by characteristics of the aggregation equation \eqref{eq:aggregation}, and let $r_0(t,r)$ the foot of the characteristic through $(t,r)$. Then
	\begin{equation}
		\label{eq:aggregation derivative of solution by characteristics}
		\frac{\partial \rho  }{\partial r} (t,r)  =  \left( \frac { r_0} r \right)^{n-1} \rho_0 (r_0)  \frac{\frac{\partial \widetilde V }{\partial r} (r_0) } {  \left(  \frac{\partial \widetilde V }{\partial r} (r) \right)^2  }    \Bigg ( - \Delta \widetilde V (r) +  \Delta \widetilde V (r_0) + \rho_0(r_0) ^{-1} \frac {\diff \rho_0 }{\diff r} (r_0)  {  \frac{\partial \widetilde V}{\partial r} (r_0) }    \Bigg ).
	\end{equation}
	In particular, if $\rho$ is a decreasing solution and $ \widetilde V \in C^2 (\mathbb R^n)$ with $\Delta \widetilde V (0) = 0$, then
	\begin{equation}
		\label{eq:aggregation derivative of solution by characteristics decreasing solution C2 potential}
		\Delta \widetilde V + \rho_0^{-1} \frac{\diff \rho_0}{\diff r} \frac{\partial \widetilde V}{\partial r}   \le 0 , \qquad \text{ in } \supp \rho_0.
	\end{equation}
\end{proposition}

\begin{remark}
	For $\widetilde V(r)= r^2$ then $\Delta \widetilde V$ is constant, and we only have the last term, so all solutions with decreasing initial datum are decreasing.
	If $\Delta \widetilde V$ is non-increasing, then in \eqref{eq:aggregation derivative of solution by characteristics} we have $-\Delta \widetilde V(r) + \Delta \widetilde V(r_0) \le 0$ and all solutions are decreasing. This is the case for $\widetilde V(r) = \gamma r^\gamma$ with $\lambda \in (0,2]$.
	When $\widetilde V (r) = \gamma r^\lambda$ with $\lambda > 2$, let us show that decreasing solutions of \eqref{eq:aggregation} are not $L^1 (\mathbb R^n)$. Hence, any decreasing integrable initial data produces a solution that losses monotonicity. Indeed, if $\widetilde V(r) = r^\lambda$ then $\Delta \widetilde V= (n + \lambda - 2) r^{\lambda - 2}$ and integrating in \eqref{eq:aggregation derivative of solution by characteristics decreasing solution C2 potential} we recover $\rho_0 \ge Cr ^{-(n + \lambda - 2)}$ which is not integrable for $\lambda > 2$.
\end{remark}

\begin{proof}[Proof of \Cref{prop:monotone solutions of aggregation}]
	Taking the derivative directly on $M(t,r) = n \omega_n r^{n-1} \frac{\partial \rho}{\partial r}$, we recover that
	\begin{align*}
		\allowdisplaybreaks
		\frac{\partial \rho  }{\partial r} (t,r)   &= (n \omega_n)^{-1} \frac{\partial}{\partial r} \left( r^{1-n} \frac{\partial }{\partial r}  (M (t,r))\right)
		= (n \omega_n)^{-1} \frac{\partial}{\partial r} \left( r^{1-n} \frac{\partial }{\partial r}  (M_0 (r_0(t,r) ))\right)   \\
		&= r^{1-n} r_0^{n-1} \frac{\partial r_0}{\partial r} (t,r) \rho_0 (r_0) \Bigg ( - (n-1) r^{-1}  + (n-1) r_0^{-1}  \frac{\partial r_0}{\partial r} + \rho_0(r_0) ^{-1} \frac {\diff \rho_0 }{\diff r} (r_0) \frac{\partial r_0}{\partial r} +  \frac{ \frac{\partial^2 r_0}{\partial r^2} }{ \frac {\partial r_0}{\partial r}} \Bigg ).
	\end{align*}
	Going back to \eqref{eq:pure aggregation condition characteristics} and taking a derivative in $r$, we deduce
	\begin{equation*}
		\frac{\partial r_0}{\partial r} (t,r) = \frac{  \frac{\partial \widetilde V}{\partial r} (r_0 (t,r) ) }{  \frac{\partial \widetilde V}{\partial r} (r) } \ge 0.
	\end{equation*}
	Taking another derivative we have that
		\begin{equation*}
		\frac{\partial^2 r_0}{\partial r^2} (t,r) = \frac{  \frac{\partial \widetilde V}{\partial r} (r_0 (t,r) ) }{  \frac{\partial \widetilde V}{\partial r} (r)^2  } \left(   \frac{\partial^2 \widetilde V}{\partial r^2} (r_0) - \frac{\partial^2 \widetilde V}{\partial r^2} (r) \right) .
	\end{equation*}
	Joining this information and collecting terms we recover \eqref{eq:aggregation derivative of solution by characteristics}.
	Clearly, \eqref{eq:aggregation derivative of solution by characteristics decreasing solution C2 potential} and the convexity of $\rho_0$ guarantee that $\rho (t, \cdot)$ is decreasing.
	Let us show that the condition holds in general. If $\rho$ is decreasing, then this value is not positive. For $r_0 \in \supp \rho_0$ we therefore have
	\begin{equation*}
		 - \Delta \widetilde V (r) +  \Delta \widetilde V (r_0) + \rho_0(r_0) ^{-1} \frac {\diff \rho_0 }{\diff r} (r_0)  {  \frac{\partial \widetilde V}{\partial r} (r_0 ) } \le 0.
	\end{equation*}
	The support of $\rho_0$ is a ball. Fixing a value a value of $r \in \supp \rho_0$ we have that
	\begin{equation*}
		 - \Delta \widetilde V (r_c(t,r)) +  \Delta \widetilde V (r) + \rho_0(r) ^{-1} \frac {\diff \rho_0 }{\diff r} (r)  {  \frac{\partial \widetilde V}{\partial r} (r) } \le 0.
	\end{equation*}
	Letting $t \to +\infty$, since $r_c (t,r) \to 0$, $\Delta \widetilde V$ is continuous and $\Delta \widetilde V (0) = 0$, we recover \eqref{eq:aggregation derivative of solution by characteristics decreasing solution C2 potential}.
	This completes the proof.
\end{proof}

Now we have the tools to show that concentration does not happen in finite time if $\frac{\partial V}{\partial r} \le C_v r$ close to $0$. We construct a super-solution using the pure-aggregation equation.

\begin{proof}[Proof of \Cref{prop:Rd no concentration in finite time}]
Take
\begin{equation*}
	\overline \rho_0 (x) =  { \rho_0 (x) }  \frac{\int_{\mathbb R^n} \rho_0 (x) \diff x} {\int_{B_{R_V}} \rho_0 (x) \diff x} \chi_{B_{R_V}}
\end{equation*}
and
\begin{equation}
	\label{eq:pure aggregation replacement V}
	\widetilde V (r) =
		\frac{ C_V } 2 r^2.
\end{equation}
Obtain $\overline M$ as the solution by characteristics of \eqref{eq:aggregation mass} constructed in \Cref{prop:solutions of pure aggregation}.
Due the definition of $\widetilde V $,  we know that it satisfies the hypothesis of \Cref{prop:monotone solutions of aggregation} and we have $\Delta \widetilde V  = n C_V \ge 0$. Thus,
	\eqref{eq:aggregation derivative of solution by characteristics} shows that $\overline \rho (t, \cdot)$ is decreasing, and non-negative. Therefore, it holds that, in the viscosity sense
	$
		\frac{\partial \overline M}{\partial v} \ge 0$ and  $\frac{\partial^2 \overline M}{\partial v^2} \ge 0.
	$
	Hence, still in the viscosity sense
	\begin{align*}
		\frac{\partial \overline M}{\partial t} -
		(n \omega_n^{\frac 1 n} v^{\frac{n-1} n })^2 \left\{ m \left( \frac{\partial \overline M}{\partial v}\right)^{m-1}   \frac{\partial^2 \overline M}{\partial v^2} +  \frac{\partial \overline M}{\partial v} \frac{\partial V}{\partial v} \right\}
		&\ge \frac{\partial \overline M}{\partial t} -
		(n \omega_n^{\frac 1 n} v^{\frac{n-1} n })^2 \left\{   \frac{\partial \overline M}{\partial v} \frac{\partial V}{\partial v} \right\} \\
		&=  \frac{\partial \overline M}{\partial t} - \frac{\partial \overline M}{\partial r} \frac{\partial V}{\partial r}
		= \frac{\partial \overline M}{\partial r} \left(C_V r  - \frac{\partial V}{\partial r} \right).
	\end{align*}
	Since characteristics retract, $\supp \frac{\partial \overline M}{\partial v} \subset B_{R_V}$ so the last term is non-negative by the assumption, because either $\frac{\partial \overline M}{\partial v} = 0$ or $C_V r  - \frac{\partial V}{\partial r} \le 0$. 	
	Thus, using the comparison principle in $\Br$ for $R \ge R_V$ given in \Cref{thm:comparison principle m} we have that
	$
		M_R \le \overline M $ for all $t \ge 0 , v \in [0,R_v]
	$
	Since $M$ is constructed by letting $R \to \infty$, we conclude $M \le \overline M$ for $t,v \ge 0$. 	
\end{proof}

\section{Final comments}	
\label{sec:final comments}
\begin{enumerate}
	\item Blow-up is usually associated in the literature to superlinear nonlinearities, both in reaction diffusion or in Hamilton-Jacobi equations, cf.  instance \cite{QuittSoupBook, GalVaz99} and its many references. Here  it is associated to sublinear diffusion,
	notice that \eqref{Vgrowth} implies, at least, $0<m<1$.
	This might seem surprising but it is not,
	due to two facts.
	First, recall that $0 < m <1$ means that the diffusion coefficient $m u^{m-1}$ is large when $u$ is small, and small when $u$ is large.
	This translates into fast diffusion of the support but slow diffusion of level sets with high values (see e.g. \cite{ChassVaz2002} for a thorough discussion%
	).
	This explains why $\delta_0$ may not be diffused for $m$ small (see \cite{Brezis1983}).
	Secondly,
	the confinement potential $V$ needs to be strong enough at the origin to compensate the diffusion and produce a concentration. In $\Br$, this is translated in the assumption $\int_{\Br} \rho_V < 1$ (recall that, for $V(x) = |x|^{\lambda_0}$, this implies $0 <m < \frac{n-\lambda_0}n < 1$). In $\Rd$ we need to deal with the behaviour at infinity, as mentioned in the introduction.

\item Formation of a concentrated singularity in finite time is a clear possibility in this kind of problem.
In this paper,
we do not consider the case $V \notin W^{2,\infty}_{loc} (\Rd)$ (e.g. $V(x) = |x|^\lambda$ with $\lambda < 2$).
	So long as $\frac{\partial V}{\partial v}$ is continuous (e.g. $\lambda \ge 1$), it makes sense to use the theory of viscosity solutions of the mass equation \eqref{eq:mass}.
	In principle, there could be concentration in finite time, even in \eqref{eq:main bounded domain}.
	Notice that, in our results, the estimate for $\rho(t) \in L^q(\Br)$ depends on $\| \Delta V \|_{L^\infty (\Br)}$. For more general $V$, better estimates for $\rho$ are needed in order to pass the limits $\Phi_k(s) \to s^m$ and $R \to \infty$.
	Some of these issues will be studied elsewhere.

	\item For $\rho_0 \in L^{1}_+(\Br)$, $S_R(t) \rho_0$ is constructed extending the semigroup through a density argument. We do not know whether it is the limit of the solutions $u_k$ of \eqref{eq:main regularised bounded} with \eqref{eq:Phik}. Furthermore, this question can be extended to initial data so that $\mathcal F_R[\rho_0] < \infty$.

	\item Non-radial data.
		We provide a well-posedness theory in $\Br$ when $\rho_0 > 0$, but not in $\Rd$.
		In $\Br$, as mentioned in \Cref{rem:bounded concentration non-radial}, we can show concentration in some non-radial cases,
		but the exact splitting of mass in the asymptotic distribution is still unknown.
		The asymptotic behaviour in the non-radial case is completely open.
\end{enumerate}

\appendix

\section{Recalling some classical regularity results}
\label{sec:classical regularity}
The equation for the mass of the solution of $u_t = \diver ( \nabla \Phi (u) + u \nabla V)$ is given by
\begin{align}
\tag{M$_\Phi$}
\label{eq:MPhi}
	\frac{\partial M}{\partial t} &= (n \omega^{\frac 1 n } v ^{\frac{n-1} n})^2   \left[ \frac{\partial }{\partial v} \Phi \left(  \frac{\partial M }{\partial v}  \right) + \frac{\partial V}{\partial v} \frac{\partial M }{\partial v} \right].
\end{align}
Let us prove local regularity of bounded solutions by applying the results in \cite{DiBenedetto1993}. To match the notation of \cite{DiBenedetto1993}, in this appendix we choose the notation $
x = v$,  $u = M$, and $a_0(x) = (n \omega^{\frac 1 n } v ^{\frac{n-1} n})^2
$.
We write the problem \eqref{eq:MPhi} as
\begin{equation}
\label{eq:DiBenedetto}
	u_t = \diver a (x,t,u,Du) + b(x,t,u,Du)
\end{equation}
where
\begin{equation*}
	a(x,Du) = a_0(x) \Phi(Du) , \qquad b(x,Du) = - D a_0(x) \Phi(Du) + a_0(x) DV \cdot Du.
\end{equation*}
The standard hypothesis set in \cite{DiBenedetto1993} are that for some $p > 1$ we have
\begin{align*}
\tag{A$_1$}
\label{eq:A1}
a(x,t,u,Du) \cdot Du &\ge C_0 |Du|^p - \varphi_0 (t,x), \\
\tag{A$_2$}
\label{eq:A2}
|a(x,t,u,Du) | &\le C_1 |Du|^{p-1} +\varphi_1 (t,x), \\
\tag{A$_3$}
\label{eq:A3}
	|b(x,Du)| &\le C_2 |Du|^p + \varphi_2 (t,x).
\end{align*}
We set $\Phi(s) = |s|^{m-1}s$ so $p = m+1 \in (1,2)$.
We aim to recover local estimates on a set $\Omega \Subset \Rd$.
We will be able to get local estimates outside $0$. Hypothesis \eqref{eq:A1} and \eqref{eq:A2} are easy to check with
$
	C_0 = \inf_{\Omega } a_0$, $
	C_1 = \sup_{\Omega } a_0 $ , and  $  \varphi_0 = \varphi_1 = 0.$
However, \eqref{eq:A3} is initially not trivial. Since $p \in (1,2)$, $|Du|^{p-1}$ is not controlled by $|Du|^p$ but
we have
\begin{equation*}
	|b(x, Du) | = ( \max_\Omega |Da_0| + \max_\Omega a_0 |DV|) ( 1 + |Du|^p ),
\end{equation*}
so we choose
$	C_2 = \max_\Omega |Da_0|  + \max_\Omega a_0 |DV|$
and
$ \varphi_2 =  \max_\Omega |Da_0|  + \max_\Omega a_0 |DV|.
$
Our functions $\varphi_i$ are bounded, so we also have the hypothesis
\begin{equation}
\tag{A$_4$}
\label{eq:A4}
	\varphi_0, \quad \varphi_1^{ \frac p {p-1} } , \quad \varphi_2 \in L^{\widehat q, \widehat r} ((0,T) \times \Omega)
\end{equation}
where
\begin{equation}
\tag{A$_5$}
\label{eq:A5}
\frac 1 {\widehat r} + \frac{N}{p \widehat q } = 1 - \kappa_1
\end{equation}
are trivially satisfied. A weak solution in DiBenedetto's notation requires the regularity
$	
u \in C_{loc} (0,T; L^2_{loc} (\Omega)).
$
The notion of sub-solution (resp. super-) of \eqref{eq:DiBenedetto} is, for every $K \Subset \Omega$ and $0 < t_1 < t_2 \le T$, we have
\begin{equation*}
	\int_K u \varphi \diff x \Bigg| _{t_1}^{t_2} + \int_{t_1}^{t_2} \int_{K} ( - u \varphi_t + a (x,Du) \cdot D \varphi) \diff x \diff \tau \le (\ge) \int_{t_1}^{t_2} \int_K b(x,Du) \varphi \diff x \diff \tau,
\end{equation*}
for test functions
$
	0 \le \varphi \in W_{loc}^{1,2} (0,T; L^2 (K)) \cap L^p_{loc} (0,T; W_0^{1,p} (K))$.
Let us denote
$	
	\Omega_T = (0,T) \times \Omega.
$
We have the following result
\begin{theorem}[\cite{DiBenedetto1993} Chapter III, Theorem 1.1]
\label{thm:regularity dibenedetto}
Let $p > 1$, assume \eqref{eq:A1}, \eqref{eq:A2}, \eqref{eq:A3}, \eqref{eq:A4} and \eqref{eq:A5} and let $u$ be a local weak solution of \eqref{eq:DiBenedetto}. Then, there exists constants $\gamma > 1$ and $\alpha \in (0,1)$ depending only on the constant of \eqref{eq:A1}-\eqref{eq:A5},
	$
		\| u \|_{L^\infty ( \Omega_T ) },$ and  $\| \varphi_0, \varphi_1^{\frac{p-1} p} , \varphi_2 \|_{L^{\widehat q, \widehat r} (\Omega_T)}
	$
	such that for all $K \Subset (0,T) \times \Omega$
	\begin{equation*}
			|u(t_1,x_1) - u(t_2,x_2) | \le \gamma \| u \|_{L^\infty(\Omega_T)} \left( \frac{|x_1-x_2| + \| u \|_{L^\infty(\Omega_T)} ^{\frac {p-2} p } |t_1 - t_2|^{\frac 1 p} }{\mathrm{dist} (K, \Gamma, p)} \right)^\alpha
	\end{equation*}
where, for $\Gamma = \{ (y,s) : s = 0 \text{ or } y \in \partial \Omega \}$ we have
\begin{equation*}
	\mathrm{dist} (K, \Gamma; p) = \inf_{ \substack { (x,t) \in K \\ (y,s) \in \Gamma} }\left ( |x-y| + \| u \|_{L^\infty (\Omega_T)} ^{\frac {p-2} p } |t - s|^{\frac 1 p} \right ).
\end{equation*}
\end{theorem}

\paragraph{Regularity at $t = 0$.} For the regularity at $t = 0$, if $\rho_0$ is only integrable, then $M_{\rho_0}$ is continuous. In order to construct a modulus of continuity, we introduce the essential oscillation on a set $K$ defined as
$
	\essosc_K u = \esssup_K u - \essinf_K u.
$
Notice that the previous result in the whole space stated that, for any $K \Subset (0,\infty) \times \Omega $ we have
\begin{equation*}
	\omega_{i,u} (K,h) = \essosc_{ \substack{ |t-s| \le h^p \\ |x-y| \le h \\ (t,x), (s,y) \in K} } u \le \gamma \| u \|_{L^\infty(\omega_T)} \left( \frac{1 + \| u \|_{L^\infty(\Omega_T)} ^{\frac {p-2} p }   }{\mathrm{dist} (K, \Gamma, p)} \right)^\alpha h^\alpha \to 0,
\end{equation*}
as $h \to 0$.
This in an interior modulus of continuity (with scaling). A similar estimate on the essential oscillations holds near $t = 0$, but the modulus of continuity now depends on the one from $u_0$.

\begin{theorem}[\cite{DiBenedetto1993} Chapter III, Proposition 11.1]
Fix $x_0 \in \Omega$ and $T_0 > 0$ and $R_0 > 0$ so that $B_{2R_0} (x_0) \subset \Omega$.
Then, for $k \ge 1$, there exist sequences $R_k, T_k \searrow 0$ and $\delta_k \to 0$ depending only on the constant of \eqref{eq:A1}-\eqref{eq:A5}, $R_0$ and $\|u\|_{L^\infty([0,T] \times B_{2R_0}(x_0))}$  such that
\begin{equation}
	\essosc_{ [0,T_k] \times B_{R_k / 2} (x_0) } u \le \max \left\{ \delta_k; C  \essosc_{B_{R_k} (x_0) } u_0 \right\}.
\end{equation}
\end{theorem}
As a consequence of the previous theorem we conclude that
\begin{equation*}
	\omega_{b,u} (x_0,h) = \esssup_{0 \le t \le h} |u(t,x_0)-u(0,x_0)| \le \essosc_{[0,h] \times B_{h}(x_0)} u ,
\end{equation*}
tends to $0$ as $h \to 0$.
This modulus of continuity depends only on the constants of \eqref{eq:A1}-\eqref{eq:A5} and
$	
	\omega_{u_0} (x_0,h) = \essosc_{|x-x_0| \le h} u_0.
$
For any $K$ compact, there exists
$\omega_{u_0} (K,h)$ such that
$	
	\omega_{u_0} (x_0,h) \le \omega_{u_0} (K,h),$ for all $ x_0 \in K,
$
also going to $0$ as $h \to 0$.

\begin{corollary}
	\label{cor:uniform regularity interior in x up to 0 in t}
	Let $K \Subset \Omega$, and $u_0 \in C(K)$. Then, for any $T > 0$ and $\ee >0$, there exists $\delta > 0$, depending only on $T, \ee$, the constants of \eqref{eq:A1}-\eqref{eq:A5} and $\omega_{u_0} (K, \cdot)$, such that if $(t,x), (s,y) \in [0,T] \times K  $, $|t-s| \le \delta^p$ and $|x-y| \le \delta$ then
	\begin{equation*}
		|u(t,x)- u(s,y)| \le \ee.
	\end{equation*}
\end{corollary}

\begin{proof}
	First, we point out that there exists $\omega_{b,u} (K,h)$ depending only on \eqref{eq:A1}-\eqref{eq:A5} and $\omega_{u_0} (K,\cdot)$ such that
$	\omega_{b,u} (x_0,h) \le \omega_{b,u} (K,h), $ for all $ x_0 \in K$.
Fix $T> 0$ and $\ee > 0$. Since we want to use the interior and boundary regularity, we first fix $\delta_t > 0$ such that
\begin{equation*}
	0 \le t \le \delta_t^p \implies \esssup_{x \in K} |u(t,x) - u(0,x)| \le \frac \ee 3.
\end{equation*}
Due to the uniform continuity of $u_0$, there exists $\delta_x > 0$ such that
$	
	\omega_{u_0} (K,\delta_x) \le \frac \ee 3.
$
Lastly, we take $h_i > 0$ such that
$	
	\omega_{i,u} ([\delta_t, T] \times K, h ) \le \frac \ee 3.
$
We then take $\delta = \min \{ \delta_t, \delta_x, h_i \}$. Let us now check the condition. We distinguish cases:
If $t, s > \delta_t$ then $|u(t,x) - u(s,y)| \le \frac \ee 3 < \ee$.
If $t < \delta_t  \le s $ (or viceversa), then we write
		\begin{equation*}
			|u(t,x) - u(s,y)| \le |u(t, x) - u(0, x) | + |u(0,x) - u(\delta_t,x)| + |u(\delta_t, x) - u(s, y)| \le \ee.
		\end{equation*}
Finally, if $t,s < \delta_t $, then we write
	\begin{equation*}
		|u(t,x) - u(s,y)| \le |u(t,x) - u(0,x)| + |u(0,x) - u(0,y)| + |u(0,y) - u(s,y)| \le \ee.
	\end{equation*}
This completes the proof.
\end{proof}

\section{Relating space and time regularities}
\label{sec:appendix b}
\begin{theorem}
	\label{thm:relation between space and time regularities}
	Let $I \subset \mathbb R$ and $u \in L^\infty (0,T; C^\alpha (\overline I)) \cap C^\beta (0,T; L^1 (\overline I))$. Then
	\begin{equation*}
		|u(t,x) - u(s,y)| \le C ( |x-y|^\alpha + |t-s|^{\frac{\alpha\beta}{\alpha + 1}})
	\end{equation*}
	where $C$ depends only on the norms of $u$ in the spaces above.
\end{theorem}

\begin{proof}
	We the following splitting
	$
		|u(t,x) - u(s,y)| \le |u(t,x) - u(s,y)| + |u(t,y) - u(s,y)|.
	$
	The bound for the first term is evident and yields $C|x-y|^\alpha$. For the second term we write, for some $h > 0$
	\begin{align*}
		|u( t,y) - u(s,y)| &\le \left| \frac{1}{2h} \int_{y-h}^{y+h} (u(t,z) - u(s,z) ) \diff s \right| + \left|\frac{1}{2h} \int_{y-h}^{y+h} (u(t,z) - u(t,y) ) \diff s \right|
 			\\
 			&\qquad + \left|\frac{1}{2h} \int_{y-h}^{y+h} (u(s,z) - u(s,y) ) \diff s \right|  \\
 			&\le \frac{1}{2h} \| u(t) - u(s)\| _{L^1} + C \int_{y-h}^{y+h} |z-y|^\alpha \diff s
 			\le  C \left (\frac{|t-s|^\beta}h +  h^{\alpha} \right).
	\end{align*}
	By choosing $h = |t-s|^\gamma$, the optimal rate is achieved when $\beta - \gamma = \alpha$, i.e. $\gamma = \frac{\beta}{\alpha + 1}$. This choice yields
	\begin{equation*}
		|u( t,y) - u(s,y)| \le C |t-s|^{\frac{\alpha \beta}{\alpha + 1}}. \qedhere
	\end{equation*}
\end{proof}

\section* {Acknowledgments}
The authors 
are thankful 
to the anonymous referee for the detailed reading of the manuscript, and their insightful suggestions.
The research of JAC and DGC was supported by the Advanced Grant Nonlocal-CPD (Nonlocal PDEs for Complex Particle Dynamics:
Phase Transitions, Patterns and Synchronization) of the European Research Council Executive Agency (ERC) under the European Union’s Horizon 2020 research and innovation programme (grant agreement No. 883363).
JAC was partially supported by EPSRC grant number EP/T022132/1.
The research of JLV was partially supported by grant PGC2018-098440-B-I00 from the Ministerio de Ciencia, Innovación y Universidades of the Spanish Government. JLV was an Honorary Professor at Univ.\ Complutense.

\printbibliography

\end{document}